\documentclass[reqno]{amsart}
\usepackage[a4paper]{geometry}

\usepackage[T2A]{fontenc}
\usepackage[utf8]{inputenc}
\usepackage[ukrainian,russian,english]{babel}
\usepackage{comment}

\usepackage{amssymb}
\usepackage[mathscr]{eucal}
\usepackage{tikz}
\usepackage{mathtools}
\usepackage{enumerate}
\usepackage{accents}
\usepackage{dsfont}
\usepackage{color}
\usepackage[colorlinks=true]{hyperref}
\hypersetup{urlcolor=blue,citecolor=red,linkcolor=blue}
\usepackage[initials]{amsrefs}
\definecolor{darkgreen}{rgb}{0.13, 0.55, 0.13}
\numberwithin{equation}{section}
\theoremstyle{plain}
\newtheorem{theorem}{Theorem}[section]
\newtheorem{proposition}[theorem]{Proposition}

\theoremstyle{definition}
\newtheorem*{rh-pb*}{Basic RH problem}
\newtheorem*{rh-I*}{RH problem normalized as $I$ at $\infty$}
\newtheorem*{sol-rh-pb*}{Soliton RH problem}
\newtheorem*{data*}{Data of this RH problem associated with $\BS{u_0(x)}$}
\theoremstyle{remark}
\newtheorem{remark}[theorem]{Remark}
\newtheorem*{notations*}{Notations}
\providecommand{\BS}[1]{\boldsymbol{#1}}  
\providecommand{\D}[1]{\mathbb{#1}}
\newcommand{\dd}{\mathrm{d}}
\newcommand{\eul}{\mathrm{e}}
\newcommand{\ee}{\mathrm{e}}
\newcommand{\ii}{\mathrm{i}}
\newlength{\dhatheight}
\newcommand{\doublehat}[1]{%
    \settoheight{\dhatheight}{\ensuremath{\hat{#1}}}%
    \addtolength{\dhatheight}{-0.3ex}%
    \hat{\vphantom{\rule{1pt}{\dhatheight}}%
    \smash{\hat{#1}}}}

\renewcommand{\Im}{\operatorname{Im}}
\renewcommand{\Re}{\operatorname{Re}}

\newcommand{\ord}{\mathrm{O}}
\newcommand{\supp}{\operatorname{supp}}
\DeclareMathOperator{\Res}{Res}

\newcommand{\red}{\textcolor{red}}

\newif\ifshort
\shorttrue

\title{The periodic Camassa-Holm equation by the Riemann-Hilbert problem approach}

\author[Anne Boutet de Monvel]{Anne Boutet de Monvel}
\address{AB: Institut de Mathématiques de Jussieu–PRG,
Université Paris Cité,
Campus des Grands Moulins,
Bâtiment Sophie Germain
8, place Aurélie Nemours
75013 Paris, France}
\email{\href{mailto:anne.boutet-de-monvel@imj-prg.frr }{anne.boutet-de-monvel@imj-prg.fr }}

\author[Iryna Karpenko]{Iryna Karpenko}
\address{IK: Faculty of Mathematics\\ University of Vienna\\
Oskar-Morgenstern-Platz 1\\ 1090 Wien\\ Austria\\ and B. Verkin Institute for Low Temperature Physics and Engineering\\ 47, Nauky ave\\ 61103 Kharkiv\\ Ukraine}
\email{\href{mailto:iryna.karpenko@univie.ac.at}{iryna.karpenko@univie.ac.at}}

\author[Dmitry Shepelsky]{Dmitry Shepelsky}
\address{DS: B.Verkin Institute for Low Temperature Physics and Engineering, Kharkiv, Ukraine\\ and Kyiv School of Economics, Ukraine}
\email{\href{mailto:shepelsky@yahoo.com }{shepelsky@yahoo.com}}

\author[Lech Zielinski]{Lech Zielinski}
\address{LZ: Université du Littoral Côte d'Opale, France}
\email{\href{mailto:lech.zielinski@univ-littoral.fr }{lech.zielinski@univ-littoral.fr }}

\begin{document}

\begin{abstract}

This work addresses the development of the Riemann-Hilbert problem (RHP) formalism (the Fokas method) for the Camassa-Holm equation under periodic boundary conditions. Particularly, we present a representation of the solution to this problem in terms of the solution of the associated Riemann-Hilbert problem, the data for which are determined by the intial data for the problem in terms of the associated spectral functions.

\end{abstract}

\maketitle

\section{Introduction}

The inverse scattering transform (IST) method, introduced nearly fifty years ago, provides a powerful analytical framework for solving a class of one-dimensional nonlinear evolution equations known as \emph{integrable equations} on the real line. The method is based on the analysis of an associated Lax pair -- a pair of linear operators typically depending on a spectral parameter whose compatibility condition yields the nonlinear equation \cites{IST01,IST02,ablowitz1974inverse, ablowitz1991solitons}.

The IST consists of three main components: (i) the solving of a direct spectral problem, which maps initial data to spectral data via a linear auxiliary system; (ii) the characterization of the time evolution of the spectral data; and (iii) the solving of an inverse problem, through which the original nonlinear PDE is reconstructed at arbitrary times.  The IST thereby transforms a nonlinear evolution equation into a linear spectral flow, making it a nonlinear analogue of the Fourier transform.

The IST method was originally developed for initial value problems on the real line, where both the initial data and the solution decay rapidly as the spatial variable tends to infinity. In this setting, the direct spectral problem reduces to a standard scattering problem. The inverse step, in its classical formulation, is carried out via the Gelfand–Levitan–Marchenko integral equations \cite{novikov1984theory}. Alternatively, the inverse problem can be reformulated as a Riemann--Hilbert (RH) factorization problem in the complex spectral plane \cite{trogdon2015riemann}.

Extending the inverse scattering transform to problems on the half-line or finite intervals presents significant analytical challenges due to the presence of boundary conditions. These difficulties were overcome by the development of the Unified Transform Method, also known as the Fokas Method. Introduced by Fokas in \cite{fokas1997}, and further developed by numerous authors  \cite{fokas2008unified}, this method generalizes the IST by simultaneously analyzing both equations of the Lax pair. A central component of the approach is the \emph{global relation} which connects the spectral transforms of the initial and boundary data. For certain boundary conditions -- classified as \emph{linearizable} -- the global relation can be explicitly resolved, allowing the formulation of a Riemann–Hilbert problem in terms of given data alone. In such cases, the method leads to a complete spectral characterization of the solution for a well-posed problem, where the initial and boundary data are sufficient to determine the evolution uniquely and consistently.

In \cites{DFL21, FL21}, it was shown that the initial-boundary value problem for the nonlinear Schrödinger (NLS) equation on a finite interval $[0, L] \subset \mathbb{R}$, under periodic boundary conditions $q(0,t) = q(L,t)$, $q_x(0,t) = q_x(L,t)$, belongs to the linearizable case. For both the focusing and defocusing cases, the solution $q(x,t)$ can be expressed in terms of an RH problem whose jump matrix and residue conditions are constructed from the scattering matrix for a spectral problem on the whole line associated with $q(x,0)$ continued by $0$ on ${\mathbb R}\setminus [0,L]$. The RH contour consists of the real and imaginary axes, together with a (possibly infinite) collection of finite segments symmetric with respect to the real axis.

Recently \cite{SKPB24},
building on this method,
it has been shown that periodic 
finite-band (algebro-geometric)  solutions to the NLS equation can be 
characterized in terms of their values as data for initial value problem in the periodic setting. More precisely,
a finite-band, periodic solution can be 
given in terms of the solution of  associated RH problem, the 
jump contour for which is specified by the parameters of the 
Riemann surface associated with this  algebro-geometric solution, and the associated 
jump conditions are given in terms of the scattering 
functions that are associated with the  solution of the NLS equation evaluated at a fixed time.

The present paper is devoted to periodic problems 
for the Camassa-Holm (CH) equation
\begin{equation}\label{CH}
    u_t - u_{txx} + 2 u_x + 3uu_x = 2u_x u_{xx} + u u_{xxx},
\end{equation}
which can be written more compactly in terms
the \emph{momentum} $m:=u-u_{xx}$ as
\begin{equation}\label{CH-m}
m_t+2 u_x m + u m_x+ 2  u_x=0, \quad m=u-u_{xx}
\end{equation}
or, assuming that $m+1>0$, as
\begin{equation}\label{CH-m1}
(\sqrt{m+1})_t=-(u\sqrt{m+1})_x,\quad m=u-u_{xx}.
\end{equation}
This nonlinear partial differential equation models the unidirectional propagation of shallow water waves over a flat bottom. It is derived as an asymptotic approximation of the Euler equations, and it preserves key geometric and Hamiltonian structures, making it an integrable system with a bi-Hamiltonian formulation and an infinite sequence of conservation laws.

The CH equation  exhibits a rich variety of wave phenomena, including both smooth periodic waves and peaked traveling waves (peakons), which are continuous but possess discontinuous first derivatives.  Unlike classical dispersive equations such as the Korteweg–de Vries (KdV) equation, the CH equation permits wave breaking: the solution may remain bounded while its spatial derivative blows up in finite time.

We consider the CH equation on a finite spatial interval \( x \in [0, L] \), supplemented with the periodicity conditions
\begin{equation}\label{period}
    u(0,t) = u(L,t), \quad u_x(0,t) = u_x(L,t), \quad 
    u_{xx}(0,t) = u_{xx}(L,t), \quad
    \text{for all } t \geq 0,
\end{equation}
and the initial condition
\begin{equation}\label{ic}
    u(x,0) = u_0(x), \quad x \in [0,L].
\end{equation}
Periodic problems for the CH equation have been studied extensively since 1998
\cite{CE98}.
Particularly, it was shown that 
 the equation is locally well-posed in Sobolev spaces, and for a wide class of initial data, the corresponding solutions exist globally in time. In particular, if the initial momentum $m=u-u_{xx}$
  does not change sign, the solution exists for all time.
 On the other hand, smooth initial data that violate this sign condition may lead to wave breaking in finite time. 

 A significant step in understanding the post-blow-up regime came through the construction of global weak solutions. Holden and Raynaud~\cite{HoldenRaynaud2007}, and later Grunert, Holden, and Raynaud~\cite{GrunertHoldenRaynaud2013}, developed a Lagrangian formulation that allows conservative global solutions even after singularity formation. This framework tracks the evolution of characteristics and the associated momentum, enabling the continuation of solutions beyond wave breaking without energy dissipation.

In the present paper, we develop the RH method for solving the initial value problem 
for the CH equation with the periodic boundary
conditions. Correspondingly, our main aim 
is to obtain a representation of the solution 
of this problem in terms of the solution of an associated RH problem, the data for which
are determined in terms of the spectral functions associated with the initial data.

The outline of the paper is as follows. In Section 2, we discuss 
additional assumptions concerning the initial data.
In Section 3, we briefly present the ideas 
of  the RH approach to the  initial value problem
(on the line) for the CH equation and their
 adaptation to initial boundary 
value problems. In Section 4, we first discuss 
the relationship between solutions of the CH equation in the original variables $(x,t)$
and solutions of the CH equation in variables $(y,t)$ suitable for the RH formalism.
Then we discuss 
how a function solving, locally, the CH equation 
in $(y,t)$ variables appears from the solution of a Riemann-Hilbert
problem parametrized by $y$ and $t$. In Section 5, we describe how, using the 
characterization of  the boundary values in the spectral terms (``Global Relations''),
to arrive at a RH problem with data determined in terms of the initial data only, and prove that 
the solution of the CH problem obtained in terms of the solution of this RH problem 
(i) indeed satisfied the prescribed initial condition and (ii) is periodic in $y$ for any fixed $t$.

\section{Assumptions on initial data}

We consider the periodic problem for the CH equation \eqref{CH-m}, \eqref{period}, \eqref{ic}.
We additionally assume that (i) $m(x,0)+1=u_0(x)-u_{0xx}(x)+1> 0$ for all $x\in[0,L]$ and (ii) $m(0,0)=m(L,0)=0$.

\begin{remark}
    Observe that the condition 
 $m(x,0)+1> 0$ ensures that $m(x,t)+1>0$ for all $t$ \cite{CE98}.
 This positivity is crucial for our analysis, as it allows us to properly define the Jost solutions, which are fundamental for construction of the associated Riemann-Hilbert problem.

\end{remark}

\begin{remark}

The second assumption $m(0,0)=m(L,0)=0$ may seem restrictive, but it turns out that many cases can be reduced to it through suitable linear transformations.

Namely, suppose that $u(x,t)$ satisfies the 
equation 
\begin{equation}\label{CH-m-}
m_t+2 u_x m + u m_x+ 2\omega  u_x=0, \quad m=u-u_{xx}
\end{equation}
with some $\omega\geq0$ and conditions
\eqref{period}, \eqref{ic} 
such that $m_0(0)=m_0(L)=A$ (where $m_0(x):=u_0(x)-u_{0xx}(x)$) with $A$ satisfying 
 either (i) $A+\omega>0$ or (ii) $A+\omega<0$. 
Furthermore, assume that, respectively, either 
\begin{enumerate}[(i)]
    \item $m_0(x)+\omega>0$ for all $x\in [0,L]$  or
    \item $m_0(x)+\omega<0$ for all $x\in [0,L]$.
\end{enumerate}
Then the function
\[\tilde u(x,t):=\frac{1}{A+\omega}\left(u\left(x+\frac{A}{A+\omega}t,\frac{t}{A+\omega}\right)-A\right)\]
satisfies \eqref{CH-m}, \eqref{period}. Moreover, for $\tilde m= \tilde u-\tilde u_{xx}$ we have $\tilde m(0,0)=\tilde m(L,0)=0$ and $\tilde m(x,0)+1>0$ for all $x\in[0,L]$.

\end{remark}

\section{RH problem formalism for the CH equation}\label{sec:3}
\subsection{The Lax pair}

The Camassa-Holm (CH) equation is the compatibility condition ($\psi_{xxt}=\psi_{txx}$) for 2 linear equations ({Lax pair):
\begin{equation}\label{lax-prelim}
    \psi_{xx}=\frac{1}{4}\psi+\lambda(m+1)\psi,\quad 
    \psi_{t}=\left(\frac{1}{2\lambda}-u\right)\psi_x+\frac{1}{2}u_x\psi
\end{equation}
Introducing $\Phi$  by
\begin{equation}\label{lax-uv}
    \Phi:=\frac{1}{2}\left(\begin{matrix}
    1 & -\frac{1}{ik} \\ 1 & \frac{1}{ik} 
\end{matrix}\right)\left(\begin{matrix}
    (m+1)^\frac{1}{4} & 0  \\ 0 & (m+1)^{-\frac{1}{4}}
\end{matrix}\right)
\left(\begin{matrix}
    \psi  \\ \psi_x
\end{matrix}\right),
\end{equation}
the Lax pair can be written in the form of two-component,  first order differential  equations 
\begin{equation}\label{lax}
\begin{aligned}
    \Phi_{ x}+\ii k p_x \sigma_3  \Phi& =U  \Phi,\\
   \Phi_{t}+\ii k p_t \sigma_3  \Phi&=V \Phi,
   \end{aligned}
\end{equation}
where $\sigma_3=\left(\begin{smallmatrix}
    1 & 0 \\ 0 & -1
\end{smallmatrix}\right)$, $k^2=-\lambda-\frac{1}{4}$, $p_x=\sqrt{m+1}$, $p_t=-u\sqrt{m+1}+\frac{1}{2\lambda}$
(the compatibility of these equations is just the CH equation
in the form \eqref{CH-m1}), 
\begin{equation}\label{U}
    U(x,t,k)=\frac{1}{4}\frac{m_x}{m+1}\begin{pmatrix}
        0 & 1 \\
1 & 0
    \end{pmatrix}-\frac{1}{8\ii k}\frac{m}{\sqrt{m+1}}\begin{pmatrix}
        -1 & -1 \\
1 & 1
    \end{pmatrix},
\end{equation}
\begin{equation}\label{V}
\begin{aligned}
    V(x,t,k)=&\left (\frac{1}{4}\frac{m_t}{m+1}+\frac{u_x}{2}\right )\begin{pmatrix}
        0 & 1 \\
1 & 0
    \end{pmatrix}+\frac{1}{8\ii k}\frac{u(m+2)}{\sqrt{m+1}}\begin{pmatrix}
        -1 & -1 \\
1 & 1
    \end{pmatrix}+\\
    &+\frac{\ii k}{4\lambda}\left( \sqrt{m+1} \begin{pmatrix}
        -1 & 1 \\
-1 & 1
    \end{pmatrix} + \frac{1}{\sqrt{m+1}}\begin{pmatrix}
        -1 & -1 \\
1 & 1
    \end{pmatrix}
\right)+\frac{\ii k}{2\lambda}\sigma_3.
\end{aligned}
\end{equation}

As for $p(x,t,k)$, it is defined up to a constant (independent of $x$ and $t$), which can be chosen to match the context; 
e.g., dealing with problems on the whole line, it is convenient to define $p$ by 
\begin{equation}\label{p-inf}
   p(x,t,k)= x+ \int_{-\infty}^x (\sqrt{m(\xi,t)+1}-1)\dd\xi+\frac{t}{2\lambda} 
\end{equation}
whereas for problems in domains with the boundary $x=0$,
the appropriate definition for $p$ is
\begin{equation}\label{p}
   p(x,t,k)=\int_0^x \sqrt{m(\xi,t)+1}\dd\xi-\int_0^t u(0,\zeta)\sqrt{m(0,\zeta)+1}\dd\zeta+\frac{t}{2\lambda}. 
\end{equation}

Notice that at $\lambda=0$ ($k=\pm\frac{i}{2}$), the coefficients of the 
first equation ($x$-equation) in \eqref{lax-prelim} becomes independent of $u$
whereas the coefficients of the 
second equation  ($t$-equation) in \eqref{lax-prelim} are 
singular. This suggest 
 introducing another Lax pair 
 having its coefficients (apart from a diagonal term)
 regular at 
 $k\to\pm\frac{i}{2}$. 
Let $\Phi_0:=\frac{1}{2}\left(\begin{smallmatrix}
    1 & -\frac{1}{ik} \\ 1 & \frac{1}{ik}
\end{smallmatrix}\right)
\left(\begin{smallmatrix}
    \psi  \\ \psi_x
\end{smallmatrix}\right)$, 
$p_0(x,t,k) = x+ \frac{t}{2\lambda}=x-\frac{2t}{4k^2+1}$; then
\begin{equation}\label{lax-0}
\begin{aligned}
    \Phi_{0 x}+\ii k p_{0x} \sigma_3  \Phi_0&=U_0  \Phi_0\\
   \Phi_{0t}+\ii k p_{0t} \sigma_3  \Phi_0&=V_0 \Phi_0
\end{aligned}
\end{equation}
with 
\begin{equation}\label{U0}
    U_0(x,t,k)=\frac{\lambda}{2\ii k}m(x,t)\begin{pmatrix}
        -1 & -1 \\
1 & 1
    \end{pmatrix},
\end{equation}

\begin{align}\label{V0}
    V_0(x,t,k)=&\frac{u_x}{2}\begin{pmatrix}
        0 & 1 \\
1 & 0
    \end{pmatrix}+\frac{u}{4\ii k}\begin{pmatrix}
        0 & -1 \\
1 & 0
    \end{pmatrix}+\frac{\lambda u}{2 \ii k}\left( 2\sigma_3 - m\begin{pmatrix}
        -1 & -1 \\
1 & 1
    \end{pmatrix}
\right).
\end{align}
We will see below that the Lax pair \eqref{lax-0} is convenient
for studying boundary value problems for the CH equation, where
the $t$-equation play the role similar to that of the $x$-equation.

\begin{remark}
    The matrix coefficients $U$, $V$ in \eqref{lax}
    are singular at $k=0$ as well, but these singularities 
    are of the special matrix type: the leading coefficients 
    are nilpotent matrices.
\end{remark}

\subsection{RH problem formalism for problem on the line}\label{sec:RH-line-xt}
\subsubsection{Jost solutions and scattering relation}

The Riemann-Hilbert problem formalism for studying problems
for integrable nonlinear PDEs is based on using dedicated 
solutions of the Lax pair equations. In particular, for the Cauchy problem on the whole line, where the equation is considered in the half-plane $-\infty<x<\infty$, $t>0$
and the data are the initial condition $u(x,0)=u_0(x)$, 
$x\in (-\infty, \infty)$, the RH formalism is built 
on the Jost solutions to \eqref{lax} normalized
at $x=\pm\infty$ \cite{BS08}.

Namely, 
assuming that $u(x,t)$ is a solution of the CH equation \eqref{CH} satisfying $(1+|x|)u(x,t)\in L^1 (-\infty, \infty)$ for all $t$,
the (matrix-valued) Jost solutions $\Phi_\pm$ for \eqref{lax}
are defined by $\Phi_\pm=\hat\Phi_\pm e^{-ikp(x,t,k)\sigma_3}$,
where 
 $\hat\Phi_\pm$ are 
the solutions of the integral equations 
 ($e^{\hat\sigma_3}A\equiv e^{\sigma_3}Ae^{-\sigma_3}$)
\begin{equation}\label{Phi-hat}
    \hat\Phi_\pm(x,t,k) = I+\int_{\pm\infty}^x
e^{-ik(p(x,t,k)-p(y,t,k))\hat\sigma_3}(U\hat\Phi_\pm)dy,
\end{equation}
where $p(x,t,k)-p(y,t,k)=\int_y^x \sqrt{m(\xi,t)+1}\dd\xi$
doesn't depend on $k$.

It turns out that  
\begin{enumerate}
    \item 
    the columns of $\Phi_\pm$ are analytic either in the upper
    (${\mathbb C}_+$)
    or  lower (${\mathbb C}_-$) half-planes w.r.t. $k$; namely, $\Phi_+^{(1)}$ and $\Phi_-^{(2)}$ are analytic 
    in ${\mathbb C}_-$ whereas $\Phi_-^{(1)}$ and $\Phi_+^{(2)}$ are analytic 
    in ${\mathbb C}_+$;
    \item 
    $\hat\Phi_\pm\to I$ as $k\to\infty$ (to be understood columnwise, for $k$ in the corresponding half-plane);
    \item 
    being the solutions of  the same system of ordinary
    differential equations \eqref{lax}, $\Phi_+(x,t,k)$
    and $\Phi_-(x,t,k)$ are related by a matrix independent 
    of $x$ and $t$:
    \begin{equation}\label{sr}
        \Phi_+(x,t,k)=\Phi_-(x,t,k) s(k),\quad k\in\mathbb{R}\setminus \{0\},
        \end{equation}
        where, due to the respective symmetries of $U$ and $V$,
        \begin{equation}\label{a-b} 
s(k) = \begin{pmatrix}
 a^*(k) & b(k) \\ b^*(k) & a(k)
\end{pmatrix} 
    \end{equation}
with $a^*(k):=\overline{a(\bar k)}$ and  
$b^*(k):=\overline{b(\bar k)}$;
\item 
evaluating \eqref{Phi-hat} and \eqref{sr} at $t=0$,  the spectral functions
$\{a(k),b(k)\}$
can be obtained solving the equations for the Jost solutions taken at $t=0$
and thus they are determined by  $u_0(x)$ only.
\item 
in general (if $(1+|x|)u_0(x)\in L^1 (-\infty, \infty)$) the spectral (scattering) function $a(k)$ is analytic in 
${\mathbb C}_+$ whereas $b(k)$ is determined
for $x\in \mathbb{R} $ only. On the other hand,
if $u_0(x)$ is finitely supported, e.g., $\supp u_0 =[0,L]$, then $a(k)$
and $b(k)$ are entire functions satisfying the estimates
\begin{equation}\label{ab-estim}
    a(k) = 1 + O\left(\frac{1}{k}\right)+O\left(\frac{\eul^{2\ii k\theta}}{k}\right)  , \quad b(k) = O\left(\frac{1}{k}\right)+O\left(\frac{\eul^{2\ii k\theta}}{k}\right), \quad k\to\infty;
\end{equation}
\item 
as $k\to 0$, $\Phi_\pm(x,t,k) = \frac{i \alpha(x,t)}{k}
\begin{pmatrix}
    1 & 1 \\ -1 & -1
\end{pmatrix}+o(1)$ with some scalar $\alpha_\pm\in \mathbb R$ (understood columnwise).
It follows that 
\begin{equation}\label{ab-0}
    a(k)=\frac{\ii \rho}{k} +a_0+O(k), \quad b(k)=-\frac{\ii \rho}{k} +b_0+O(k), \quad k\to 0,
\end{equation}
where $a_0\in \mathbb R$ and $b_0\in \mathbb R$
with $\rho\in \mathbb R$ such that, gererically,    $\rho\ne 0$ (whereas in the non-generic case,
$\rho = 0$; in this case, $a_0^2-b_0^2=1$).

\end{enumerate}

\subsubsection{Pre-Riemann-Hilbert problem  in $x,t$ variables}

 Collecting the columns of the Jost solutions analytic in the
same half-plane and introducing $\Psi:=\begin{cases}
	\left(\frac{\Phi_-^{(1)}(x,t,k)}{a(k)}, \Phi_+^{(2)}(x,t,k)\right), & k\in \D{C}_+,\\
	\left(\Phi_+^{(1)}(x,t,k), \frac{\Phi_-^{(2)}(x,t,k)}{a^*(k)}\right), & z\in \D{C}_-
\end{cases} $ and $r(k):=\bar b(k)/a(k)$, 
the scattering relation \eqref{sr} can be rewritten as follows:
\begin{equation}\label{jump-psi}
    \Psi_+(x,t,k)=\Psi_-(x,t,k) J_{0\mathbb R}(k), \qquad 
k\in \mathbb{R}\setminus\{0\}
\end{equation}
with 
\begin{equation}\label{j0}
    J_{0\mathbb R}(k):=\begin{pmatrix}
    1-|r(k)|^2 & \bar r(k) \\ -r(k) & 1
\end{pmatrix}.
\end{equation}
The division by $a$ and $a^*$ above, which has been introduced 
in order to have $\det\Psi\equiv 1$, necessitates addressing the singularities of $\Psi$ at eventual zeros of $a$ 
in ${\mathbb{C}}_+$ (as well as the (associated by symmetry) zeros of  $a^*$ in ${\mathbb{C}}_-$), which are  eigenvalues of the operator generated by the first equation 
in \eqref{lax} (in what follows, we assume that the zeros are simple, which corresponds to the generic case). Indeed, if $a(k_j)=0$ at some $k_j=i\nu_j$, $j=1,\dots,N$, with 
$0<\nu_j<1/2$, then 
$\Phi_-^{(1)}(x,t,k_j)b_j = \Phi_+^{(2)}(x,t,k_j)$ with some $b_j$ and thus
\begin{equation}\label{res-psi1}
    \Res_{k=k_j} \Psi^{(1)}(x,t,k) = c_j \Psi^{(2)}(x,t,k_j) \quad \text{with}\ 
    c_j=\frac{1}{b_j\dot a(k_j)}.
\end{equation}
Similarly, 
\begin{equation}\label{res-psi2}
    \Res_{k=\bar k_j} \Psi^{(2)}(x,t,k) = \bar c_j \Psi^{(1)}(x,t,\bar k_j).
\end{equation}

Notice that in the case of finitely supported $u_0(x)$, $\{b_j\}$is given in terms 
of $b(k)$: 
\begin{equation}
    \label{b-bj}
    b_j=b(k_j)=-\frac{1}{b^*(k_j)},
\end{equation}
where the last equality is due to the determinant relation 
$\det s(k)=a^*(k)a(k)-b^*(k)b(k)\equiv 1$.

Introduce $M(x,t,k):= \Psi(x,t,k)e^{ikp(x,t,k)\sigma_3}$. Then the analytic properties
of $\Phi_\pm$ presented above allows \emph{characterizing} $M$ by 
a set of conditions involving $u_0(x)=u(x,0)$ only, through the associated spectral data.
\begin{proposition}\label{prop:RHP-xt}
The $2\times 2$ function $M(x,t,k)$ constructed from the Jost solutions of the 
Lax pair equations generated by a solution of the CH equation $u(x,t)$
can be characterized as that satisfying the following conditions:
\begin{enumerate}
    \item For all $x$ and $t$, $M(x,t,k)$ is a sectionally meromorphic function (w.r.t.     $\mathbb R$).
    \item The limiting values of $M$ as $k$ approaches $\mathbb R$ from the upper 
    and lover half-planes ($M_+$ and $M_-$ respectively) are related by the jump 
    condition
    \begin{equation}\label{jump-xt_line}
        M_+(x,t,k) = M_-(x,t,k)J_{\mathbb R}(x,t,k)\ \ \text{with}\ 
       J_{\mathbb R}(x,t,k)= e^{-ikp(x,t,k)\sigma_3}J_{0 \mathbb R}(k)  
       e^{ikp(x,t,k)\sigma_3},\ \ 
               k\in \mathbb R\setminus \{0\}.
    \end{equation}
  \item 
  as $k\to 0$, $M_+(k)(x,t,k) = \frac{i \alpha(x,t)}{k}
\begin{pmatrix}
    \delta & 1 \\ -\delta & -1
\end{pmatrix}+o(1)$,
where \[
\delta = \begin{cases}
    0, & \lim\limits_{k\to 0}ka(k)\ne 0 \ \ \text{and}\ 
    \lim\limits_{k\to 0}kb(k)\ne 0\ \text{(generic case)},\\
    1-\frac{b_0}{a_0}, & a_0= \lim\limits_{k\to 0}a(k), \ \ 
    b_0= \lim\limits_{k\to 0}b(k),
\end{cases}
\]
with some  $\alpha\in \mathbb R$.
\item $M\to I$ as $k\to\infty$.
\item 
\begin{equation}\label{M-x-sym}
    M(-k)=\begin{pmatrix}
	0 & 1 \\ 1 & 0
	\end{pmatrix} M(k) \begin{pmatrix}
	0 & 1 \\ 1 & 0
	\end{pmatrix}.
\end{equation}
\item 
$M^{(1)}$ has simple poles at 
$k_j=i\nu_j$, $j=1,\dots,N$, with 
$0<\nu_j<1/2$
and $M^{(2)}$ has simple poles at $k=\bar k_j=-i\nu_j$, and the residue conditions hold:
\begin{align}\label{M-x-res}
    \Res_{k=i\nu_j} M^{(1)}(x,t,k) & = c_j e^{-2\nu_j p(x,t,i\nu_j)}M^{(2)}(x,t,i\nu_j),\\
    \Res_{k=-i\nu_j} M^{(2)}(x,t,k) & = \bar c_j e^{-2\nu_j p(x,t,i\nu_j)} M^{(1)}(x,t,-i\nu_j),
\end{align}
where $c_j=\frac{b_j}{\dot a(k_j)}\in i{\mathbb R}$.
\item 
\begin{equation}\label{struc}
    M\left(\frac{i}{2}\right)=\frac{1}{2}\begin{pmatrix}
	q+\dfrac{1}{q} & q-\dfrac{1}{q} \\[5mm] q-\dfrac{1}{q} & q+\dfrac{1}{q}
	\end{pmatrix} \begin{pmatrix}
	f & 0 \\ 0 & f^{-1}	\end{pmatrix},
	\end{equation}
    with some  $q(x,t)>0$ and $f(x,t)>0$.
\end{enumerate}
Here $\alpha(x,t)$, $q(x,t)$, and $f(x,t)$ are not specified
whereas $a(k)$, $b(k)$, $\{\nu_j\}_1^N$, and $\{c_j\}_1^N$ are determined 
by $u_0(x)=u(x,0)$ through the solutions of the associated integral equations 
for  the Jost functions considered for $t=0$.
\end{proposition}
\begin{remark}
    The last condition \eqref{struc} follows from the relationship between 
    the Jost solutions of the Lax pair \eqref{lax} and those
    of the Lax pair \eqref{lax-0} taking into account that the first
    equation in \eqref{lax-0} doesn't involve $u$ at $k=\frac{i}{2}$,
    see \cite{BS08}. Moreover, \eqref{struc} gives means for obtaining
    the solution $u(x,t)$ of the CH equation provided $M(x,t,k)$ is known:
    \[
    m(x,t)=q^4(x,t)-1.
    \]
\end{remark}
\begin{remark}
    Considering the problem on the line is important for considering
    the problem on an interval, in particular, the periodic problem,
    since it provides means to verify that a constructed solution 
    of the periodic problem indeed satisfies the prescribes initial conditions.
\end{remark}
\begin{remark}
    Conditions in Proposition \ref{prop:RHP-xt} cannot be immediately 
    interpreted as a Riemann-Hilbert problem
    for the problem on the line due to the exponentials 
    involving (through $p(x,t,k)$) the solution of  the CH equation
    $m(x,t)$ for all $t$ (not only for $t=0$ associated with the given initial data).
    The remedy is to introduce a new spatial variable, see Section \ref{sec:CH-y}
    below.
\end{remark}

\subsection{RH problem formalism for problems on an interval}\label{sec:RH-y-int}
\subsubsection{Dedicated solutions of the Lax pair equations}
The problem of adaptation of the Riemann-Hilbert formalism 
to the case of initial \emph{boundary} value problems for formally integrable 
nonlinear equations (having the Lax pair representation) has been systematically 
addressed since the end  of 1990s, see \cites{fokas1997,FIS,BFS03}. 
Particularly, various problems on an interval (with two boundaries in the 
$(x,t)$-plane was addressed in \cites{FIS,BFS06}. The main idea is to use both
equations from the Lax pair simultaneously, as scattering problems. With this respect,
the $x$-equation of the Lax pair gives rise to the scattering problem
associated with the initial data while the $t$-equation gives rise to 
two scattering problems associated with the boundary values at two ends 
of  the interval. Then it is possible to define a Riemann-Hilbert
problem, the data for which (jump and residue conditions) can be given
in terms of  spectral functions involved in the scattering problems  mentioned above.

If the initial boundary value problem is considered
        on the interval $x\in [0,L]$, then,
as  building elements for the associated RHP, one uses 
4 dedicated, simultaneous solutions 
		$\Phi_j$, $j=1,\dots,4$ of the Lax pair equations normalized at 4 points: $(0,T)$, $(0,0)$, $(L,0)$,  $(L,T)$ (with some $T>0$),
        where 
\begin{figure}
    \centering
    \includegraphics[scale=.6]{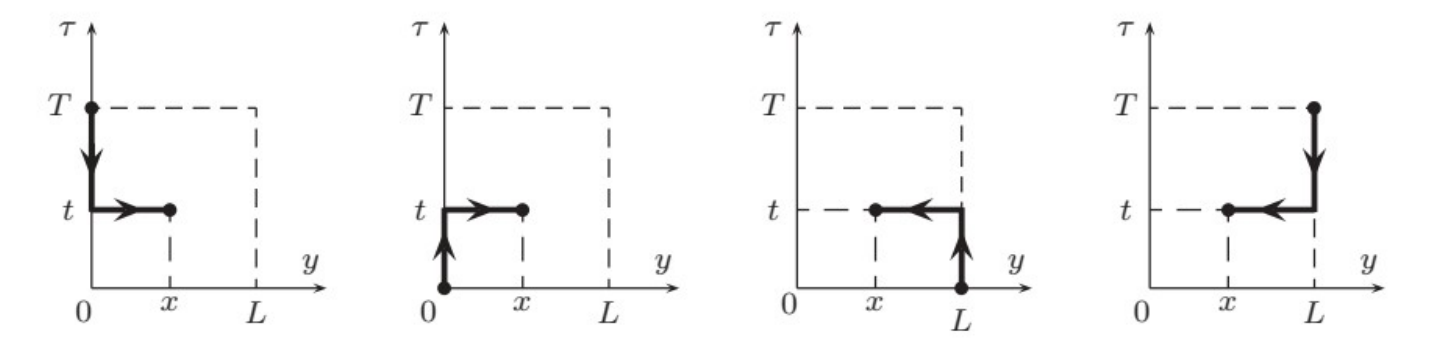}
    \caption{Paths of integration for $\Phi_j$}
    \label{fig:paths}
\end{figure}
\

In the case of the CH equation,     
	the integral equations determining the solutions $\Phi_j$ of \eqref{lax}
    are as follows:
    \begin{align}\label{Phi-j}
        \Phi_j(x,t,k) &= e^{-ikp(x,t,k)\sigma_3} +\int_{(x_j, t_j)}^{(x,t)}
	e^{ik(p(x,t,k)-p(\xi,\tau,k))\sigma_3}\left(
	 U\Phi_j(\xi,\tau,k)d\xi \right.\nonumber\\
	& \left.+  V\Phi_j(\xi,\tau,k)d\tau\right),
    \end{align}
where $p(x,t,k)$	is defined by \eqref{p}. 
Similarly to \eqref{sr}, functions $\Phi_j$ being the solutions of 
\emph{two} differential equations, can be related by matrices independent 
of $x$ and $t$:
\begin{align}\label{s-32}
	\Phi_3(x,t,k) & = \Phi_2(x,t,k)s(k), \\ \label{s-12}
	\Phi_1(x,t,k) & = \Phi_2(x,t,k)S(k), \\
    \label{s-43}
	\Phi_4(x,t,k) & = \Phi_3(x,t,k)e^{ik\theta \sigma_3}S_1(k)e^{-ik\theta \sigma_3}, 
	\end{align}
    where 
    \begin{equation}
        \theta:=p(L,0)= \int_0^L \sqrt{m(\xi,0)+1}\dd\xi
    \end{equation}
(the scattering relation \eqref{s-43} is chosen to have this form in order to provide 
$S_1(k)=S(k)$ in the periodic case, with
$u(0,t)=u(L,t)$, $u_x(0,t)=u_x(L,t)$, and $u_{xx}(0,t)=u_{xx}(L,t)$).

Equation \eqref{s-32} 
 can be seen as a scattering relation, for $k\in \mathbb R$, for the first equation ($x$-equation) in \eqref{lax}, with $u(x,0)$ continued by $0$ for $x \in \D{R}\setminus [0,L]$
 determining $s(k)=\Phi_3(0,0,k) = \begin{pmatrix}
 a^*(k) & b(k) \\ b^*(k) & a(k)
\end{pmatrix}$
while equations \eqref{s-12} and \eqref{s-43} can be viewed as the scattering relations for the $t$-equation in \eqref{lax}. 
More precisely, since $ V(0,t,k)$ is determined by $u(0,t)$, $u_x(0,t)$ and $u_{xx}(0,t)$,
it follows that $S(k)=\Phi_1(0,0,k)\equiv \begin{pmatrix}
			\overline{ A(\bar k )} & B(k) \\\overline{ B(\bar k)} & A(k)
		\end{pmatrix} $ is determined (considering \eqref{s-12} at $x=0)$
        by $u(0,t)$, $u_x(0,t)$ and $u_{xx}(0,t)$.
Similarly, 
        $S_1(k) = \Phi_4(L,0,k)e^{ik\theta \sigma_3}\equiv \begin{pmatrix}
			\overline{A_1(\bar k)} & B_1(k) \\ \overline{B_1(\bar k)} & A_1(k)
		\end{pmatrix} $ is determined (considering \eqref{s-43} at $L=0)$ 
        by $u(L,t)$, $u_x(L,t)$ and $u_{xx}(L,t)$.

 Notice that, similarly to $\Phi_\pm$ in the case of  the problem
 on the whole line, columns of $\Phi_j$ have well-controlled behavior as 
 $k\to\infty$ in the corresponding half-planes $\pm\Im k \ge 0 $
 provided $u(0,t)$ has a definite sign for all $t$.
 Particularly, if $u(0,t)\le 0$, then for all $x\in[0,L]$ and $t\in [0,T]$,
 \begin{itemize}
     \item 
     $\hat \Phi_2^{(1)}(x,t,k)$ and $\hat \Phi_4^{(2)}(x,t,k)$ approach 
     respectively $\begin{pmatrix}
         1 \\ 0
     \end{pmatrix}$ and $\begin{pmatrix}
         0 \\ 1
     \end{pmatrix}$ 
     as $k\to\infty$, $\Im k \ge 0 $.
     \item 
     $\hat \Phi_4^{(1)}(x,t,k)$ and $\hat \Phi_2^{(2)}(x,t,k)$ approach 
     respectively $\begin{pmatrix}
         1 \\ 0
     \end{pmatrix}$ and $\begin{pmatrix}
         0 \\ 1
     \end{pmatrix}$ 
     as $k\to\infty$, $\Im k \le 0 $.
 \end{itemize}
 At the same time, for particular $x$ and $t$ we have the following:
 as $k\to\infty$,
 \begin{itemize}
     \item 
     for $t=0$, 
     $\hat \Phi_3^{(1)}(x,0,k)\to \begin{pmatrix}
         1 \\ 0
     \end{pmatrix}$ for $\Im k \le 0 $ and 
     $\hat \Phi_3^{(2)}(x,0,k)\to \begin{pmatrix}
         0 \\ 1
     \end{pmatrix}$ for $\Im k \ge 0 $.
     \item 
     for $x=0$, 
     $\hat \Phi_1^{(1)}(0,t,k)\to \begin{pmatrix}
         1 \\ 0
     \end{pmatrix}$ for $\Im k \le 0 $ and 
     $\hat \Phi_1^{(2)}(0,t,k)\to \begin{pmatrix}
         0 \\ 1
     \end{pmatrix}$ for $\Im k \ge 0 $.
     \item 
     for $x=L$, 
     $\hat \Phi_3^{(1)}(L,t,k)\to \begin{pmatrix}
         1 \\ 0
     \end{pmatrix}$ for  $\Im k \ge 0 $ and 
     $\hat \Phi_3^{(2)}(L,t,k)\to \begin{pmatrix}
         0 \\ 1
     \end{pmatrix}$ for  $\Im k \le 0 $.
 \end{itemize}
Consequently, 
the spectral functions  have the following properties as $k\to\infty$:
\begin{itemize}
    \item 
    \begin{align}\label{s_inf}
a(k)&=1+O\left(\frac{1}{k}\right)+O\left(\frac{\eul^{2\ii k\theta}}{k}\right),\\\label{b_inf}
b(k)&=O\left(\frac{1}{k}\right)+O\left(\frac{\eul^{2\ii k\theta}}{k}\right).
\end{align}
Particularly, for $\Im k\ge 0$,
\begin{equation}\label{a-b-as}
    a(k)=1+O\left(\frac{1}{k}\right), \quad b(k)=O\left(\frac{1}{k}\right).
\end{equation}
\item For $\Im k\ge 0$,
\begin{align}\label{S_inf}
A(k)&=1+O\left(\frac{1}{k}\right), \quad A_1(k)=1+O\left(\frac{1}{k}\right),\\
B(k)&=O\left(\frac{1}{k}\right),\quad B_1(k)=O\left(\frac{1}{k}\right).
\end{align}
\end{itemize}

Due to the respective symmetries of $U$ and $V$, we have
\begin{equation}\label{sym_phi}
 \overline{\Phi_j(\bar k)}=\Phi_j(-k)=\begin{pmatrix}
        0&1\\1&0  \end{pmatrix}\Phi_j(k)\begin{pmatrix}
        0&1\\1&0
    \end{pmatrix}.   
\end{equation}
Consequently, the spectral functions satisfy the following symmetry relations:
\begin{align}\label{symm_a}
    &a(-k) = \overline{a(\bar{k})}, 
    \qquad 
    b(-k) = \overline{b(\bar{k})},\\\label{symm_A}
    &A(-k) = \overline{A(\bar{k})}, 
    \qquad 
    B(-k) = \overline{B(\bar{k})}
    ,\\\label{symm_A_1}
    &A_1(-k) = \overline{A_1(\bar{k})}, 
    \qquad 
    B_1(-k) = \overline{B_1(\bar{k})}
\end{align}
Moreover, since $\det \Phi_j\equiv 1$, the spectral functions obey the determinant identities
\begin{equation}\label{det_rel}
a(k)a^*(k) - b(k)b^*(k) = 1;\quad A(k)A^*(k) - B(k)B^*(k) = 1;\quad A_1(k)A_1^*(k) - B_1(k)B_1^*(k) = 1
\end{equation}

As for the behavior at $k=0$, we have 
\begin{equation}\label{Phi_at_0}
    \Phi_{ j}(x,t,k)=\frac{\ii}{k}c_j(x,t)\begin{pmatrix}
        1&1\\-1&-1   \end{pmatrix}+\begin{pmatrix}     a_{0j}&b_{0j}\\b_{0j}& a_{0j}
    \end{pmatrix}(x,t)+\ii k \begin{pmatrix}
        a_{1j}&b_{1j}\\-b_{1j}& -a_{1j}
    \end{pmatrix}(x,t)+ O(k^2), \quad k\to 0
    \end{equation}
    In particular, as $k\to 0$ we have
    \begin{subequations}\label{ab_at_0}
\begin{align}
    \label{a_at_0}
    & a(k)=-\frac{\ii}{k}\rho+a_{0}-\ii k a_{1}+O(k^2),\\\label{b_at_0}
    & b(k)=\frac{\ii}{k}\rho+b_{0}+\ii k b_{1}+O(k^2),\\
    \label{A*_at_0}
   & A^*(k)=\frac{\ii}{k}\tilde \rho+A_{0}+\ii k A_{1}+O(k^2),\\\label{B*_at_0}
   &  B^*(k)=-\frac{\ii}{k}\tilde \rho+B_{0}-\ii k B_{1}+O(k^2)    
\end{align}
  \end{subequations} 
where $\rho=c_3(0,0)$, $\tilde \rho=c_1(0,0)$, $a_0=a_{03}(0,0)$, $b_0=b_{03}(0,0)$, $A_0=a_{01}(0,0)$, $B_0=b_{01}(0,0)$, $a_1=a_{13}(0,0)$, $b_1=b_{13}(0,0)$, $A_1=a_{11}(0,0)$, and $B_1=b_{11}(0,0)$.

Now we notice that   
since  $V$ and $p$ are singular at $k=\pm\frac{i}{2}$ and $V$ is involved 
in all integral equations \eqref{Phi-j}, 
        in order to control the behavior of their solutions  
        as $k\to\pm\frac{i}{2}$ it is convenient to use the  Lax pair
        in the form \eqref{lax-0}.
Indeed, defining $\Phi_{0j}$ by (cf. \eqref{Phi-j})         
		\begin{align}\label{Phi-0-j}
        \Phi_{0j}(x,t,k) &= e^{-ikp_0(x,t,k)\sigma_3} +\int_{(x_j, t_j)}^{(x,t)}
	e^{ik(p_0(x,t,k)-p_0(\xi,\tau,k))\sigma_3}\left(
	 U_0\Phi_{0j}(\xi,\tau,k)d\xi \right.\nonumber\\
	& \left.+  V_0\Phi_j(\xi,\tau,k)d\tau\right)
    \end{align}
and $\hat\Phi_{0j}$ by $\Phi_{0j}=\hat\Phi_{0j} e^{-ikp_0(x,t,k)\sigma_3}$, where $p_0(x,t,k) = x-\frac{2t}{4k^2+1}$, and 
noticing that $U_0(x,t,\pm\frac{\ii}{2})\equiv 0$, we have the following 
properties of 
 $\hat\Phi_{0j}$ in vicinities $|k-\frac{\ii}{2}|<\epsilon$ and $|k+\frac{\ii}{2}|<\epsilon$
 of  $k=\pm\frac{i}{2}$ (follow from the Neumann series solutions of Volterra integral equations \eqref{Phi-j}): for all $x\in[0,L]$ and $t\in[0,T]$,

\begin{itemize}
    \item $\hat\Phi_{0 1}^{(1)}(x,t,k)$ 
    and $\hat\Phi_{0 3}^{(2)}(x,t,k)$ are analytic and bounded 
     in  $k\in\{ |k|<\frac{1}{2}, |k-\frac{\ii}{2}|<\epsilon\}$; moreover, 
     for $t>0$, $(\hat\Phi_{0 1}^{(1)}\ \hat\Phi_{0 3}^{(2)})(x,t,\frac{\ii}{2}) = I$, where the limit $k\to\frac{i}{2}$ is taken along any non-tangential direction.

\item $\hat\Phi_{0 2}^{(1)}(x,t,k)$ 
    and $\hat\Phi_{0 4}^{(2)}(x,t,k)$ are analytic and bounded 
     in  $k\in\{  |k|>\frac{1}{2}, |k-\frac{\ii}{2}|<\epsilon\}$; moreover, for $t>0$,
     $(\hat\Phi_{0 2}^{(1)}\ \hat\Phi_{0 4}^{(2)})(x,t,\frac{\ii}{2}) = I$.

      \item $\hat\Phi_{0 3}^{(1)}(x,t,k)$ 
    and $\hat\Phi_{0 1}^{(2)}(x,t,k)$ are analytic and bounded 
     in $k\in\{  |k|<\frac{1}{2}, |k+\frac{\ii}{2}|<\epsilon\}$;
     moreover, for $t<T$,
     $(\hat\Phi_{0 3}^{(1)}\ \hat\Phi_{0 1}^{(2)})(x,t,-\frac{\ii}{2}) = I$.
     
\item $\hat\Phi_{0 4}^{(1)}(x,t,k)$ 
    and $\hat\Phi_{0 2}^{(2)}(x,t,k)$ are analytic and bounded 
     in  $k\in\{  |k|>\frac{1}{2}, |k+\frac{\ii}{2}|<\epsilon\}$; moreover, 
     $(\hat\Phi_{0 4}^{(1)}\ \hat\Phi_{0 2}^{(2)})(x,t,-\frac{\ii}{2}) = I$.
\end{itemize}

Similarly to \eqref{s-32}--\eqref{s-43}, one can introduce the scattering
matrices $\tilde s$, $\tilde S$, and $\tilde S_1$, which relate 
respectively $\Phi_{0 3}$ and $\Phi_{0 2}$, $\Phi_{0 1}$ and $\Phi_{0 2}$,
and $\Phi_{0 4}$ and $\Phi_{0 3}$; we denote their entries by 
$\tilde a$, $\tilde b$, $\tilde A$, $\tilde B$, $\tilde A_1$, and $\tilde B_1$, respectively.

Due to the properties of $\Phi_{0j}$ near $k=\frac{i}{2}$,
the spectral functions $\tilde a(k)$
 and  $\tilde b(k)$ have the following properties as $k\to i/2$:
 \begin{equation}\label{tilde-a-b}
     \tilde a(k) = 1+ O\left(k-\frac{i}{2}\right), \quad 
     \tilde a^*(k) = 1+ O\left(k-\frac{i}{2}\right), \quad
     \tilde b(k) = O\left(k-\frac{i}{2}\right), \quad
     \tilde b^*(k) = O\left(k-\frac{i}{2}\right).
 \end{equation}

Since  $\Phi_0$ and $\Phi$  are related to the same system of ODEs \eqref{lax-prelim},
they are interrelated as follows:
\begin{equation}\label{Phi_0_inf}
 \Phi_{j}(x,t,k)=Q(x,t)\Phi_{0 j}(x,t,k) C_j(k)
\end{equation}
with some $C_j(k)$, where 
\begin{equation}\label{Q}
    Q(x,t)=\frac{1}{2}\begin{pmatrix}
                    q+\frac{1}{q} &  q-\frac{1}{q}\\
 q-\frac{1}{q} &  q+\frac{1}{q}
\end{pmatrix}
\end{equation}
with $q(x,t)=(m(x,t)+1)^{\frac{1}{4}}$.

Particularly, if $Q(0,0)=Q(L,0)=I$, then 
\[
 C_2=I, \quad C_3=\eul^{\ii k\left(L-\theta\right)\sigma_3}, \quad 
    s(k)=\tilde s(k)\eul^{\ii k\left(L-\theta\right)\sigma_3}
\]
   and thus
\begin{equation}\label{s-tilde-s}
a(k)=\tilde a(k)e^{ik(\theta-L)}, \quad b(k)=\tilde b(k)e^{ik(\theta-L)}.
\end{equation}
Moreover, from the  integral equation for $\hat \Phi_{03}(x,0,k)$ we  deduce that 

\[
\hat \Phi_{03}(0,0,k)=I-\ii \left(k-\tfrac{\ii}{2}\right)\begin{pmatrix}
    -\int_0^L m(y,0) \dd y & -\int_0^L m(y,0) \eul^{-y} \dd y \\
\int_0^L m(y,0) \eul^{y} \dd y & \int_0^L m(y,0) \dd y 
\end{pmatrix}+ O\left(\left(k-\tfrac{\ii}{2}\right)^2\right), \quad k\to\frac{\ii}{2}.
\]
Consequently,

\begin{align}\label{ab_at_i2}
a(k)&= \eul^{\frac{L-\theta}{2}}+\ii\left(k-\tfrac{\ii}{2}\right) \left( (\theta-L)
-\int_0^L m(y,0) \dd y  \right)\eul^{\frac{L-\theta}{2}} + O\left(\left(k-\tfrac{\ii}{2}\right)^2\right);\\
a^*(k)&= \eul^{\frac{\theta-L}{2}}+\ii \left(k-\tfrac{\ii}{2}\right)\left( -(\theta-L)+\int_0^L m(y,0) \dd y  \right)\eul^{\frac{\theta-L}{2}} + O\left(\left(k-\tfrac{\ii}{2}\right)^2\right);\\
b(k)&= \ii\eul^{\frac{L-\theta}{2}}\int_0^L m(y,0)\eul^{-y} \dd y\left(k-\tfrac{\ii}{2}\right)+O\left(\left(k-\tfrac{\ii}{2}\right)^2\right);\\
b^*(k)&= -\ii\eul^{\frac{\theta-L}{2}}\int_0^L m(y,0)\eul^{y} \dd y\left(k-\tfrac{\ii}{2}\right)+O\left(\left(k-\tfrac{\ii}{2}\right)^2\right).   
\end{align}

\subsubsection{Global relations}

Since the initial and boundary values $u(x,0)$, $u(0,t)$,
$u_x(0,t)$, $u_{xx}(0,t)$, $u(L,t)$,
$u_x(L,t)$, and $u_{xx}(L,t)$ cannot be prescribed as independent 
data for a well-posed problem for the CH equation on the interval, it follows 
that if these values correspond to some solution $u(x,t)$ of the CH equation in 
$x\in [0,L]$, $t\in [0,T]$, then the spectral functions
$a(k)$, $b(k)$, $A(k)$, $B(k)$, $A_1(k)$, $B_1(k)$ associated with 
them must be interrelated. 

Indeed, expressing $\hat \Phi_{4}(0,T,k)$ in terms of  $\hat \Phi_{1}(0,T,k)$ using \eqref{s-32}--\eqref{s-43} and taking into account that 
$\hat \Phi_{1}(0,T,k)=I$ we get
\begin{equation}\label{rel_}
    \eul^{ikp(0,T,k)\hat\sigma_3}\hat\Phi_{ 4}(0,T,k)=
    S^{-1}(k)s(k)e^{ik\theta\sigma_3}S_1(k)e^{-ik\theta\sigma_3}.
\end{equation}
Now, taking into account that $\hat\Phi_{ 4}(0,T,k)$ can be 
estimated from 
\begin{equation*}
\hat\Phi_{4}(0,T,k)=I-
\int_{0}^{L}
	\eul^{ik\int_0^y\sqrt{m(\xi,T)+1}\dd \xi\hat\sigma_3}
    ( U\hat\Phi_{4})(y,T,k) \dd y,
\end{equation*}
we deduce, in particular, that 
$\hat\Phi_{4}^{(12)} = O(1/k)$ as $k\to\infty$, $\Im k\ge 0$,
where $\hat\Phi_{4}^{(12)}$ is the 
 $(12)$ entry of $\hat\Phi_{4}(0,T,k)$.
Then, considering the $(12)$ entry of \eqref{rel_}, it follows  that 
\begin{equation}\label{gl_rel}
    B_1(k)(A(k)\overline{a(\bar k)}-B(k)\overline{b(\bar k)})\eul^{2ik\theta}+A_1(k)(A(k)b(k)-B(k)a(k))=\eul^{-2ik\int_0^T u(0,\zeta)\sqrt{m(0,\zeta)+1}\dd\zeta}\ord\left(\frac{1}{k}\right) 
\end{equation}
(notice that $A$, $B$, $A_1$, and $B_1$ are functions of $T$ as well:
$A=A(k,T)$, etc.)

Similarly, from the relations amongst $\Phi_{0j}$ we deduce that 

\begin{equation}\label{rel_i2}
    \eul^{ik\frac{T}{2\lambda}\sigma_3}\hat\Phi_{0 4}(0,T,k)\eul^{-ik\frac{T}{2\lambda}\sigma_3}=\tilde S^{-1}(k)\tilde s(k)\eul^{\ii k L\sigma_3}\tilde S_1(k)\eul^{-\ii k L\sigma_3}.
\end{equation}
Here, $\hat\Phi_{ 04}(0,T,k)$ can be 
estimated from 
\begin{equation*}
\hat\Phi_{04}(0,T,k)=I-
\int_{0}^{L}
	\eul^{-ik\xi\hat\sigma_3}
    ( U\hat\Phi_{04})(\xi,T,k) \dd \xi,
\end{equation*}
which gives that $\hat\Phi_{ 04}{(12)}(0,T,k) = O(k-\frac{i}{2})$ as $k\to\frac{\ii}{2}$
and thus

\begin{equation}\label{gl_rel_til}
    \tilde B_1(k)(\tilde A(k)\overline{\tilde a(\bar k)}-\tilde B(k)\overline{\tilde b(\bar k)})\eul^{2ikL}+\tilde A_1(k)(\tilde A(k)\tilde b(k)-\tilde B(k)\tilde a(k))
    =e^{-\frac{T}{\lambda}}\ord\left(k-\frac{\ii}{2}\right)
\end{equation}
as $ k\to\frac{\ii}{2}$, $|k|>\frac{1}{2}$.

We call the (asymptotic) relations \eqref{gl_rel} and \eqref{gl_rel_til} \emph{Global Relations}
(GRs).
A remarkable fact is the the Global Relations are not just necessary conditions 
on the spectral functions associated with the boundary values of a solution 
of a nonlinear equation in question: they \emph{characterize} these boundary values.
Indeed, one can prove (see, e.g., \cites{BFS06, fokas2004nonlinear, FIS}) that given the boundary values
such that the associated spectral functions satisfy the corresponding global relation,
there is a solution of (in general, \emph{overdetermined}!) initial  boundary value problem,
whose boundary values are as prescribed. In what follows, 
we will use the ideas of proof of this remarkable fact 
(based on appropriate  deformations of the associated Riemann--Hilbert problems)
when we show that 
our construction indeed gives a periodic solution of  the CH equation.

\subsubsection{Pre-Riemann-Hilbert problem  for the problem on an interval: jump conditions}
Similarly to the whole line problem,
it is possible to combine columns of the dedicated solutions of  the Lax pair 
equations in matrices, different in different domains of the 
complex plane of the spectral parameter $k$, such that (i) 
in each domain, corresponding matrix-valued function is bounded and continuous 
up to the boundary of the domain (including unbounded domains), (ii)
their values at the boundaries of the domains are related to each other in a multiplicative manner, 
by matrices that can be determined by the initial and boundary values of the 
solution of the CH equation,  and (iii) the determinant of all matrices equals $1$
for all $x$, $t$, and $k$.

In order to achieve this, still assuming that $u(0,t)\le 0$, introduce 
\[
d(k):=a(k)A^*(k)-b(k)B^*(k), \quad 
d_1(k)=a(k)A_1(k)+b^*(k)B_1(k)\eul^{2\ii k \theta},
\]
\[
\tilde d(k):=\tilde a(k)\tilde A^*(k)-\tilde b(k)\tilde B^*(k), \quad 
\tilde d_1(k):=\tilde a(k)\tilde A_1(k)+\tilde b^*(k)
\tilde B_1(k)\eul^{2\ii kL},
\]
and consider the following matrix-valued function defined in the domains 
in Figure \ref{fig:contour}:
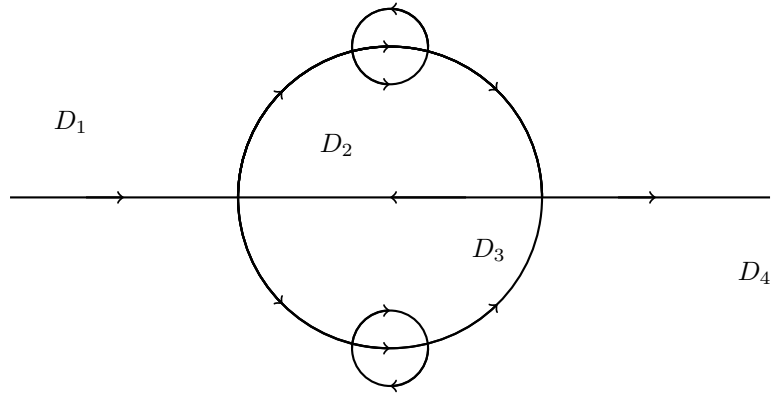
\begin{figure}
    \centering
    \begin{tikzpicture}
    \draw[thick] (0,0) circle (2);
 \draw[thick] (-5,0) -- (5,0);
 
    \draw[->, thick] (2,0) arc[start angle=0, end angle=270, radius=2];
    \draw[->, thick] (2,0) arc[start angle=0, end angle=225, radius=2];
    \draw[->, thick] (2,0) arc[start angle=0, end angle=315, radius=2];
        
    \draw[->, thick] (-2,0) arc[start angle=180, end angle=90, radius=2];
    \draw[->, thick] (-2,0) arc[start angle=180, end angle=45, radius=2];
        \draw[->, thick] (-2,0) arc[start angle=180, end angle=135, radius=2];
    
 \draw[->, thick] (3,0) -- (3.5,0);
  \draw[->, thick] (-4,0) -- (-3.5,0);

    \draw[->, thick] (1,0) -- (0,0);

 \draw[->, thick] (0.5,2) arc[start angle=0, end angle=270, radius=0.5];
 \draw[->, thick] (0.5,2) arc[start angle=0, end angle=90, radius=0.5];

  \draw[->, thick] (-0.5,-2) arc[start angle=180, end angle=90, radius=0.5];
 \draw[->, thick] (0.5,-2) arc[start angle=360, end angle=270, radius=0.5];

    \draw[thick] (0,2) circle (0.5);

        \draw[thick] (0,-2) circle (0.5);

        \node at (-4.2, 1) {$D_1$};

         \node at (4.8, -1) {$D_4$};

           \node at (-0.7, 0.7) {$D_2$}; 

         \node at (1.3,-0.7) {$D_3$}; 
\end{tikzpicture}
    \caption{Domains and the separating contour $\Sigma$}
    \label{fig:contour}
\end{figure}

\begin{equation}\label{N}
N(x,t,k)=\begin{cases}
\left( \frac{\hat\Phi_{ 2}^{(1)}(k)}{a(k)},\hat\Phi_{ 4}^{(2)}(k)\frac{a(k)}{d_1(k)}\right),\quad k\in\{ \Im k > 0, |k|>\frac{1}{2}, |k-\frac{\ii}{2}|>\epsilon\},\\

Q(x,t)\left( \frac{\hat\Phi_{0 2}^{(1)}(k)}{\tilde a(k)},\hat\Phi_{0 4}^{(2)}(k)\frac{\tilde a(k)}{\tilde d_1(k)}\right)\eul^{-\ii k(p_0(x,t,k)-p(x,t,k))\sigma_3},\quad k\in\{|k|>\frac{1}{2}, |k-\frac{\ii}{2}|<\epsilon\},\\

Q(x,t)\left( \frac{\hat\Phi_{0 1}^{(1)}(k)}{\tilde d(k)},\hat\Phi_{0 3}^{(2)}(k)\right)\eul^{-\ii k(p_0(x,t,k)-p(x,t,k))\sigma_3},\quad k\in\{  |k|<\frac{1}{2}, |k-\frac{\ii}{2}|<\epsilon\},\\

\left( \frac{\hat\Phi_{ 1}^{(1)}(k)}{d(k)},\hat\Phi_{ 3}^{(2)}(k)\right),\quad k\in\{ \Im k > 0, |k|<\frac{1}{2}, |k-\frac{\ii}{2}|>\epsilon\},\\

\left( \hat\Phi_{ 3}^{(1)}(k),\frac{\hat\Phi_{ 1}^{(2)}(k)}{d^*( k)}\right),\quad k\in\{ \Im k < 0, |k|<\frac{1}{2}, |k+\frac{\ii}{2}|>\epsilon\},\\

Q(x,t)\left( \hat\Phi_{0 3}^{(1)}(k),\frac{\hat\Phi_{0 1}^{(2)}(k)}{\tilde d^*(k)}\right)\eul^{-\ii k(p_0(x,t,k)-p(x,t,k))\sigma_3},\quad k\in\{  |k|<\frac{1}{2}, |k+\frac{\ii}{2}|<\epsilon\},\\

Q(x,t)\left( \frac{\tilde a^*(k)}{\tilde d^*_1(k)}\hat\Phi_{0 4}^{(1)}(k),\frac{\hat\Phi_{0 2}^{(2)}(k)}{\tilde a^*(k)}\right)\eul^{-\ii k(p_0(x,t,k)-p(x,t,k))\sigma_3},\quad k\in\{  |k|>\frac{1}{2}, |k+\frac{\ii}{2}|<\epsilon\},\\

\left(\frac{a^*( k)}{d^*_1( k)} \hat\Phi_{ 4}^{(1)}(k),\frac{\hat\Phi_{ 2}^{(2)}(k)}{a^*( k)}\right),\quad k\in\{ \Im k < 0, |k|>\frac{1}{2}, |k+\frac{\ii}{2}|>\epsilon\}.
\end{cases}
\end{equation} 
Then, by direct calculations, the boundary values of $N$ are related on $\Sigma$
by 
\begin{equation}\label{jump-xt}
N_-(x,t,k)=N_+(x,t,k)J(x,t,k), \quad k\in \Sigma,
\end{equation}
where $J(x,t,k) = e^{-ikp(x,t,k)\sigma_3}J_0(k)e^{ikp(x,t,k)}$ with 
\begin{equation}\label{j0-xt}
    J_0(k)=
    \begin{cases}
\begin{pmatrix}
    \ee^{\ii k(L-\theta)}& \Gamma_2(k)\\
   0 &\ee^{-\ii k(L-\theta)}
\end{pmatrix}, & |k|>\frac{1}{2}, |k-\frac{\ii}{2}|=\epsilon,\\
\begin{pmatrix}
          1& -\tilde\Gamma_1(k)\\
         0&1   
\end{pmatrix}
\begin{pmatrix}
          1& 0\\
         \tilde\Gamma(k)&1   
\end{pmatrix}, &   |k|=\frac{1}{2}, |k-\frac{\ii}{2}|<\epsilon,\\
\begin{pmatrix}
 \ee^{\ii k(L-\theta)}&0\\
  \Gamma_3(k)&\ee^{-\ii k(L-\theta)}
\end{pmatrix}, & |k|<\frac{1}{2}, |k-\frac{\ii}{2}|=\epsilon,\\
\begin{pmatrix}
          1& -\Gamma_1(k)\\
         0&1   
\end{pmatrix}
\begin{pmatrix}
          1& 0\\
         \Gamma(k)&1   
\end{pmatrix},& \Im k > 0, |k|=\frac{1}{2}, |k-\frac{\ii}{2}|>\epsilon,\\
\begin{pmatrix}
          1& -\Gamma_1(k)\\
         0&1   
\end{pmatrix}
\begin{pmatrix}
          1& -r^*(k)\\
         r(k)&1  - |r(k)|^2 
\end{pmatrix}
\begin{pmatrix}
          1& 0\\
         \Gamma^*_1(k)&1   
\end{pmatrix}, &  \Im k = 0, |k|>\frac{1}{2},\\
\begin{pmatrix}
          1& -\Gamma^*( k)\\
         0&1   
\end{pmatrix}
\begin{pmatrix}
          1 - |r(k)|^2& r^*(k)\\
         -r(k)&1  
\end{pmatrix}
\begin{pmatrix}
          1& 0\\
         \Gamma(k)&1   
\end{pmatrix}, &  \Im k = 0, |k|<\frac{1}{2}
\end{cases}
\end{equation}
on parts of $\Sigma$ in the upper half-plane 
whereas 
$J_0(k)=
\left(\begin{smallmatrix}
    0 & 1 \\ 1 & 0
\end{smallmatrix}\right)J_0^*(k)\left(\begin{smallmatrix}
    0 & 1 \\ 1 & 0
\end{smallmatrix}\right)$ for $k\in{\mathbb C}_-$. Here the following notations are used:
$
r(k)=\frac{b^*(k)}{a(k)}
$,
\begin{equation}\label{Ga}
\Gamma(k)=\frac{B^*(k)}{a(k)(a(k)A^*(k)-b(k)B^*(k))}
=\frac{R^*(k)}{ a(k)( a(k)
- b(k) R^*(k))},
\end{equation}
\begin{equation}\label{Ga1}
 \Gamma_1(k)=\frac{B_1(k)a(k)\eul^{2\ii k \theta}}{a(k)A_1(k)+b^*(k)B_1(k)\eul^{2\ii k \theta}}
 =\frac{ R_1(k) a(k)\ee^{2\ii k \theta}}{ a(k)+
 b^*(k)R_1(k)\ee^{2\ii k \theta}},       
\end{equation}
\begin{equation}\label{til-Ga}
    \tilde\Gamma(k)=\frac{\tilde B^*(k)}{\tilde a(k)(\tilde a(k)\tilde A^*(k)-\tilde b(k)\tilde B^*(k)}
=\frac{\tilde R^*(k)}{\tilde a(k)(\tilde a(k)
-\tilde b(k)\tilde R^*(k))},
\end{equation}
\begin{equation}\label{til-Ga1}
    \tilde \Gamma_1(k)=\frac{\tilde B_1(k)\tilde a(k)\eul^{2\ii k L}}{\tilde a(k)\tilde A_1(k)+\tilde b^*(k)\tilde B_1(k)\eul^{2\ii k L}}
=\frac{\tilde R_1(k)\tilde a(k)\ee^{2\ii k L}}{\tilde a(k)+\tilde b^*(k)\tilde R_1(k)\ee^{2\ii k L}}, 
\end{equation}
\begin{equation}\label{Ga2}
    \Gamma_2 = \frac{a\tilde a}{d_1\tilde d_1}(\tilde A_1 B_1-\tilde B_1 A_1)
\eul^{\ii k(L+\theta)}
= a\tilde a\eul^{\ii k (L+\theta)}\frac{R_1-\tilde R_1}{(a+b^*R_1\eul^{2\ii k \theta})
(\tilde a+\tilde b^*\tilde R_1\eul^{2\ii k L})},
\end{equation}
\begin{equation}\label{Ga3}
    \Gamma_3 =  \frac{\tilde A^* B^*-\tilde B^* A^*}{d\tilde d}
= \frac{R^*-\tilde R^*}{(a-bR^*)( \tilde a- \tilde b \tilde R^*)},
\end{equation}
where 
\begin{equation}\label{RRR}
    R(k):=\frac{B(k)}{A(k)}, \quad R_1(k):=\frac{B_1(k)}{A_1(k)}, \quad\tilde R(k)=\frac{\tilde B(k)}{\tilde A(k)}, \quad 
\tilde R_1(k)=\frac{\tilde B_1(k)}{\tilde A_1(k)}.
\end{equation}

Notice that the spectral functions associated with the boundary values 
enter the jump matrix $J_0$ trough the respective ratios only ($R$, $R_1$,
$\tilde R$, and $\tilde R_1$).

\subsubsection{Pre-Riemann-Hilbert problem  for the problem on an interval: 
symmetries}

The symmetries \eqref{sym_phi}, \eqref{symm_a}, \eqref{symm_A} and \eqref{symm_A_1} imply that   $N$ satisfies the  symmetries

\begin{equation}
    \label{sym_N}
    \overline{N(\bar k)}=N(-k)=\begin{pmatrix}
        0&1\\1&0
    \end{pmatrix}N(k)\begin{pmatrix}
        0&1\\1&0
    \end{pmatrix}.
\end{equation}

\subsubsection{Pre-Riemann-Hilbert problem  for the problem on an interval: 
behavior at $k=\infty$}

As $k \to \infty$, we have
\[
N(k) = I + O\!\left(\tfrac{1}{k}\right).
\]
Indeed,
from \eqref{s_inf} and \eqref{S_inf} it follows that
\[
d_1(k) = 1 + O\!\left(\tfrac{1}{k}\right) + O\!\left(\tfrac{e^{2 i k \theta}}{k}\right).
\]
Together with the asymptotic properties of $\hat \Phi_{ j}$, this yields the desired expansion for $N(k)$.

\subsubsection{Pre-Riemann-Hilbert problem  for the problem on an interval: 
residue  conditions}

Let $d_1(k)$ have simple zeros at $\nu_j \in D_1$, $d(k)$ simple zeros at $\lambda_j \in D_2$, and $a(k)$ simple zeros at $\kappa_j \in (0,\tfrac{i}{2})$ (\cite{Con97}).  
Then, by symmetry, $d^*_1( k)$ has simple zeros at $\bar\nu_j \in D_4$, $d^*( k)$ has simple zeros at $\bar\lambda_j \in D_3$, and $a^*( k)$ has simple zeros at $\bar\kappa_j \in (0,-\tfrac{i}{2})$.  
To fix the ideas, we assume that $a(k)$ and $d(k)$ do not have any common zeros. 

Under these assumptions, $N$ satisfies the residue conditions

\begin{align}\label{res_N-1}
\Res_{\nu_j}N^{(2)}&=N^{(1)}(\nu_j) \Res_{\nu_j} \Gamma_1 \eul^{-2\ii \nu_j p(\nu_j)}, ~ \Res_{\nu_j}\Gamma_1=\frac{B_1(\nu_j)  \eul^{2 \ii \nu_j \theta} a(\nu_j)}{\dot d(\nu_j)}, ~ \nu_j\in D_1, \\\label{res_N-2}
 \Res _{\lambda_j}N^{(1)}&=N^{(2)}(\lambda_j)  \Res _{\lambda_j} \Gamma \eul^{2\ii \lambda_j p(\lambda_j)}, ~
 \Res _{\lambda_j}\Gamma=\frac{B(\lambda_j)  }{a(\lambda_j)\dot d(\lambda_j)}, ~ \lambda_j\in D_2, \\\label{res_N-3}
 \Res _{\bar\nu_j}N^{(1)}&=N^{(2)}(\bar\nu_j)  \Res _{\bar\nu_j} \Gamma_1^* \eul^{2\ii \bar\nu_j p(\bar\nu_j)},~ \Res _{\bar\nu_j}\Gamma_1^*=\frac{B_1^*(\bar\nu_j)  \eul^{2 \ii \bar\nu_j \theta} a^*(\bar\nu_j)}{\dot d^*(\bar\nu_j)},~ \bar\nu_j\in D_4,\\\label{res_N-4}
 \Res _{\bar\lambda_j}N^{(2)}&=N^{(1)}(\bar\lambda_j)  \Res _{\bar\lambda_j} \Gamma^* \eul^{-2\ii \bar\lambda_j p(\bar\lambda_j)},~ \Res _{\bar\lambda_j}\Gamma^*=\frac{B^*(\bar\lambda_j)  }{a^*(\bar\lambda_j)\dot d^*(\bar\lambda_j)}
,~\bar\lambda_j\in D_3. 
\end{align}

Indeed, \eqref{s-32}--\eqref{s-43} imply
\begin{equation}\label{N_res_ch}
     \left( \frac{\hat\Phi_{ 1}^{(1)}(k)}{d(k)},\hat\Phi_{ 3}^{(2)}(k)\right)\eul^{-\ii k p \sigma_3}\begin{pmatrix}
         1& 0\\
         -\Gamma(k)&1
     \end{pmatrix}\eul^{\ii k p \sigma_3}= \left( \frac{\hat\Phi_{ 2}^{(1)}(k)}{a(k)},\hat\Phi_{ 3}^{(2)}(k)\right)\eul^{-\ii k p \sigma_3}\begin{pmatrix}
         1& -\Gamma_1(k)\\
         0&1     \end{pmatrix}\eul^{\ii k p \sigma_3}.
\end{equation}
First, consider $D_1$. Substituting  $\Gamma_1$ defined in \eqref{Ga1} into the second column of \eqref{N_res_ch} yields  
\[
d_1(k) \hat\Phi_{ 3}^{(2)} (x,t,k)
   = -B_1(k) e^{2 i k \theta} e^{-2 i k p(x,t,k)}\hat\Phi_{ 2}^{(1)} + a(k)\hat\Phi_{ 4}^{(2)}(x,t,k) .
\]
At $k=\nu_j$ with $d_1(\nu_j)=0$, this reduces to
\begin{equation}\label{f_res_ch_1}
  \hat\Phi_{ 4}^{(2)}(\nu_j) 
   = \frac{\hat\Phi_{ 2}^{(1)}(\nu_j)}{a(\nu_j)} B_1(\nu_j) e^{2 i \nu_j \theta} e^{-2 i \nu_j p(\nu_j)}.  
\end{equation}
Combining \eqref{f_res_ch_1} with $\Res_{\nu_j} N^{(2)} 
   = \hat\Phi_{4}^{(2)}(\nu_j)\frac{a(\nu_j)}{\dot d_1(\nu_j)}$, \eqref{res_N-1} follows.

Similarly, consider $D_2$. Substituting  $\Gamma$ defined in \eqref{Ga} into the second column of \eqref{N_res_ch} yields  
\[
d(k)\frac{\hat\Phi_{ 2}^{(1)}(x,t,k)}{a(k)}=\hat\Phi_{ 1}^{(1)}(x,t,k)-\frac{B(k)}{a(k)}\eul^{2\ii k p }\hat\Phi_{ 3}^{(2)}(x,t,k) .
\]
At $k=\lambda_j$ with $d(\lambda_j)=0$, we obtain
\begin{equation}\label{f_res_ch_2}
  \hat\Phi_{ 1}^{(1)}(\lambda_j)=\hat\Phi_{ 3}^{(2)}(\lambda_j)\frac{B(\lambda_j)}{a(\lambda_j)}\eul^{2\ii \lambda_j p(\lambda_j)}.  
\end{equation}
Combining \eqref{f_res_ch_2} with $\Res_{\lambda_j} N^{(1)} 
   = \frac{\hat\Phi_{ 1}^{(1)}(\lambda_j)}{\dot d(\lambda_j)}$, \eqref{res_N-2} follows.

The residue conditions in the lower half-plane can be obtained analogously. 

\subsubsection{Pre-Riemann-Hilbert problem  for the problem on an interval: 
behavior at $k=0$}

As $k \to 0$, relations\eqref{ab_at_0} yield the expansions
\begin{align}
    \label{d_at_0}
  & d(k)=\frac{\ii}{k}d_{-1}+d_0+O(k), \qquad d_{-1}\in \mathbb{R}, ~ d_{0}\in \mathbb{R},\\\label{D_at_0}
& (a^*B^*-b^*A^*)(k)=- \frac{\ii}{k}d_{-1}+D_0+O(k), \qquad d_{-1}\in \mathbb{R}, ~ D_{0}\in \mathbb{R}.
\end{align}
Moreover, the coefficient $d_{-1}$ is given by $d_{-1}=-\rho( A_0+B_0)+\tilde\rho (a_0+b_0)$.

If $d_{-1}\neq 0$ (generic case), then \eqref{Phi_at_0} together with
\eqref{d_at_0} implies that 
\begin{equation*}
  N(x,t,k)=  \frac{\ii}{k} c_3(x,t)\begin{pmatrix}
        0&1\\0&-1
    \end{pmatrix}+O(1), \quad k\to 0, \quad k\in\mathbb{C}_+.
\end{equation*}

If $d_{-1}= 0$ (non-generic case), then from
\eqref{s-32}, \eqref{s-12}, \eqref{d_at_0}, \eqref{D_at_0} we obtain \[c_1(x,t)=\delta c_3(x,t),\quad\delta = d_0+D_0,\]
and combining this fact with
\eqref{Phi_at_0} and
\eqref{d_at_0} gives
\begin{equation*}
  N(x,t,k)=  \frac{\ii c_3(x,t)}{k} \begin{pmatrix}
        \frac{\delta}{d_0}&1\\- \frac{\delta}{d_0}&-1
    \end{pmatrix}+O(1), \quad k\to 0, \quad k\in\mathbb{C}_+.
\end{equation*}

\subsubsection{Pre-Riemann-Hilbert problem  for the problem on an interval: 
behavior at $k=\frac{\ii}{2}$}

Expanding $\eul^{-\ii k(p_0(x,t,k)-p(x,t,k))}$ as $k\to \frac{\ii}{2}$, we obtain
\begin{equation}\label{e_i}
   \eul^{-\ii k(p_0(x,t,k)-p(x,t,k))}= \eul^{-\frac{1}{2}(\int_0^x (\sqrt{m(\xi,t)+1}-1)\dd\xi-\int_0^t u(0,\zeta)\sqrt{m(0,\zeta)+1}\dd\zeta)}+O\left(k-\tfrac{\ii}{2}\right).
\end{equation}

For $\{|k|>\frac{1}{2}\}$, the properties of 
 $\hat\Phi_{0j}$ together with the definitions of $\tilde s$ and $\tilde S$ imply 
\begin{equation}\label{d1_i2}
     \tilde d_1=1+O\left(k-\tfrac{\ii}{2}\right).
\end{equation}
Combining \eqref{e_i}, \eqref{d1_i2} and the behaviour of $\hat\Phi_{0j}$ at $\frac{\ii}{2}$, we obtain
\begin{equation}\label{Psi_i2_1}
N(x,t,k)=Q(x,t)\eul^{-\frac{1}{2}(\int_0^x \sqrt{m(\xi,t)+1}\dd\xi-\int_0^t u(0,\zeta)\sqrt{m(0,\zeta)+1}\dd\zeta-x)\sigma_3}+O\left(k-\tfrac{\ii}{2}\right).
\end{equation}

For $\{|k|<\frac{1}{2}\}$, the properties of 
 $\hat\Phi_{0j}$ and the definitions of $\tilde s$ and $\tilde S$ yield
\begin{equation}\label{d_i2}
     \tilde d=1+O\left(k-\tfrac{\ii}{2}\right).
\end{equation}
Using \eqref{e_i}, \eqref{d1_i2} and the behaviour of $\hat\Phi_{0j}$ at $\frac{\ii}{2}$ we obtain the same expansion as in \eqref{Psi_i2_1}.

\section{The CH equation,  associated Lax pair, and  RH formalism in $y,t$ variables}\label{sec:CH-y}
\subsection{The CH equation and the  associated Lax pair in $y,t$ variables}
In order to have the data for the RH problem to 
be  explicitly determined by the initial data only,  
it is convenient to introduce the new space variable $y=y(x,t)$ in such a way that 

\begin{subequations}\label{y-xt}
\begin{align}
\label{y_x} y_x & = \sqrt{m+1},\\
   \label{y_t} y_t &=-u \sqrt{m+1}
\end{align}
    \end{subequations}
    (equations \eqref{y_x} and \eqref{y_t} are compatible due to \eqref{CH-m1}).
Particularly, for the problem on the line, this can be done by setting
\begin{equation}
    y(x,t)=\int_{-\infty}^x \sqrt{m(\xi,t)+1}\dd\xi
\end{equation}
whereas for the problem on the interval $x\in [0,L]$ we set
\begin{equation}\label{y}
    y(x,t)=\int_0^x \sqrt{m(\xi,t)+1}\dd\xi-\int_0^t u(0,\zeta)\sqrt{m(0,\zeta)+1}\dd\zeta.
\end{equation}
Then, differentiating the identity $x(y(x,t),t)=x$, we obtain the equations characterizing the 
inverse mapping $x=x(y,t)$:
\begin{subequations}\label{x-yt}
\begin{align}
\label{x_y} x_y & = \frac{1}{\sqrt{\hat m+1}},\\
   \label{x_t} x_t &=\hat u,
\end{align}
    \end{subequations}
where $\hat u (y,t) = u(x(y,t),t)$ and  $\hat m (y,t) = m(x(y,t),t)$.

Notice that for the problem on $[0,L]$ with the periodic conditions, one has
\begin{equation}\label{y_0L}
    y(L,t)= y(0,t)+\theta\text{ for all } t,
\end{equation}
where
$\theta= \int_0^L \sqrt{m(\xi,0)+1}\dd\xi $.

\begin{proposition} (CH equation in the $(y,t)$ variables) Let $u(x,t)$ and $m(x,t)$ ($m(x,t)+1>0$) satisfy \eqref{CH-m1} and let $y(x,t)$ be 
such that \eqref{y-xt} hold.
Then the CH equation \eqref{CH-m1} in the $(y,t)$ variables reads as the following system of equations
for  $\hat m(y,t):=m(x(y,t),t)$, $\hat u(y,t):=u(x(y,t),t)$, and 
$\hat v(y,t):=u_x(x(y,t),t)$:

\begin{subequations}\label{CH_in_y}
\begin{align}\label{CH_in_y-1}
(\sqrt{\hat m+1})_t&=-\hat v\sqrt{\hat m+1},\\\label{CH_in_y-2}
\hat v&=\hat u_y\sqrt{\hat m+1},\\\label{CH_in_y-3}
\hat m&=\hat u-\hat v_y\sqrt{\hat m+1}.
\end{align}
\end{subequations}
    
\end{proposition}

\begin{proof}
As we discussed above, \eqref{y-xt} implies \eqref{x-yt}.
 Substituting $\left(\sqrt{m+1}\right)_t=-\left(u\sqrt{m+1}\right)_x$ from \eqref{CH-m1} and $x_t=\hat u$ from \eqref{x_t} into the equality
\[
\left(\sqrt{\hat m(y,t)+1}\right)_t=\left. \left(\left(\sqrt{ m(x,t)+1}\right)_x x_t(y,t)+\left(\sqrt{ m(x,t)+1}\right)_t\right)\right|_{x=x(y,t)}
\]
we get
\[
\left(\sqrt{\hat m(y,t)+1}\right)_t=-\hat u_x(x(y,t),t)\sqrt{m(x(y,t),t)+1}=-\hat v(y,t)\sqrt{\hat m(y,t)+1}
\]
and thus \eqref{CH_in_y-1} follows.

Now, substituting $x_y=\frac{1}{\sqrt{\hat m+1}}$ from \eqref{x_y} into the equality
$\hat u_y(y,t)=\left. \left(u_x(x,t)x_y(y,t)\right)\right|_{x=x(y,t)}$
we get \eqref{CH_in_y-2}.

As for \eqref{CH_in_y-3}, it follows from substituting \eqref{x_y} into
$
\hat v_y(y,t)=\left. \left(u_{xx}(x,t) x_y(y,t)\right)\right|_{x=x(y,t)}$.
\end{proof}

Similarly, the reverse  change of variable $(y,t)\mapsto(x,t)$ 
reduces \eqref{CH_in_y} to \eqref{CH-m1} with 
$m(x,t):=\hat m(y(x,t),t)$ and $u(x,y):=\hat u(y(x,t),t)$. 
\begin{proposition}
    Let $\hat u(y,t)$, $\hat v(y,t)$,  and $\hat m(y,t)$ with $\hat m(y,t)+1>0$ satisfy 
    \eqref{CH_in_y} and let $x(y,t)$ be 
such that \eqref{x-yt} holds. Define $u(x,t):=\hat u(y(x,t),t)$ and $m(x,t):=\hat m(y(x,t),t)$.
Then \eqref{CH-m1} holds for $u(x,t)$ and $m(x,t)$.
\end{proposition}
\begin{proof}
Similarly to above, \eqref{x-yt} implies \eqref{y-xt}.
Then 
\[
u_x=\hat v_y y_x = \frac{\hat v}{\sqrt{m+1}}\sqrt{m+1} = \hat v,
\]
where we have used \eqref{CH_in_y-2} and \eqref{y_x}.
Further, differentiating this w.r.t. $x$ and using \eqref{CH_in_y-3} we get
\[
u_{xx}=\hat v_y y_x= \hat v_y \sqrt{m+1} = u - m
\]
and thus $m(x,t)=u(x,t)-u_{xx}(x,t)$.

Finally, using \eqref{CH_in_y-3} and \eqref{y_t} we have
\begin{align*}
\left(\sqrt{m+1}\right)_t & = \left(\sqrt{\hat m+1}\right)_t + \left(\sqrt{\hat m+1}\right)_y 
y_t = -\hat v \sqrt{m+1} + \left(\sqrt{m+1}\right)_x x_y (-u\sqrt{m+1}) \\
& = -u_x \sqrt{m+1} -u \left(\sqrt{m+1}\right)_x = - \left(u\sqrt{m+1}\right)_x.
\end{align*}

\end{proof}

Now let us show that a solution of the periodic problem for the CH equation 
in the $x,t$ variables \eqref{CH-m1} gives rise to a periodic solution 
of the CH equation 
in the $y,t$ variables \eqref{CH_in_y} and vice versa. 
\begin{proposition}\label{prop:periodic1}
    Let  $u(x,t)$ and $m(x,t)$ (satisfying $m(x,t)+1>0$) be the components of a periodic solution of  \eqref{CH-m1} with the period $L$, 
        and let $y(x,t)$ be defined by \eqref{y}. Then 
    $\hat m(y,t):=m(x(y,t),t)$, $\hat u(y,t):=u(x(y,t),t)$, and 
$\hat v(y,t):=u_x(x(y,t),t)$ are the components of a periodic solution of 
\eqref{CH_in_y}, with the period
\begin{equation}\label{theta}
    \theta:= \int_0^L \sqrt{m(\xi,0)+1} d\xi.
\end{equation}
    
\end{proposition}
\begin{proof}
    Introducing $x(y,t)$ as the inverse to $y(x,t)$ and setting $y=0$ and $y=\theta$
    in the l.h.s. of \eqref{y} we have
    \begin{align*}
        0 & = \int_0^{x(0,t)} \sqrt{m(\xi,t)+1} d\xi 
            -\int_0^t u(0,\zeta)\sqrt{m(0,\zeta)+1}d\zeta, \\
            \theta & = \int_0^{x(\theta,t)} \sqrt{m(\xi,t)+1} d\xi 
            -\int_0^t u(0,\zeta)\sqrt{m(0,\zeta)+1}d\zeta,
    \end{align*}
    from which we deduce that 
    \[
    \int_{x(0,t)}^{x(\theta,t)} \sqrt{m(\xi,t)+1} d\xi =\theta = \int_0^L \sqrt{m(\xi,0)+1} d\xi
    \]
    for all $t$. Since $m(x,t)$ is periodic and $\sqrt{m(x,t)+1}>0$, it follows that 
    \[
    x(\theta,t) - x(0,t) = L \ \ \text{for all} \ \ t.
    \]
    Consequently,
    \[
    \hat u(\theta,t) =  u(x(\theta,t),t) = u(x(0,t),t) = \hat u(0,t)
    \]
    and thus $\hat u(y,t)$ is periodic as function of $y$, for all fixed $t$, with period $\theta$. Similarly for $\hat v = u_x$ and $\hat m$.
\end{proof}

\begin{proposition}\label{prop:periodic2}
    Let  $\hat u(y,t)$, $\hat v(y,t)$,  and $\hat m(y,t)$ (satisfying $\hat m(y,t)+1>0$) be the components of a periodic solution of  \eqref{CH_in_y} with the period $\theta$, 
        and let $x(y,t)$ be defined by 
\begin{equation}
    \label{x}
    x(y,t)=\int_0^y \frac{1}{\sqrt{\hat m(\xi,t)+1}}\dd\xi+
    \int_0^t \hat u(0,\zeta)\dd\zeta.
\end{equation}
Then 
    $m(x,t):=\hat m(y(x,t),t)$ and $u(x,t):=\hat u(y(x,t),t)$
    are the components of a periodic solution of 
\eqref{CH-m1}, with the period
\begin{equation}\label{L}
    L= \int_0^\theta \frac{1}{\sqrt{\hat m(\xi,0)+1}}\dd\xi.
\end{equation}
    Moreover, \eqref{theta} holds true.
\end{proposition}
\begin{proof}
    Introducing $y(x,t)$ as the inverse to $x(y,t)$ and arguing as in the proof of Proposition
    \ref{prop:periodic1} we arrive at
    \[
    \int_{y(0,t)}^{y(L,t)} \frac{1}{\sqrt{\hat m(\xi,t)+1}} d\xi =L
    = \int_0^\theta \frac{1}{\sqrt{\hat m(\xi,0)+1}}  d\xi
    \]
    for all $t$, which implies  
    \[
    y(L,t) - y(0,t) = \theta \ \ \text{for all} \ \ t.
    \]
     Consequently,
    \[
    u(L,t) =  \hat u(y(L,t),t) = \hat u(y(0,t),t) = u(0,t)
    \]
    and thus $u(x,t)$ is periodic as function of $x$, for all fixed $t$, with period $L$
    determined by \eqref{L}.  Similarly for $m$.

    Finally,
    \[
    \int_0^L \sqrt{m(\xi,0)+1} d\xi = \int_0^L \sqrt{\hat m(y(\xi,0),0)+1} d\xi=
    \int_0^\theta \sqrt{\hat m(y,0)+1} \xi_y dy = \int_0^\theta dy =\theta.
    \]
\end{proof}

Introducing $\hat p(y,t,k)$
in terms of $y(x,t)$ by 
 \begin{equation}\label{py}
    \hat p(y,t,k):=y+\frac{t}{2\lambda}= y-\frac{t}{2(k^2+\frac{1}{4})}
\end{equation}  
we observe that it is suitable in the both cases (the line and the interval).

Now, let us reformulate the original Lax pair equations in the $(y,t)$ variables.

\begin{proposition}
The Lax pair \eqref{lax} in the variables $(y,t)$ takes the form 
\begin{subequations}\label{Lax_y}
\begin{align}\label{Lax_y_y}
    &\hat \Psi_{ y}+\ii k \sigma_3 \hat \Psi=\left(\frac{1}{4}\frac{\hat m_y}{\hat m+1}\begin{pmatrix}
        0&1\\1&0
    \end{pmatrix}-\frac{1}{8\ii k}\frac{\hat m}{\hat m+1}\begin{pmatrix}
        -1&-1\\1&1
    \end{pmatrix}\right)\hat \Psi,\\
   &\hat\Psi_{ t}=\left(-\frac{\ii k}{4\lambda}\begin{pmatrix}
       \sqrt{\hat m+1}+\frac{1}{\sqrt{\hat m+1}} &  -\sqrt{\hat m+1}+\frac{1}{\sqrt{\hat m+1}}\\\label{Lax_y_t}
        \sqrt{\hat m+1}-\frac{1}{\sqrt{\hat m+1}}& -\sqrt{\hat m+1}-\frac{1}{\sqrt{\hat m+1}}
   \end{pmatrix} + \frac{1}{4\ii k}\frac{\hat u}{\sqrt{\hat m+1}}   \begin{pmatrix}
        -1&-1\\1&1
    \end{pmatrix} \right) \hat\Psi.
\end{align}
\end{subequations}
\end{proposition}

\begin{proof}
  Introducing $\hat\Psi(y,t) = \hat\Phi (x(y,t),t)$ and taking into account \eqref{x_y} and \eqref{x_t}, the Lax pair \eqref{lax} in the variables $(y,t)$ takes the form \eqref{Lax_y}.
\end{proof}

\subsection{The  RH formalism in $y,t$ variables}
In this subsection we discuss how to arrive at a (local) solution of the CH equation in the 
$y,t$ variables starting from a RH problem, parametrized by $y$ and $t$, suggested 
by the considerations in Section \ref{sec:3}.

Consider the following \textbf{RH problem parametrized by $y$ and $t$}:
find a piece-wise meromorphic ($k\in {\mathbb C}\setminus\Gamma$), $2\times 2$-matrix valued function $\hat M(y,t,k)$ satisfying the following conditions:
\begin{enumerate}[\textbullet]
\item
\emph{Jump} condition
\begin{equation}\label{jump-y}
\hat M_+(y,t,k)=\hat M_-(y,t,\lambda) \hat J(y,t,k),\qquad k\in\Gamma,
\end{equation}
where 
$\Gamma$ is a contour containing the real axis $\mathbb R$ and 
$\hat J(y,t,k)=\eul^{-ik\hat p(y,t,k)\sigma_3}J_0(k)\eul^{ik\hat p(y,t,k)\sigma_3}$ with some 
$\hat J_0(k)$ with $\det J_0(k)\equiv 1$.

\item  \emph{Normalization} condition:
\begin{equation}\label{norm-m-hat}
\hat M(y,t,k)=I+\ord\left(\frac{1}{k}\right), \quad k\to\infty.
\end{equation}

\item
\emph{Singularity} condition at $0$:
\begin{equation}\label{hatM_at_0nongen}
  \hat{ M}(y,t,k)=  \frac{\ii}{k} \hat c(y,t)\begin{pmatrix}
        \delta&1\\-\delta&-1
    \end{pmatrix}+O(1), \quad k\to 0, \quad k\in\mathbb{C}_+,
\end{equation}
\begin{equation}\label{hatM_at_0nongen_-}
   \hat M(y,t,k)=  -\frac{\ii}{k} \hat c(y,t)\begin{pmatrix}
        -1&-\delta\\1&\delta
    \end{pmatrix}+O(1), \quad k\to 0, \quad k\in\mathbb{C}_-,
\end{equation} 
where $\hat c(y,t)\in\mathbb{R}$ is not specified.

\item \emph{Structural} condition at $k=\pm\frac{\ii}{2}$:

\begin{equation}\label{hatM_at_i2}
    \hat M(y,t,k)=\frac{1}{2}\begin{pmatrix}
    ( \hat q+\hat q^{-1})\hat f&(\hat q-\hat q^{-1})\hat f^{-1}\\
   (\hat q-\hat q^{-1})f&(\hat q+\hat q^{-1})\hat f^{-1}
\end{pmatrix}+
O\left(k-\tfrac{\ii}{2}\right),\quad k\to\frac{\ii}{2},
\end{equation}
\begin{equation}\label{hatM_at_-i2}
    \hat M(y,t,k)=\frac{1}{2}\begin{pmatrix}
    ( \hat q+\hat q^{-1})\hat f^{-1}&(\hat q-\hat q^{-1})\hat f\\
   (\hat q-\hat q^{-1})f^{-1}&(\hat q+\hat q^{-1})\hat f
\end{pmatrix}+
O\left(k+\tfrac{\ii}{2}\right),\quad k\to-\frac{\ii}{2},
\end{equation}
where $\hat q=\hat q(y,t)>0$ and $\hat f=\hat f(y,t)>0$ are not specified.

\item \emph{Residue} conditions:
\begin{align}\label{res_hatM}
\Res_{\mu_j}\hat M^{(2)}(y,t,k)&= c_j \hat M^{(1)}(y,t,\mu_j) \eul^{-2\ii \mu_j \hat p(y,t,\mu_j)}, ~ \mu_j\in D_1, \\\label{res_hatM_}
\Res_{\eta_j}\hat M^{(1)}(y,t,k)&=d_j \hat M^{(2)}(y,t,\eta_j)\eul^{2\ii \eta_j \hat p(y,t,\eta_j)} , ~ \eta_j\in D_2
\end{align}
with some $\{\mu_j\}$, $\{\eta_j\}$, $\{c_j\}$ and $\{d_j\}$, where ${\mathbb C}\setminus\Gamma = D_1\cup D_2$.
\end{enumerate}

\begin{proposition}
 Assume that $\hat M $ is a solution of RH problem \eqref{jump-y}--\eqref{res_hatM_}. Then $\det \hat M=1$.
\end{proposition}
\begin{proof}
The conditions for $\hat M$ imply that $\det\hat M$ has neither a jump across $\Gamma$ no singularities at $\mu_j$ and $\nu_j$. Moreover, $\det\hat M$ tends to $1$ as $k\to\infty$, and the only possible singularity of $\det\hat M$  is at $0$.

    As $k\to0$, we have: 
\begin{equation*}
  \hat M(y,t,k)=  \frac{\ii}{k} \hat c(y,t)\begin{pmatrix}
        \delta&1\\\delta&-1
    \end{pmatrix}+\begin{pmatrix}
        \hat m_1^0(y,t)&\hat m_2^0(y,t)\\ \hat m_3^0 (y,t)&\hat m_4^0 (y,t)
    \end{pmatrix}+O(k), \quad k\to 0, \quad k\in\mathbb{C}_+.
\end{equation*}
This implies that 
\begin{equation*}
\det \hat M(k)=\frac{\ii \hat A(y,t)}{k}+O(1), \quad k\to 0.
\end{equation*}
with $\hat A(y,t)=\hat c(y,t)(\delta(\hat m_{4}^0(y,t)+\hat m_{2}^0(y,t))-\hat m_{3}^0(y,t)-\hat m_{1}^0(y,0))$. 
Then, by Liouville's theorem, $\det\hat M(k) \equiv 1 +\frac{\ii \hat A(y,t)}{k}$ with some $\hat A$. In particular, $\det\hat M (\frac{\ii}{2}) =1 +2 \hat A(y,t)$.

On the other hand, the structural condition at $\frac{\ii}{2}$ \eqref{hatM_at_i2} implies that $\det\hat M(\frac{\ii}{2}) = 1$. Thus $\hat A(y,t)\equiv 0$, which implies that $\det \hat M(k)\equiv 1$.

In particular, we have that either
\begin{subequations}\label{hatc_0}
    \begin{equation}\label{hatc}
    \hat c(y,t)=0
    \end{equation}
or
    \begin{equation}\label{hatm_i}
    \delta(m_{4}^0(y,t)+m_{2}^0(y,t))-m_{3}^0(y,t)-m_{1}^0(y,0)=0
    \end{equation}
\end{subequations}
Notice that in case $\hat c(y,t)=0$, $\hat M(y,t,k)$ is non-singular at $k=0$.

\end{proof}

\begin{proposition}\label{prop:uniqness}
     If a solution of the RH problem \eqref{jump-y}--\eqref{res_hatM_} exists, it is unique.
\end{proposition}

\begin{proof}
    Suppose that $\hat M_1$ and $\hat M_2$ are two solutions of the RH problem and consider $P\coloneqq\hat M_1(\hat M_2)^{-1}$. Obviously, $P$ has neither a jump across $\Gamma$ nor singularities at $\mu_j$ and $\nu_j$. Moreover, $P$ tends to $I$ as $k\to\infty$, and the only possible singularity of $P$ is a simple pole at $k=0$.

    Consider the development of $\hat M_j$, $j=1,2$ as $k\to 0$ with $\Im k>0$:

             \begin{equation*}
  \hat M_j(y,t,k)=  \frac{\ii}{k} \hat c_j(y,t)\begin{pmatrix}
        \delta&1\\- \delta&-1
    \end{pmatrix}+\begin{pmatrix}
        m_{1}^{0j}&m_{2}^{0j}\\m_{3}^{0j}&m_{4}^{0j}
    \end{pmatrix}+O(k), \quad k\to 0, \quad k\in\mathbb{C}_+.
\end{equation*}
Since $\det \hat M_j=1$, using \eqref{hatm_i} we have

     \begin{equation*}
         P(y,t,k)=\frac{\ii \hat B(y,t)}{k}\begin{pmatrix}
        1&1\\-1&-1
    \end{pmatrix} +O(1), \quad k\to 0, \quad k\in\mathbb{C}_+ 
     \end{equation*}
with $\hat B=-\hat c_1\left(\delta\hat m_{4}^{02}-\hat m_{3}^{02}\right)-\hat c_2\left(\delta\hat m_{2}^{01}-\hat m_{1}^{01}\right)$.
Then by Liouville's theorem we have
  \begin{equation}\label{Pexp}
         P(y,t,k)=I+\frac{\ii \hat B(y,t)}{k}\begin{pmatrix}
        -1&-1\\1&1
    \end{pmatrix} . 
     \end{equation}
In particular, we have
  \begin{equation}\label{Pi2}
         P(y,t,\frac{\ii}{2})=\begin{pmatrix}
        1-2\hat B(y,t)&-2\hat B(y,t)\\2\hat B(y,t)&1+2\hat B(y,t)
    \end{pmatrix} . 
     \end{equation}
    
     It follows that  the structural condition (at $k=\frac{\ii}{2}$) 
     considered for $P\hat M_2=\hat M_1$ reads as
\begin{equation}
    \begin{pmatrix}
        \hat f_2\left(\hat q_2+\hat q_2^{-1}-4\hat B\hat q_2\right)&\hat f_2^{-1}\left(\hat q_2-\hat q_2^{-1}-4\hat B\hat q_2\right)\\
         \hat f_2\left(\hat q_2-\hat q_2^{-1}+4\hat B\hat q_2\right)&\hat f_2^{-1}\left(\hat q_2+\hat q_2^{-1}+4\hat B\hat q_2\right)
    \end{pmatrix}=\begin{pmatrix}
        \hat f_1\left(\hat q_1+\hat q_1^{-1}\right)&\hat f_1^{-1}\left(\hat q_1-\hat q_1^{-1}\right)\\
         \hat f_1\left(\hat q_1-\hat q_1^{-1}\right)&\hat f_1^{-1}\left(\hat q_1+\hat q_1^{-1}\right)
    \end{pmatrix},
\end{equation}
 which implies the following relations among $\hat q_i$ and $\hat f_i$:
\[
f_2\hat q_2=\hat f_1\hat q_1,\quad
    \hat f_2^{-1}\hat q_2=\hat f_1^{-1}\hat q_1,\quad
    \hat f_2\hat q_2^{-1}-4\hat B \hat q_2=\hat f_1\hat q_1^{-1}.
\]
  Since $\hat q_i>0$ and $\hat f_i>0$, we conclude that 
  $\hat q_1=\hat q_2$, $\hat f_1=\hat f_2$, and $\hat B=0$.

\end{proof}

Now  assume that the RH problem \eqref{jump-y}--\eqref{res_hatM_} 
has a solution $\hat M(y,t,k)$ that satisfies the \emph{symmetries}
\begin{equation}
 \label{sym_1}
 \hat M(k)=\overline{\hat M(-\bar k)}
 \end{equation}
and 
 \begin{equation}
 \label{sym_2}
\hat M(k)=\begin{pmatrix}
    0&1\\1&0
\end{pmatrix}\hat M(-k)\begin{pmatrix}
    0&1\\1&0
\end{pmatrix}
 \end{equation} 
and  is differentiable w.r.t. $y$ and $t$.

\begin{proposition}
    Evaluating $\hat M(y,t,k)$ at particular points of $\overline{\mathbb C}$ one can 
    obtain a solution of the CH equation in the $y,t$ variables.
\end{proposition}
We proceed as follows:
\begin{enumerate}[(a)]
\item 
Starting from $\hat M(y,t,k)$, define $2\times 2$-matrix valued functions \[ \hat \Psi (y,t,k):= \hat M(y,t,k)\eul^{-\ii k p(y,t,k)\sigma_3}\] and show that $\hat\Psi(y,t,k)$ satisfies the system of differential equations:
\begin{equation}\label{Lax-hat-hat}
\begin{split}
\hat\Psi_y&=\doublehat{U}\hat\Psi, \\
\hat \Psi_t&=\doublehat{V}\hat\Psi,
\end{split}
\end{equation}
where $\doublehat{U}$ and $\doublehat{V}$ have the same (rational) dependence on $k$ as in \eqref{Lax_y_y} and \eqref{Lax_y_t}, with coefficients given in terms of $\hat M(y,t,k)$ evaluated at appropriate values of $k$.
\item
Show that the compatibility condition for \eqref{Lax-hat-hat}, i.e., the equality $\doublehat{U}_t - \doublehat{V}_y + [\doublehat{U},\doublehat{V}]=0$, reduces to 
\eqref{CH_in_y}.
\end{enumerate}

\begin{proposition}\label{Prop_Lax_y}
    Let $\hat M(y,t,k)$ be the solution of the RH problem \eqref{jump-y}--\eqref{res_hatM_} that satisfies symmetries \eqref{sym_1} and \eqref{sym_2}. Define $\hat \Psi (y,t,k):= \hat M(y,t,k)\eul^{-\ii k p(y,t,k)\sigma_3}$. 

    Then $\hat\Psi(y,t,k)$ satisfies the differential equation
\begin{equation}\label{Lax_y_Psi}
    \hat\Psi_y=\doublehat{U}\hat\Psi
\end{equation}
with 
\begin{equation}\label{hat_hat_U}
    \doublehat{U}(y,t,k)=-\ii k \sigma_3 + \hat\alpha(y,t)\begin{pmatrix}
    0&1\\1&0
\end{pmatrix}+\frac{\ii}{k} \hat \beta(y,t)\begin{pmatrix}
        1&1\\-1&-1
    \end{pmatrix},
\end{equation}
where 
\begin{enumerate}[(i)]
    \item 
$\hat \alpha(y,t)\in\mathbb{R}$ can be obtained from the large $k$ expansion of $\hat M(y,t,k)$: 
\begin{equation}\label{alpha}
\hat \alpha(y,t)=-\hat m_2^\infty(y,t),
\end{equation}
where
\begin{equation}\label{alpha-M}
\hat M(y,t,k)=I+\frac{\ii}{k}\begin{pmatrix}
    \hat m_1^\infty & \hat m_2^\infty\\
    -\hat m_2^\infty& -\hat m_1^\infty
\end{pmatrix}(y,t)+O\left(\frac{1}{k^2}\right),\qquad k\to\infty.
\end{equation}
\item
$\hat \beta(y,t)$ can be obtained from the expansion of $\hat M(y,t,k)$ as $k\to 0$: 
\begin{equation}\label{beta}
\hat \beta(y,t)=\hat c_y(y,t)\hat m_1^0(y,t)-\hat c(y,t)\hat m_{1y}^0(y,t)+\delta\left(\hat c(y,t)\hat m^0_{2y}(y,t)-\hat c_y(y,t)\hat m_2^0(y,t)\red{-}2\hat c^2(y,t)\right),
\end{equation}
where
\begin{equation}\label{beta-M}
  \hat M(y,t,k)=  \frac{\ii}{k} \hat c(y,t)\begin{pmatrix}
        \delta&1\\-\delta&-1
    \end{pmatrix}+\begin{pmatrix}
        \hat m_1^0(y,t)&\hat m_2^0(y,t)\\ \hat m_3^0 (y,t)&\hat m_4^0(y,t) 
    \end{pmatrix}+O(k), \quad k\to 0, \quad k\in\mathbb{C}_+.
\end{equation}
\end{enumerate}

\end{proposition}
\begin{proof}
    First, notice that $\hat\Psi(y,t,\lambda)$ satisfies the jump condition
\[
\hat\Psi^+(y,t,k)=\hat\Psi^-(y,t,k)J_0(k)
\]
with the jump matrix $J_0(k)$ independent of $y$. Hence, $\hat\Psi_y(y,t,k)$ satisfies the same jump condition.

As for the residue conditions, we notice that the symmetry assumptions  
\eqref{sym_1} and \eqref{sym_2} imply that if $\mu_j$ is a pole of $\hat M^{(2)}$ then $-\bar\mu_j$ is also a pole of $\hat M^{(2)}$ with $c_{\mu_j}=-\overline{c_{-\bar\mu_j}}$. Analogously, if $\eta_j$ is a pole of $\hat M^{(1)}$ then $-\bar\eta_j$ is also a pole of $\hat M^{(1)}$ with $d_{\eta_j}=-\overline{d_{-\bar\eta_j}}$. Moreover,
\begin{equation}\label{res-sym}
    \eta_j=-\mu_j, \quad d_j=-c_j, \ \ \text{and}\ 
    d_{\bar \mu_j}=\overline{c_{\mu_j}}.
\end{equation}

Since $\{c_j\}$ are independent of $y$, 
$\hat\Psi_y(y,t,k)$ satisfies the same residue conditions as $\hat\Psi(y,t,k)$ does:

\begin{align*}
\Res_{\mu_j}\hat \Psi^{(2)}&= c_j \hat \Psi^{(1)}(\mu_j) , ~ \mu_j\in D_1, \\
\Res_{\bar\mu_j}\hat \Psi^{(1)}&=\bar{ c}_j \hat \Psi^{(2)}(\bar\mu_j) , ~ \bar\mu_j\in D_2.
\end{align*}
with constants $c_j$ independent of $y$.
Consequently,  $\hat \Psi_y \hat \Psi^{-1}=\hat M_y \hat M^{-1}-\ii k \hat p_y\hat M \sigma_3 \hat M ^{-1}$ (with $\hat p_y=1$)  has neither jump no singularities at $\mu_j$ and thus it is a meromorphic function, with possible singularities at $k=\infty$ and $k=0$.

Let us analyze the behavior of $\hat \Psi_y \hat \Psi^{-1}$ as $k\to\infty$ and as $k\to 0$.

\begin{enumerate}[(i)]
    \item As $k\to\infty$, we have

    \begin{equation}
   \hat M = I+\frac{\hat M_\infty}{k}+O\left(\frac{1}{k^2}\right),
   \end{equation}
   where, due to the symmetries \eqref{sym_1} and \eqref{sym_2},
   \[
\hat M_\infty =i\begin{pmatrix}
    \hat m_1^\infty & \hat m_2^\infty\\
    -\hat m_2^\infty& -\hat m_1^\infty
\end{pmatrix}
\]
with $\hat m_j^\infty \in\mathbb{R}$.
Consequently, $
\hat M_y= \frac{\hat M_{\infty y}}{k}+O\left(\frac{1}{k^2}\right)$  and thus 
 $\hat M_y \hat M^{-1}=O\left(\frac{1}{k}\right)$, which leads to 
the following expansion for $\hat\Psi_y\hat\Psi^{-1}$:

\begin{equation}\label{Psi_y_Psi_inf}
  \hat \Psi_y \hat \Psi^{-1}=-\ii k \sigma_3 - \hat m_2^\infty\begin{pmatrix}
    0&1\\1&0
\end{pmatrix}+O(\frac{1}{k}),  \quad k\to \infty.
\end{equation}

\item As $k\to 0$, we have:

\begin{equation*}
  \hat M(y,t,k)=  \frac{\ii}{k} \hat c(y,t)\begin{pmatrix}
        \delta&1\\-\delta&-1
    \end{pmatrix}+\begin{pmatrix}
        \hat m_1^0(y,t)&\hat m_2^0(y,t)\\ \hat m_3^0 (y,t)&\hat m_4^0 (y,t)
    \end{pmatrix}+O(k), \quad k\to 0, \quad k\in\mathbb{C}_+, 
\end{equation*}
where, due to the symmetries \eqref{sym_1} and \eqref{sym_2} and determinant condition \eqref{hatm_i} (in case $\hat c \neq 0$), we have  $\delta(\hat m_{4}^0+\hat m_{2}^0)=\hat m_{3}^0+\hat m_{1}^0$.
Hence
\begin{equation*}
  \hat M_y \hat M^{-1}=  \frac{\ii}{k} \left(\hat c_y\hat m_1^0-\hat c\hat m^0_{1y}+\delta(\hat c\hat m^0_{2y}-\hat c_y\hat m_2^0)\right)\begin{pmatrix}
        1&1\\-1&-1
    \end{pmatrix}+O(1), \quad k\to 0, \quad k\in\mathbb{C}_+
\end{equation*}
and thus 
\begin{equation}\label{Psi_y_Psi_0}
\hat \Psi_y \hat \Psi^{-1}=\frac{\ii}{k} \left(\hat c_y\hat m_1^0-\hat c\hat m^0_{1y}+\delta(\hat c\hat m^0_{2y}-\hat c_y\hat m_2^0\red{-}2\hat c^2)\right)\begin{pmatrix}
        1&1\\-1&-1
    \end{pmatrix}+O(1), \quad k\to 0
\end{equation}
By Liouville's theorem, \eqref{Psi_y_Psi_inf} and \eqref{Psi_y_Psi_0}  imply \eqref{Lax_y_Psi}--\eqref{beta-M}.

\end{enumerate}

\end{proof}

\begin{proposition}\label{Prop_Lax_t}
$\hat\Psi(y,t,k)$ defined as in Proposition \ref{Prop_Lax_y} satisfies the differential equation
\begin{equation}\label{Lax_t_Psi}
    \hat\Psi_t=\doublehat{V}\hat\Psi
\end{equation}
with 
\begin{equation}\label{hat_hat_V}
    \doublehat{V}=\frac{\ii}{k} \hat \gamma(y,t)\begin{pmatrix}
        1&1\\-1&-1
    \end{pmatrix}-\frac{\ii k}{4\lambda}\begin{pmatrix}
      \hat q^2(y,t)+\hat q^{-2}(y,t)&-\hat q^2(y,t)+\hat q^{-2}(y,t)\\
    \hat q^2(y,t)-\hat q^{-2}(y,t)&-\hat q^2(y,t)-\hat q^{-2}(y,t) 
    \end{pmatrix},
\end{equation}
where 
\begin{enumerate}[(i)]
    \item 
$\hat q(y,t)\in\mathbb{R}$ can be obtained from the expansion of $\hat M(y,t,k)$ as $k\to \frac{\ii}{2}$: 
\begin{equation}\label{q-M}
    \hat M(y,t,k)=\frac{1}{2}\begin{pmatrix}
    ( \hat q+\hat q^{-1})\hat f&(\hat q-\hat q^{-1})\hat f^{-1}\\
   (\hat q-\hat q^{-1})f&(\hat q+\hat q^{-1})\hat f^{-1}
\end{pmatrix}+
O\left(k-\tfrac{\ii}{2}\right),\quad k\to \frac{\ii}{2}.
\end{equation}

\item
$\hat \gamma(y,t)$ can be obtained from the expansion of $\hat M(y,t,k)$ as $k\to 0$: 
\begin{equation}\label{gamma}
\hat \gamma(y,t)=\hat c_t(y,t)\hat m_1^0(y,t)-\hat c(y,t)\hat m_{1t}^0(y,t)+\delta\left(\hat c(y,t)\hat m^0_{2t}(y,t)-\hat c_t(y,t)\hat m_2^0(y,t)+2\hat c^2(y,t)\right),
\end{equation}
where
\begin{equation}\label{gamma-M}
  \hat M(k)=  \frac{\ii}{k} \hat c\begin{pmatrix}
        \delta&1\\-\delta&-1
    \end{pmatrix}+\begin{pmatrix}
        \hat m_1^0&\hat m_2^0\\ \hat m_3^0 &\hat m_4^0 
    \end{pmatrix}+O(k), \quad k\to 0, \quad k\in\mathbb{C}_+.
\end{equation}
\end{enumerate}

\end{proposition}

\begin{proof}

Similarly to Proposition \ref{Prop_Lax_y}, we notice that $\hat\Psi_t \hat\Psi^{-1}=\hat M_t\hat M^{-1}-\ii k\hat p_t\hat M\sigma_3\hat M^{-1}$ has neither jump nor poles at $\mu_j$, and thus it is a meromorphic function, with possible singularities at $k=\infty$, $k=0$, and $k=\pm \frac{\ii}{2}$, the latter being due to the singularity of $\hat p_t$ at $k=\pm \frac{\ii}{2}$:
\begin{equation}\label{p_t_i}
    p_t(k)=\frac{1}{k^2+\frac{1}{4}}=\mp\frac{\ii}{k\mp\frac{\ii}{2}}+O(1), \quad k \to\pm\frac{\ii}{2}.
\end{equation}
Evaluating $\hat\Psi_t\hat\Psi^{-1}$ near these points, we have the following.
\begin{enumerate}[(i)]

\item As $k\to\infty$, we have $\hat M\hat M_t^{-1}=O\left(\frac{1}{k}\right)$, $p_t(k)=O\left(\frac{1}{k^2}\right)$ and thus
\begin{equation}\label{psi-inf-t}
\hat\Psi_t\hat\Psi^{-1}(k)=\ord\left(k^{-1}\right),\qquad k\to\infty.
\end{equation}

\item As $k\to 0$, proceeding as in Proposition \ref{Prop_Lax_y} to get \eqref{Psi_y_Psi_0}, we have

\begin{equation}\label{Psi_t_Psi_0}
\hat \Psi_t \hat \Psi^{-1}=\frac{\ii}{k} \left(\hat c_t\hat m_1^0-\hat c\hat m^0_{1t}+\delta(\hat c\hat m^0_{2t}-\hat c_t\hat m_2^0\red{-}2\hat c^2)\right)\begin{pmatrix}
        1&1\\-1&-1
    \end{pmatrix}+O(1), \quad k\to 0.
\end{equation}

\item As $k\to\frac{\ii}{2}$, the structural condition at $\frac{\ii}{2}$ together with \eqref{p_t_i} yields

\begin{equation}\label{Psi_Psi_t_i}
    \hat \Psi_t \hat \Psi^{-1}=\frac{\ii}{8}\frac{1}{k-\frac{\ii}{2}}\begin{pmatrix}
    \hat q^2+\hat q^{-2}&-\hat q^2+\hat q^{-2}\\
    \hat q^2-\hat q^{-2}&-\hat q^2-\hat q^{-2}
\end{pmatrix}+O(1),\quad k \to \frac{\ii}{2}.
\end{equation}
Then, the symmetry \eqref{sym_2} implies that 
\begin{equation}\label{Psi_Psi_t_-i}
    \hat \Psi_t \hat \Psi^{-1}=\frac{\ii}{8}\frac{1}{k+\frac{\ii}{2}}\begin{pmatrix}
    \hat q^2+\hat q^{-2}&-\hat q^2+\hat q^{-2}\\
    \hat q^2-\hat q^{-2}&-\hat q^2-\hat q^{-2}
\end{pmatrix}+O(1),\quad k \to -\frac{\ii}{2}.
\end{equation}

\end{enumerate}

Combining \eqref{psi-inf-t},\eqref{Psi_t_Psi_0}, \eqref{Psi_Psi_t_-i}, and \eqref{Psi_Psi_t_-i}, and applying Liouville's theorem, we get \eqref{Lax_t_Psi}--\eqref{gamma-M}.

\end{proof}

Having obtained the pair of differential equations \eqref{hat_hat_U} and \eqref{hat_hat_V}, we can explore their compatibility, i.e., the equation 
\begin{equation}\label{compat}
\doublehat{U}_t - \doublehat{V}_y + [\doublehat{U},\doublehat{V}]=0
\end{equation}
in view of obtaining a representation of a solution of the CH equation in the $(y,t)$ variables \eqref{CH_in_y}.
For this purpose, substituting into \eqref{compat} the expressions for $\doublehat{U}$ and 
$\doublehat{V}$ from \eqref{hat_hat_U} and \eqref{hat_hat_V} respectively and
equating to $0$ the coefficients  at various expressions involving $k$ and $\lambda$, 
we get algebraic and differential equations  amongst $\hat \alpha$, $\hat \beta$, $\hat \gamma$, and $\hat q$:

\begin{proposition}

Let $\hat \alpha(y,t)$, $\hat \beta(y,t)$, $\hat \gamma(y,t)$ and $\hat q(y,t)$ be the functions determined in terms of $\hat M(y,t,k)$ as in Propositions \ref{Prop_Lax_y} and \ref{Prop_Lax_t}. Then they satisfy the following equations:
\begin{subequations}\label{rel}
\begin{align}\label{sys_rec1_}
&\hat \alpha_t+2\hat \gamma+\frac{1}{2}(\hat q^{2}-\hat q^{-2})=0,\\\label{sys_rec3_}
&\hat \beta_t-\hat \gamma_y-2\hat \gamma\hat \alpha=0,\\\label{sys_rec4_}
&-(\hat q^{-2}-\hat q^2)+8\hat \beta \hat q^2=0,\\\label{sys_rec5_}
&(\hat q^2+ \hat q^{-2})_y-2\hat \alpha(\hat q^2- \hat q^{-2})=0,\\\label{sys_rec6_}
&(\hat q^2- \hat q^{-2})_y-2\hat \alpha(\hat q^2+ \hat q^{-2})=0.
\end{align}
\end{subequations}
\end{proposition}

\begin{proposition}\label{prop:12}
\label{reduce}
Introducing  $\hat m(y,t)$, $\hat u(y,t)$, and 
$\hat v(y,t)$
in terms of $\hat \alpha$, $\hat \beta$,  $\hat \gamma$, and $\hat q$ by
\begin{equation}\label{hmbbgg}
\hat u=4\hat q^2 \hat \gamma,\quad \hat m = \hat q^4-1,\quad \hat v = \hat u_y \hat q^2,
\end{equation}
 equations \eqref{rel} reduce to \eqref{CH_in_y}.
\end{proposition}

\begin{proof}
    First, notice that adding \eqref{sys_rec5_} and \eqref{sys_rec6_} gives
\begin{equation}\label{alpha_}
 \hat \alpha=\frac{\hat q_y}{\hat q}   
\end{equation}
 whereas \eqref{sys_rec4_} yields 
 \begin{equation}\label{beta_} 
  \hat\beta=-\frac{1}{8}\frac{\hat q^4-1}{\hat q^4}  . 
 \end{equation}

Then, introduce $\hat u:=4\hat q^2 \hat \gamma$. Then $\hat u_y=8 \hat q \hat q_y\hat \gamma+4\hat q^2\hat \gamma_y$, which, in view of \eqref{alpha_} and \eqref{sys_rec3_}, reduces to $\hat u_y=4\hat q^2\hat \beta_t$. 
Differentiating  \eqref{beta_} gives 
\[
\beta_t = \frac{1}{4}\left(\frac{1}{q^2}\right)_t\frac{1}{q^2}
\]
and thus
\begin{equation}\label{hatu_y}
    \hat u_y=\left(\frac{1}{\hat q^2}\right)_t.
\end{equation}

Now observe that using  \eqref{alpha_} and \eqref{hatu_y} one can express $\left( \hat u_y \hat q^2 \right)_y $
as follows:
\[
\left( \hat u_y \hat q^2 \right)_y  = \left( \left(\frac{1}{\hat q^2}\right)_t \hat q^2 \right)_y = \frac{2q_tq_y-2q q_{ty}}{q^2} = -2\hat \alpha_t,
\]
In view of this, equation \eqref{sys_rec1_} can be written as
\begin{equation} \label{gamma_eq_}
\hat \gamma = \frac{1}{4\hat q^2} \left( \hat q^4 - 1 + \left( \hat u_y \hat q^2 \right)_y \hat q^2 \right).
\end{equation}

Equations \eqref{hatu_y} and \eqref{gamma_eq_} give us the following system of equations:
\begin{equation}\label{ch_pre}
   \begin{aligned}
   & \hat u_y=\left(\frac{1}{\hat q^2}\right)_t,\\
   & \hat u-\hat q^4+1-(\hat u_y \hat q^2)_y\hat q^2=0.
\end{aligned} 
\end{equation}
Introducing $\hat m:=\hat q^4-1$,  system \eqref{ch_pre} becomes

\begin{equation}\label{sys_CH_y_rec_}
 \begin{aligned}
&(\sqrt{\hat m(y,t)+1})_t=-\hat u_y(y,t)(\hat m(y,t)+1),\\
&\hat m(y,t)=\hat u(y,t)-\left[\hat u_y(y,t)\sqrt{\hat m(y,t)+1}\right]_y\sqrt{\hat m(y,t)+1}.
\end{aligned} 
\end{equation}
Introducing $\hat v = \hat u_y \sqrt{\hat m + 1}$,  system 
\eqref{sys_CH_y_rec_} reduces to \eqref{CH_in_y}.

\end{proof}

Now let us discuss sufficient conditions on the contour $\Gamma$,  jump matrix $\hat J_0(k)$ and residue conditions that provide the symmetries \eqref{sym_1} and \eqref{sym_2}.

\begin{proposition}\label{prop:sym}
    Let the contour $\Gamma$ be invariant w.r.t. $k\mapsto-k$ and $k\mapsto-\bar k$, the jump matrix $\hat J_0$ satisfy the symmetries
    \begin{equation}
 \label{sym_J_1}
 \hat J_0(k)=\overline{\hat J_0(-\bar k)}
 \end{equation}
and 
 \begin{equation}
 \label{sym_J_2}
J_0(k)=\begin{pmatrix}
    0&1\\1&0
\end{pmatrix}J_0(-k)\begin{pmatrix}
    0&1\\1&0
\end{pmatrix},
 \end{equation} 
 and the parameters of the residue conditions satisfy \eqref{res-sym}.

    Then the solution of the RH problem  \eqref{jump-y}--\eqref{res_hatM_} $\hat M$ that satisfies symmetries \eqref{sym_1} and \eqref{sym_2}
\end{proposition}

\begin{proof}
    In order to prove this statement, we will show that (i) $\overline{\hat M(-\bar k)}$ satisfies the same RH problem as $\hat M(k)$, (i) $\begin{pmatrix}
    0&1\\1&0
\end{pmatrix}\hat M(-k)\begin{pmatrix}
    0&1\\1&0
\end{pmatrix}$ satisfies the same RH problem as $\hat M(k)$.

\begin{enumerate}[(i)]
    \item Direct computations show that $\overline{\hat M(-\bar k)}$ satisfies the same normalization condition, singularity condition at $0$ and structural condition at $\pm\frac{\ii}{2}$ as $\hat M(k)$. Moreover, using $\overline{\hat p(y,t,-\bar k)}=\hat p(y,t,k)$ together with \eqref{sym_J_1} and \eqref{res-sym} we can conclude that $\overline{\hat M(-\bar k)}$ satisfies the same jump condition and the same residue condition as $\hat M(k)$. 
    Since the solution of RH problem   \eqref{jump-y}--\eqref{res_hatM_} is unique, we  conclude that $\hat M(k)=\overline{\hat M(-\bar k)}$.

    \item Similarly, direct computations show that $\begin{pmatrix}
    0&1\\1&0
\end{pmatrix}\hat M(-k)\begin{pmatrix}
    0&1\\1&0
\end{pmatrix}$ satisfies the same normalization condition, singularity condition at $0$ and structural condition at $\pm\frac{\ii}{2}$. Moreover, using $\hat p(y,t,- k)=\hat p(y,t,k)$ together with \eqref{sym_J_2} and \eqref{res-sym} we can conclude that $\begin{pmatrix}
    0&1\\1&0
\end{pmatrix}\hat M(-k)\begin{pmatrix}
    0&1\\1&0
\end{pmatrix}$ satisfies the same jump condition and the same residue condition as $\hat M(k)$. 
Since the solution of RH problem   \eqref{jump-y}--\eqref{res_hatM_} is unique, we conclude that $\hat M(k)=\begin{pmatrix}
    0&1\\1&0
\end{pmatrix}\hat M(-k)\begin{pmatrix}
    0&1\\1&0
\end{pmatrix}$.
\end{enumerate}
\end{proof}

\section{RH problem formalism for the periodic CH problem in $y,t$ variables}\label{sec:RHpbl}

In Section \ref{sec:CH-y} we discussed how a function solving, locally, the CH equation 
in $y,t$ variables appears from the solution of a Riemann-Hilbert
problem parametrized by $y$ and $t$ in an appropriate way. In this section, we 
present a particular  Riemann-Hilbert problem associated with the periodic problem
for the CH equation in $y,t$ variables in the sense (i) that  it can be formulated in terms
of the spectral functions determined by the initial data only, (ii) whose solution
evaluated at $t=0$ recovers the initial data, and (iii) whose solution provides
the periodicity, i.e.  ensures that the values of the solution 
of the CH equation at the ends of the interval match.

\subsection{Construction of the master RH problem}

It is  the pre-RH problem associated with the problem on the interval that \emph{suggests}
considering the RH problem with the jump conditions (i) parametrized by $y$ and $t$
and (ii) given in terms of the spectral functions $a(k)$, $b(k)$, 
and \emph{ a function ${\mathfrak R}(k)$ that can be calculated in terms of $a(k)$ and $b(k)$}.

Namely, we already noticed that in the case of the periodic problem, 
$S_1(k)=S(k)$ and thus $R(k)=R_1(k)$ and $\tilde R(k)=\tilde R_1(k)$.
It follows that in this case, the l.h.s. of the Global Relations 
\eqref{gl_rel} and \eqref{gl_rel_til}, after dividing respectively by
$A(k)A_1(k)$ and $\tilde A(k)\tilde A_1(k)$, becomes 
 quadratic expressions w.r.t. respectively $R(k)$ and $\tilde R(k)$,
 with coefficients expressed in terms of $a$, $b$, $\tilde a$, and $\tilde b$.
 The decisive step consists in replacing the r.h.s. in the Global relations by $0$,
 which gives two quadratic equations, respectively for $\mathfrak R$ and $\tilde{\mathfrak R}$:
 \begin{equation}\label{calR}
   -e^{2ik\theta}b^*(k){\mathfrak R}^2(k)
		+ \left(e^{2ik\theta} a^*(k) - a(k)\right) {\mathfrak R}(k)  +b(k) = 0  
 \end{equation}
        and 
\[
		-e^{2ikL}\tilde b^*(k){\tilde{\mathfrak R}}^2(k)
		+ \left(e^{2ikL} \tilde a^*(k) - \tilde a(k)\right) {\tilde{\mathfrak R}}(k)  +\tilde b(k) = 0.
		\]
         Taking into consideration that $\tilde a(k)=a(k) e^{ik(L-\theta)}$ and 
         $\tilde b(k)=b(k) e^{ik(L-\theta)}$, the latter equation reduces to \eqref{calR},
         which will be taken as the (part of) \emph{definitions} of functions 
         $\mathfrak R(k)$ and ${\tilde{\mathfrak R}}(k)$. 

Being solutions of a  quadratic equation, ${\mathfrak R}(k)$ and $\tilde{\mathfrak R}(k)$ 
 can be expressed analytically by  
		\begin{equation}\label{calR-form}
		    {\mathfrak R}(k) = 
		\frac{e^{-ik\theta}a(k)-e^{ik\theta}a^*(k)
			-\sqrt{\Delta^2({k})-4}}{-e^{ik\theta}b^*(k)}
		\end{equation}
        and similarly for $\tilde{\mathfrak R}(k)$,
	where $\Delta ({k}) :=e^{-ik\theta}a(k) + e^{ik\theta}a^*(k)$ and where  
		in order to fix a one-valued branches of ${\mathfrak R}(k)$  and 
        $\tilde{\mathfrak R}(k)$ (which can be different) for $k$ in the complex plane, one needs to introduce cuts
        connecting 
		simple zeros of $\Delta^2({z})-4$} and to specify the values 
        of ${\mathfrak R}(k)$  and 
        $\tilde{\mathfrak R}(k)$ at dedicated  points.      
        Looking at the r.h.s. of the Global Relations suggests to
specify them as follows:
\begin{equation}\label{spec_R}
{\mathfrak R}(\infty)=0, \quad \tilde{\mathfrak R}\left(\frac{i}{2}\right)=0.    
\end{equation}

Notice that there are two possibilities: either 
${\mathfrak R}\left(\frac{i}{2}\right)=0$ and thus
${\mathfrak R}(k)\equiv\tilde{\mathfrak R}(k)$ or 
${\mathfrak R}\left(\frac{i}{2}\right)\ne 0$ and thus
${\mathfrak R}(k)$ and $\tilde{\mathfrak R}(k)$  are the different branches of the same analytic function defined on the corresponding Riemannian surface.

Before passing to the formulation of the RH problem with the jump given in terms of 
$a(k)$, $b(k)$, ${\mathfrak R}(k)$,  and $\tilde{\mathfrak R}(k)$, we present the analytic 
properties of ${\mathfrak R}(k)$  and $\tilde{\mathfrak R}(k)$ following their definitions.
But before doing that, we notice that the scattering matrix $s(k)$
in the case when $m$ in the first equation of the Lax pair \eqref{lax} (considered for $t=0$)
equals $0$ outside the interval $x\in[0,L]$ is related to the monodromy 
matrix $\mathcal M(k):= \Phi(L,0,k)$, where $\Phi(x,L,k)$ is  
the solution of this equation fixed by the condition $\tilde \Phi(0,0,k)=I$.
Namely, 
\[
\mathcal M(k)=
\eul^{-ikp(L,0,k)\sigma_3} s^{-1}(k)=
\begin{pmatrix}
          a(k) \eul^{-\ii k \theta} & -b(k)\eul^{-\ii k \theta}\\
        -{b^*( k)}\eul^{\ii k \theta} & {a^*(k)}\eul^{\ii k \theta}  
\end{pmatrix}.
\]

\begin{proposition}\label{prop:R}
    (1) Let $\mathfrak{R}_{(1)}(k)$ and $\mathfrak{R}_{(2)}(k)$ be two solutions of \eqref{gl_rel}
    as functions in the complex plane with the branch cuts as above.
Then the following properties hold:

    \begin{equation}\label{propR1}
        \mathfrak{R}_{(1)}(k)\mathfrak{R}_{(2)}(k)=-\frac{\mathcal M_{12}(k)}{\mathcal M_{21}(k)}
        =-\frac{b(k)\eul^{-\ii k \theta}}{{b^*(k)}\eul^{\ii k \theta}}
        =-\frac{\tilde b(k)\eul^{-\ii k L}}{{\tilde b^*(k)}\eul^{\ii k L}} ,
    \end{equation}
        \begin{equation}\label{propR2}
        \mathfrak{R}_{(1)}(k)+\mathfrak{R}_{(2)}(k)=\frac{\mathcal M_{11}(k)-\mathcal M_{22}(k)}{\mathcal M_{21}(k)}
        =-\frac{a(k)\eul^{-\ii k \theta}-{a^*(k)}\eul^{\ii k \theta}}{{b^*(k)}\eul^{\ii k \theta}}
        =-\frac{\tilde a(k)\eul^{-\ii k L}-{\tilde a^*(k)}\eul^{\ii k L}}{{\tilde b^*(k)}\eul^{\ii k L}},
    \end{equation}

            \begin{equation}\label{propR3}
        \mathfrak{R}_{(2)}^*(k)=\frac{1}{\mathfrak{R}_{(1)}(k)}, \quad \mathfrak{R}_{(1)}^*(k)=\frac{1}{\mathfrak{R}_{(2)}(k)}.
    \end{equation}

(2) Let $\mathfrak{R}(k)$ be either $\mathfrak{R}_{(1)}(k)$ or $\mathfrak{R}_{(2)}(k)$.
Then
\begin{equation}\label{propR5}
      \mathfrak{R}(k) {b^*(k)}\eul^{\ii k \theta}=-\mathfrak{R}^*(k)b(k)\eul^{-\ii k \theta},
    \end{equation}

    \begin{equation}\label{propR6}
        (a(k)-b(k)\mathfrak{R}^*(k))(a^*(k)-b^*(k)\mathfrak{R}(k))=1
    \end{equation}
and thus 
       \item 
    \begin{equation}\label{propR7}
        (a-b\mathfrak{R}^*)(k)\neq 0 \ \ \text{for all}\ k.
    \end{equation} 

(3) Let $\mathfrak{R}_\pm (k)$ be the limiting values of $\mathfrak{R}$ as its argument approaches 
           the (oriented) branch cut from the corresponding side. Then
    \begin{equation}\label{propR8}
        \mathfrak{R}_+(k) \mathfrak{R}^{*}_-(k)=\mathfrak{R}_-(k) \mathfrak{R}^{*}_+(k)=1,
    \end{equation}  

        \begin{equation}\label{propR9}
        (a^*-b^*\mathfrak{R}_-)(k)=\eul^{-2\ii k \theta} (a-b\mathfrak{R}^{*}_+)(k),\quad (a-b\mathfrak{R}^{*}_-)(k)=\eul^{2\ii k \theta}(a^*-b^*\mathfrak{R}_+)(k).
    \end{equation} 

\end{proposition}
Additionally, \begin{equation}\label{propR10}
        \mathfrak{R}(-k)=\mathfrak R^*(k),
    \end{equation}
    due to the corresponding symmetries $a(k)$ and $b(k)$.

The RH problem formalism for the periodic problem for the CH equation involves the following
steps.

\noindent\textbf{Step 1}. Formulate the RH problem for $\hat M(y,t,k)$, depending on $y$ and $t$ as parameters,
providing
\begin{itemize}
    \item Jump conditions
    \begin{equation}\label{hatM-jump}
        \hat M_+(y,t,k)= \hat M_-(y,t,k)\hat J(y,t,k), \quad k\in \hat\Sigma
    \end{equation}
    replacing in \eqref{jump-xt} and \eqref{j0-xt}
    $p(x,t,k)$ by $\hat p(y,t,k)=y+\frac{t}{2\lambda}=y-\frac{t}{2(k^2+\frac{1}{4})}$,
    $R(k)$ and $R_1(k)$ by $\mathfrak R(k)$, and $\tilde R(k)$ and $\tilde R_1(k)$ 
    by $\tilde{\mathfrak R}(k)$;
    \item 
    Normalization condition $\hat M(k)\to I$ as $k\to \infty$;
    \item 
    Singularity conditions as $ k\to 0$:
     \begin{equation}\label{hatM_at_0}
  \hat M(y,t,k)=  \frac{\ii}{k} \hat c(y,t)\begin{pmatrix}
        0&1\\0&-1
    \end{pmatrix}+O(1), \quad k\in\mathbb{C}_+,
\end{equation}
\begin{equation}\label{hatM_at_0_-}
   \hat M(y,t,k)=  -\frac{\ii}{k} \hat c(y,t)\begin{pmatrix}
        -1&0\\1&0
    \end{pmatrix}+O(1),  \quad k\in\mathbb{C}_-,
\end{equation} 
with some (unspecified) $\hat c(y,t)\in\mathbb{R}$;
\item 
Structural condition 
\begin{equation}\label{hatM_at_i2_}
    \hat M(y,t,k)=\frac{1}{2}\begin{pmatrix}
    ( \hat q+\hat q^{-1})\hat f&(\hat q-\hat q^{-1})\hat f^{-1}\\
   (\hat q-\hat q^{-1})f&(\hat q+\hat q^{-1})\hat f^{-1}
\end{pmatrix}+
O\left(k-\tfrac{\ii}{2}\right)
\end{equation}
where $\hat q=\hat q(y,t)$, $\hat f=\hat f(y,t)$ are not specified;

\item Residue conditions at zeros $\mu_j$ of $\mathfrak R(k)$ in $D_1$ and $D_3$
and at zeros  $\bar\mu_j$ of $\mathfrak R^*(k)$ in $D_2$ and $D_4$:
\begin{subequations}\label{res_hatM_R}
    \begin{align}
    \label{res_hatM-a}
\Res_{\mu_j}\hat M^{(2)}&= \hat M^{(1)}(\mu_j)a^2(\mu_j) \eul^{-2\ii \mu_j \hat p(\mu_j)} \Res_{\mu_j} \mathfrak R, ~ \mu_j\in D_1,\\
\label{res_hatM-b}
\Res_{\bar\mu_j}\hat M^{(1)}&= \hat M^{(2)}(\bar\mu_j)\eul^{-2\ii \bar\mu_j\theta} \eul^{2\ii \bar\mu_j \hat p(\bar\mu_j)} \Res_{\bar\mu_j} {\mathfrak R}^*, ~ \bar\mu_j\in D_2, \\ 
\label{res_hatM-c}
\Res_{\mu_j}\hat M^{(2)}&=\hat M^{(1)}(\mu_j)  \eul^{2\ii \mu_j \theta} \eul^{-2\ii \mu_j \hat p(\mu_j)}\Res_{\mu_j} \mathfrak R,~\mu_j\in D_3, \\
\label{res_hatM-d}
\Res_{\bar\mu_j}\hat M^{(1)}&=\hat M^{(2)}(\bar\mu_j) a^{*2}(\bar\mu_j) \eul^{2\ii \bar\mu_j \hat p(\bar\mu_j)}\Res_{\bar\mu_j} {\mathfrak R}^*,~ \bar\mu_j\in D_4. 
\end{align}
\end{subequations}

\end{itemize}

\noindent\textbf{Step 2}. Prove that the solution of the RH problem above,
being evaluated at $t=0$, gives rise to the initial condition $\hat u(y,0)$.

\smallskip

\noindent\textbf{Step 3}. Prove that the solution of the RH problem above
gives rise to the solution of the CH equation periodic in $y$ for all fixed $t>0$.

\begin{remark}
    Conditions at $k=0$ formulated above correspond to the generic case, 
    when $a(k)$ and $b(k)$ are singular at this point. The general case 
    (see item (3) of Proposition \ref{prop:RHP-xt})
    can be treated as well, but it order to fix the ideas and to make the presentation 
    more concise, in the present paper  we restrict ourselves to the generic case.
\end{remark}
\begin{remark}
The fact that the construction of the master  RH problem involves two functions,
${\mathfrak R}(k)$ and $\tilde{\mathfrak R}(k)$, that replace the spectral functions
associated with the boundary values in the pre-RH problem, distinguishes the case of the 
CH equation from other nonlinear PDE like the NLS equation, the KdV equation and its modifications, etc. This is due to an additional singularity in the associated
$t$-equation of the Lax pair, which in turn gives rise to two global relations.
Remarkably, these two functions appear to be branches of a single function analytic 
on two-sheeted Riemannian surface. Moreover, they can be either different branches
or a single branch. In the latter case, where ${\mathfrak R}(k)=\tilde{\mathfrak R}(k)$,
the construction of the RH problem simplifies. 
\end{remark}

Now let's make the jump conditions \eqref{hatM-jump} precise.
First, we notice that since we are going to replace analytic functions (in the pre-RH problem) by ones having jumps across the branch cuts (${\mathfrak R}(k)$ and 
${\mathfrak R}^*(k)$), the resulting jump contour $\hat\Sigma$ has to include the branch cuts:
$\hat\Sigma =\check\Sigma\cup_{j=1}^N\Sigma_j$, where 
$\check\Sigma = \Sigma\setminus\{\cup_{j=1}^N\Sigma_j\} $ and
$\{\Sigma_j\}$ are segments connecting 
simple zeros of $\Delta^2({k})-4$.
Due to the symmetry properties of the Lax pair equations, these zeros   are located either on the real axis or on the imaginary
axis inside $\D{T}=\{k: |k|=\frac{1}{2}\}$.
In what follows we assume that there is a finite number 
of such zeros.

A typical $\hat\Sigma$ is shown on Figure \ref{fig:hatSigma}, where the small circles (of radius 
$\epsilon$) are chosen such that $\{\Sigma_j\}$
do not intersect the corresponding disks.
However, it is possible that the point $k=0$ belongs to either a vertical or a horizontal
segment. 

Then $\hat J(y,t,k) = \ee^{-ik\hat p \sigma_3}\hat J_0(k)\ee^{ik\hat p \sigma_3}$,
where $\hat J_0(k)$ is as follows.

\noindent\textbf{(i) $k\in\check\Sigma$}. Here the construction of $\hat J_0(k)$
follows that of $J_0(k)$ \eqref{j0-xt}, 
 with 
$\Gamma(k)$ and $\tilde\Gamma(k)$ replaced by respectively  $G(k)$ and
$\tilde G(k)$,  and $\Gamma_j(k)$ replaced by   $G_j(k)$,
$j=1,2,3$:
\begin{equation}\label{j0-yt}
    \hat J_0(k)=
    \begin{cases}
\begin{pmatrix}
    \ee^{\ii k(L-\theta)}& G_2(k)\\
   0 &\ee^{-\ii k(L-\theta)}
\end{pmatrix}, & |k|>\frac{1}{2}, |k-\frac{\ii}{2}|=\epsilon,\\
\begin{pmatrix}
          1& -\tilde{G}_1(k)\\
         0&1   
\end{pmatrix}
\begin{pmatrix}
          1& 0\\
         \tilde{G}(k)&1   
\end{pmatrix}, &   |k|=\frac{1}{2}, |k-\frac{\ii}{2}|<\epsilon,\\
\begin{pmatrix}
 \ee^{\ii k(L-\theta)}&0\\
  {G}_3(k)&\ee^{-\ii k(L-\theta)}
\end{pmatrix}, & |k|<\frac{1}{2}, |k-\frac{\ii}{2}|=\epsilon,\\
\begin{pmatrix}
          1& -{G}_1(k)\\
         0&1   
\end{pmatrix}
\begin{pmatrix}
          1& 0\\
         {G}(k)&1   
\end{pmatrix},& \Im k > 0, |k|=\frac{1}{2}, |k-\frac{\ii}{2}|>\epsilon,\\
\begin{pmatrix}
          1& 0 \\ -{G}^*_1(k) & 1
\end{pmatrix}
\begin{pmatrix}
          1 - |r(k)|^2& r^*(k)\\
         -r(k)&1  
\end{pmatrix}
\begin{pmatrix}
          1& G_1(k)\\
         0&1   
\end{pmatrix}, &  \Im k = 0, |k|>\frac{1}{2},\\
\begin{pmatrix}
          1& -{G}^*( k)\\
         0&1   
\end{pmatrix}
\begin{pmatrix}
          1 - |r(k)|^2& r^*(k)\\
         -r(k)&1  
\end{pmatrix}
\begin{pmatrix}
          1& 0\\
         {G}(k)&1   
\end{pmatrix}, &  \Im k = 0, |k|<\frac{1}{2}
\end{cases}
\end{equation}
on parts of $\Sigma$ in the upper half-plane 
and 
$\hat J_0(k)=
\left(\begin{smallmatrix}
    0 & 1 \\ 1 & 0
\end{smallmatrix}\right)\hat J_0^*(k)\left(\begin{smallmatrix}
    0 & 1 \\ 1 & 0
\end{smallmatrix}\right)$ for $k\in{\mathbb C}_-$. 
Here
\begin{align}\label{G-G1}
        G&:=\frac{{\mathfrak R}^*}{a(a-b{\mathfrak R}^*)}={\mathfrak R}^*\ee^{-2\ii k \theta}+\frac{b^*}{a}, 
        \quad
    G_1:=\frac{\ee^{2\ii k \theta}a{\mathfrak R}}{a-b{\mathfrak R}^* }=a^2({\mathfrak R}-\frac{b}{a})\\ \label{tildeG-G1}
 \tilde G&:=\frac{ \tilde{\mathfrak R}^*}{ \tilde a( \tilde a- \tilde b 
 \tilde{\mathfrak R}^*)}= \tilde{\mathfrak R}^*\ee^{-2\ii k L}+\frac{ \tilde b^*}{ \tilde a},\quad
     \tilde G_1:=\frac{\ee^{2\ii k L} \tilde a \tilde{\mathfrak R}}{ \tilde a- \tilde b \tilde{\mathfrak R}^* }= \tilde a^2( \tilde{\mathfrak R}-\frac{ \tilde b}{ \tilde a}),\\
     \label{G2}
      G_2&:=a\tilde a\ee^{\ii k (L+\theta)}\frac{{\mathfrak R}-\tilde{\mathfrak R}}{(a+b^*{\mathfrak R}\ee^{2\ii k \theta})(\tilde a+\tilde b^*\tilde{\mathfrak R}\ee^{2\ii k L})}=\begin{cases}
          0, & \text{if}\ \mathfrak R = \tilde{\mathfrak R},\\
          a\tilde a({\mathfrak R}-\tilde{\mathfrak R}), & \text{if}\ \mathfrak R \ne \tilde{\mathfrak R},
      \end{cases}\\   \label{G3}
      G_3&:=\frac{{\mathfrak R}^*-\tilde{\mathfrak R}^*}{(a-b{\mathfrak R}^*)( \tilde a- \tilde b \tilde{\mathfrak R}^*)}=\begin{cases}
          0, & \text{if}\ \mathfrak R = \tilde{\mathfrak R},\\
      \ee^{-\ii k(L+\theta)}({\mathfrak R}^*-
      \tilde{\mathfrak R}^*), &  \text{if}\ \mathfrak R \ne \tilde{\mathfrak R}.
      \end{cases}
\end{align}
(we have used properties of ${\mathfrak R}$ and $\tilde{\mathfrak R}$
from Proposition \ref{prop:R}).

\medskip

\noindent\textbf{(ii)}
On $\{\Sigma_j\}$, the jumps are given in terms of the jumps of ${\mathfrak R}$; namely,
on the the vertical segments,
	\begin{equation}\label{sigma_j_ver}
	    \hat J_0(k)=\begin{cases} \begin{pmatrix}
		1  & \ee^{2\ii k \theta}({\mathfrak R}_+(k)-{\mathfrak R}_-(k)) \\ 0 & 1
	\end{pmatrix}, & k \in \Sigma_j\cap \D{C}_-\\
	\begin{pmatrix}
		1  & 0 \\
		\ee^{-2\ii k \theta}({{\mathfrak R}^*}_+(k)-{{\mathfrak R}^*}_-(k))  & 1
	\end{pmatrix}, & k \in \Sigma_j\cap \D{C}_+
	\end{cases}
	\end{equation}
    whereas on the horizontal segments,
\begin{equation}\label{sigma_j_hor}
\hat J_0(k)=\begin{cases}
    \begin{pmatrix}
          1& 0 \\ -{G}^*_{1-}(k) & 1
\end{pmatrix}
\begin{pmatrix}
          1 - |r(k)|^2& r^*(k)\\
         -r(k)&1  
\end{pmatrix}
\begin{pmatrix}
          1& G_{1+}(k)\\
         0&1   
\end{pmatrix}, &  \Im k = 0, |k|>\frac{1}{2},\\
\begin{pmatrix}
          1& -{G}_-^*( k)\\
         0&1   
\end{pmatrix}
\begin{pmatrix}
          1 - |r(k)|^2& r^*(k)\\
         -r(k)&1  
\end{pmatrix}
\begin{pmatrix}
          1& 0\\
         {G}_+(k)&1   
\end{pmatrix}, &  \Im k = 0, |k|<\frac{1}{2}.
\end{cases}
    \end{equation}

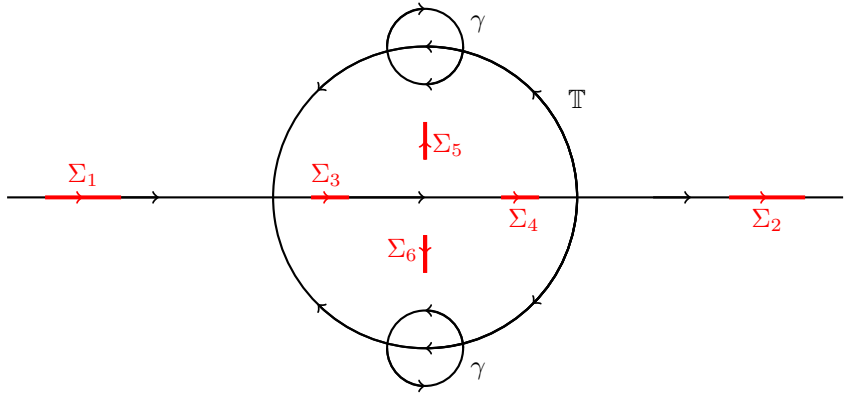
\begin{figure}
    \centering
    \begin{tikzpicture}
    \draw[thick] (0,0) circle (2);
 \draw[thick] (-5.5,0) -- (5.5,0);
    \draw[->, thick] (2,0) arc[start angle=0, end angle=90, radius=2];
    \draw[->, thick] (2,0) arc[start angle=360, end angle=270, radius=2];
 \draw[->, thick] (3,0) -- (3.5,0);
  \draw[->, thick] (-4,0) -- (-3.5,0);
    \draw[->, thick] (-1,0) -- (0,0);
    
            \draw[line width=1.5pt, red] (0, -0.5) -- (0, -1);
             \draw[thick, red, ->] (0, -0.5) -- (0, -0.75);
        \draw[line width=1.5pt, red] (0, 0.5) -- (0, 1);
\draw[thick, red, ->] (0, 0.5) -- (0, 0.75);

                    \draw[line width=1.5pt, red] (0, -0.5) -- (0, -1);
        \draw[line width=1.5pt, red] (-5, 0) -- (-4, 0);
\draw[thick, red, ->] (-5, 0) -- (-4.5, 0);

        \draw[line width=1.5pt, red] (4, 0) -- (5, 0);
\draw[thick, red, ->] (4, 0) -- (4.5, 0);

        \draw[line width=1.5pt, red] (-1.5, 0) -- (-1, 0);
\draw[thick, red, ->] (-1.5, 0) -- (-1.25, 0);

        \draw[line width=1.5pt, red] (1, 0) -- (1.5, 0);
\draw[thick, red, ->] (1, 0) -- (1.25, 0);

\draw[->, thick] (0.5,2) arc[start angle=0, end angle=-90, radius=0.5];
\draw[->, thick] (-0.5,2) arc[start angle=180, end angle=90, radius=0.5];

\draw[->, thick] (0.5,-2) arc[start angle=0, end angle=90, radius=0.5];
\draw[->, thick] (-0.5,-2) arc[start angle=180, end angle=270, radius=0.5];

\draw[thick] (0,2) circle (0.5);

\draw[thick] (0,-2) circle (0.5);

    \draw[->, thick] (2,0) arc[start angle=0, end angle=45, radius=2];
    \draw[->, thick] (2,0) arc[start angle=0, end angle=135, radius=2];

    \draw[->, thick] (2,0) arc[start angle=0, end angle=-45, radius=2];
        \draw[->, thick] (2,0) arc[start angle=0, end angle=-135, radius=2];

         \node at (0.7, 2.3) {$\gamma$};

         \node at (0.7, -2.3) {$\gamma$};

      \node at (2, 1.3) {$\mathbb{T}$};

              \node at (-4.5, 0.3) {$\red{\Sigma_1}$};

         \node at (4.5, -0.3) {$\red{\Sigma_2}$};

                   \node at (-1.3, 0.3) {$\red{\Sigma_3}$};

         \node at (1.3, -0.3) {$\red{\Sigma_4}$};

                         \node at (0.3, 0.7) {$\red{\Sigma_5}$};

         \node at (-0.3, -0.7) {$\red{\Sigma_6}$};

\end{tikzpicture}
    \caption{Contour $\hat\Sigma$}
    \label{fig:hatSigma}
\end{figure}

\begin{proposition}\label{prop:}
    For all $y\in [0,\theta]$ and $t>0$, if the RH problem \eqref{hatM-jump}-\eqref{res_hatM_R},
    where the jump matrix $\hat J_0(k)$ is given by \eqref{j0-yt}-\eqref{sigma_j_hor},
    has a solution $\hat M(y,t,k)$, it is unique. Moreover, 
    $\hat M(y,t,k) = \left(
    \begin{smallmatrix}
        0 & 1 \\ 1 & 0
    \end{smallmatrix}\right) \hat M(y,t,-k)
    \left(\begin{smallmatrix}
        0 & 1 \\ 1 & 0
    \end{smallmatrix}\right)$.
\end{proposition}
\begin{proof}
    The symmetry property \eqref{propR10} together with the related properties of 
    $a(k)$ and $b(k)$ provides that 
    $\hat J(y,t,k) = \left(
    \begin{smallmatrix}
        0 & 1 \\ 1 & 0
    \end{smallmatrix}\right) \hat J(y,t,-k)
    \left(\begin{smallmatrix}
        0 & 1 \\ 1 & 0
    \end{smallmatrix}\right)$. Consequently, 
    (i) the uniqueness of the solution 
    of the RH problem follows from Proposition \ref{prop:uniqness};
    (ii) the symmetry of $\hat M(y,t,k)$ follows from the uniqueness of the solution 
    of the RH problem, see Proposition \ref{prop:sym}.
\end{proof}

\subsection{Verification of initial conditions}
The idea of verification of initial conditions (Step 2) is as follows: transform the master RH problem
evaluated at $t=0$ to a RH problem  that characterizes the initial condition $\hat u_0(y)$
(or $\hat m_0(y)$) is such a way that does not affect the procedure of ``extracting''
$\hat m$ form the solution of the RH problem.

We begin with the formulation of the RH problem associated with the Lax equation 
\eqref{Lax_y_y} at $t=0$, i.e., with the coefficients expressed in terms of $\hat m_0(y)$.
This RHP can be derived from the direct problem on the whole line with the potential supported on $[0,L]$. Similarly to the problem on the interval, its construction follows
the analysis in the $x,t$ variables, see Section \ref{sec:RH-line-xt},
taking into account that our while line problem is such that 
$\hat m_0(y) $ is continued by $0$ for $y\not\in [0,\theta]$
and thus $\Phi_-$ is actually  $\Phi_2$ from Section \ref{sec:RH-y-int}
(similarly for $\Phi_+$ and $\Phi_3$). Then the pre-RH problem from Proposition 
\ref{prop:RHP-xt} transforms to the following RH problem,
which is parametrized by $y$ and provides the solution of the inverse problem, 
i.e., the reconstruction of $m_0(y)$ from the spectral functions $a(k)$ and $b(k)$
and the discrete parameters $\{\nu_j,c_j\}_1^N$, where $c_j=\frac{1}{b(i\nu_j)\dot a(i\nu_j)}=-\frac{b^*(i\nu_j)}{\dot a(i\nu_j)}$ \cite{BS08}.

\begin{proposition}\label{prop:M-y}
    Let $a(k)$, $b(k)$ ($k\in\mathbb R)$, and $\{\nu_j,c_j\}_1^N$ ($0<\nu_j<1/2$, 
    $c_j\in \ii\mathbb R$) be the spectral
    data determined by $u_0(x)$, $x\in[0,L]$
    through \eqref{sr} and \eqref{a-b}, where \eqref{sr} is considered at $t=0$.
    Assume that $\lim_{k\to 0} k a(k)\ne 0$ and $\lim_{k\to 0} k b(k)\ne 0$ (generic case).
   Let $m_0(x)=u_0(x)-u_{0xx}(x)$ and $\hat m_0(y)=m_0(x(y))$, where 
   $y(x)=\int_0^x \sqrt{m_0(\xi)+1}\dd\xi$. Then $\hat m_0(y)$ 
   can be given in terms of the solution of the following RH problem: for all 
   $y\in [0,\theta]$, 
   find a piece-wise (w.r.t. $\mathbb R$) meromorphic, $2\times 2$-valued 
   function $M^{(y)}(y,k)$ such that
   \begin{enumerate}
      \item The limiting values of $M^{(y)}(y,k)$ as $k$ approaches $\mathbb R$ from the upper 
    and lover half-planes ($M_+$ and $M_-$ respectively) are related by the jump 
    condition
    \begin{equation}\label{M-y-jump}
        M^{(y)}_+(y,k) = M^{(y)}_-(y,k)J^{(y)}_{\mathbb R}(y,k)\ \ \text{with}\ 
       J^{(y)}_{\mathbb R}(y,k)= e^{-iky\sigma_3}J_{0 \mathbb R}(k)  
       e^{iky\sigma_3},\ \ 
               k\in \mathbb R,
    \end{equation}
    where $J_{0 \mathbb R}(k) $ is given by \eqref{j0};
  \item 
  as $k\to 0$,
  \begin{equation}\label{M-y-k0}
\begin{aligned}
  M^{(y)}_+(y,k) & = \frac{i \alpha^{(y)}(y)}{k}
\begin{pmatrix}
    0 & 1 \\ 0 & -1
\end{pmatrix}+O(1),\\
M^{(y)}_-(y,k) & = -\frac{i \alpha^{(y)}(y)}{k}
\begin{pmatrix}
   1 & 0 \\ -1 & 0
\end{pmatrix}+O(1)
\end{aligned}
\end{equation}
with some (unspecified) $\alpha^{(y)}(y)>0$;
\item $M^{(y)}(y,k)\to I$ as $k\to\infty$;
\item 
$M^{(y)(1)}$ has simple poles at 
$k_j=i\nu_j$, $j=1,\dots,N$,
$M^{(y)(2)}$ has simple poles at $k=-i\nu_j$, and the residue conditions hold:
\begin{equation}\label{M-y-res}
\begin{aligned}
    \Res_{k=i\nu_j} M^{(y)(1)}(y,k) & = c_j e^{-2\nu_j y} M^{(y)(2)}(y,i\nu_j),\\
    \Res_{k=-i\nu_j} M^{(y)(2)}(y,k) & = \bar c_j e^{-2\nu_j y} M^{(y)(1)}(y,-i\nu_j);
\end{aligned}
\end{equation}
\item 
\begin{equation}\label{y-struc}
    M^{(y)}\left(y,\frac{i}{2}\right)=\frac{1}{2}\begin{pmatrix}
	\hat q_0(y)+\dfrac{1}{\hat q_0(y)} & \hat q_0(y)-\dfrac{1}{\hat q_0(y)} \\[5mm]
    \hat q_0(y)-\dfrac{1}{\hat q_0(y)} & \hat q_0(y)+\dfrac{1}{\hat q_0(y)}
	\end{pmatrix} \begin{pmatrix}
	\hat f_0(y) & 0 \\ 0 & \hat f_0^{-1}(y)	\end{pmatrix},
	\end{equation}
    with some (unspecified) $ \hat q_0(y)>0$ and $\hat f_0(y)>0$.
       \end{enumerate}
       Namely, $m_0(y)=\hat q_0^4(y)+1$.
\end{proposition}

\begin{proposition}\label{prop:con_pr}
    The solutions $\hat M(y,0,k)$ and $M^{(y)}(y,k)$ of the RH problems 
    \eqref{hatM-jump}-\eqref{sigma_j_hor} and \eqref{M-y-jump}-\eqref{y-struc} respectively
    can be related as follows:
    \begin{equation}
        \label{y-ini}
         M^{(y)}(y,k)=\hat{M} (y,0,k)\begin{cases}
\begin{pmatrix}
   1&-G_1(k)\eul^{-2iky}\\
 0&1  
\end{pmatrix},\quad 
k\in\{ \Im k > 0, |k|>\frac{1}{2},|k-\frac{\ii}{2}|>\epsilon \},\\

\begin{pmatrix}
   1&-\tilde G_1(k)\eul^{-2iky}\\
 0&1  
\end{pmatrix}\begin{pmatrix}
    \eul^{\ii k(L-\theta)}&0\\0& \eul^{-\ii k(L-\theta)}
\end{pmatrix},
\quad k\in\{  |k|>\frac{1}{2},|k-\frac{\ii}{2}|<\epsilon \},\\

\begin{pmatrix}
   1&0\\
 -G(k)\eul^{2iky}&1  
\end{pmatrix}
,\quad k\in\{\Im k > 0,  |k|<\frac{1}{2}, |k-\frac{\ii}{2}|>\epsilon\},\\

\begin{pmatrix}
   1&0\\
 -\tilde G(k)\eul^{2iky}&1  
\end{pmatrix}\begin{pmatrix}
    \eul^{\ii k(L-\theta)}&0\\0 &\eul^{-\ii k(L-\theta)}
\end{pmatrix}
,\quad k\in\{ |k|<\frac{1}{2}, |k-\frac{\ii}{2}|<\epsilon\},\\

\begin{pmatrix}
   1&-G^*(k)\eul^{-2iky}\\
 0&1  
\end{pmatrix}
,\quad k\in\{ \Im k < 0, |k|<\frac{1}{2},  |k+\frac{\ii}{2}|>\epsilon\},\\

\begin{pmatrix}
   1&-\tilde G^*(k)\eul^{-2iky}\\
 0&1  
\end{pmatrix}\begin{pmatrix}
    \eul^{\ii k(L-\theta)}&0\\0 &\eul^{-\ii k(L-\theta)}
\end{pmatrix}
,\quad k\in\{  |k|<\frac{1}{2}, |k+\frac{\ii}{2}|<\epsilon\},\\

\begin{pmatrix}
   1&0\\
 -G_1^*(k)\eul^{2iky}&1  
\end{pmatrix}
,\quad k\in\{ \Im k < 0, |k|>\frac{1}{2}, |k+\frac{\ii}{2}|>\epsilon\},\\

\begin{pmatrix}
   1&0\\
 -\tilde G_1^*(k)\eul^{2iky}&1  
\end{pmatrix}\begin{pmatrix}
    \eul^{\ii k(L-\theta)}&0\\0 &\eul^{-\ii k(L-\theta)}
\end{pmatrix}
,\quad k\in\{ |k|>\frac{1}{2}, |k+\frac{\ii}{2}|<\epsilon\}.
\end{cases}
\end{equation}
Particularly, it follows that $\hat q(y,0,k)=\hat q_0(y)$ and thus $\hat m(y,0,k)=\hat m_0(y)$.
\end{proposition}
\begin{proof}
    The proof consists in the demonstration that the r.h.s. of \eqref{y-ini}
    meets all requirements of the RH problem for  $M^{(y)}(y,k)$ 
    in Proposition \ref{prop:M-y}.

    \medskip

\noindent (1) The jump conditions match by construction.

\medskip

\noindent (2)
In the generic case we have $\rho\ne 0$ in \eqref{ab-0}, which implies that
(i) $\frac{b^*}{a}(0)=1$ and (ii) ${\mathfrak R}^*(0)=-1$ (follows from \eqref{calR}).
Consequently, $G(0)=0$, and the conditions at $k=0$ match.

\medskip

\noindent (3) Noticing that $G_1(k)\to 0$ as $k\to\infty$ in ${\mathbb C}_+$
since 
 $b(k)\to 0$ in ${\mathbb C}_+$ and $\mathfrak R(\infty)=0$, and that $\eul^{-2iky}$ is bounded in ${\mathbb C}_+$ for $y\ge 0$, the normalization conditions match.

\medskip

\noindent (4) The structures at $k=\frac{\ii}{2}$ match because 
\[
\tilde b\left(\frac{i}{2}\right) = \tilde b^*\left(\frac{i}{2}\right)
= \tilde{\mathfrak R}\left(\frac{i}{2}\right) =
\tilde{\mathfrak R}^*\left(\frac{i}{2}\right) =0,
\]
which implies that $\tilde G\left(\frac{i}{2}\right)=\tilde G_1\left(\frac{i}{2}\right)=0$.

\medskip

\noindent (5)
Finally, let's check that the residue condition match. 

Let's denote by $\{\mu_j\}$ the poles of $\mathfrak R$ in $D_1=\{k: \Im k>0, |k|>\frac{1}{2}, |k-\frac{i}{2}|>\varepsilon\}$. For $k\in D_1$ we have
that 
$M^{(y)}(y,k)= \hat M (y,0,k)
\begin{pmatrix}
   1&-G_1(k)\eul^{-2iky}\\
 0&1  
\end{pmatrix}$, particularly, that the r.h.s. above is not singular at $\mu_j$.
Indeed,
\[
\hat M (y,0,k)
\begin{pmatrix}
   1&-G_1(k)\eul^{-2iky}\\
 0&1  
\end{pmatrix}=(\hat M^{(1)} (y,0,k) \ \ \  -\hat M^{(1)} (y,0,k)G_1(k)
\eul^{-2iky}+\hat M^{(2)} (y,0,k)).
\]
Taking into account that  $\Res_{\mu_j}G_1 = a^2(\mu_j)\Res \mathfrak R $ (see \eqref{G-G1}) and the residue condition \eqref{res_hatM-a},
the singularity of the second column cancels.

Now consider $D_2=\{k: \Im k>0, |k|<\frac{1}{2}, |k-\frac{i}{2}|>\varepsilon\}$,
where we want to have that $M^{(y)}(y,k)= \hat M (y,0,k)
\begin{pmatrix}
   1& 0 \\ -G(k)\eul^{2iky} & 1
\end{pmatrix}$.
Here we have two possibilities: (i) either some zero $\bar \mu_j$ of ${\mathfrak R }^*$
coincides with a zero $\ii \nu_j$ or (ii) they are all different.
In the latter case,  we want to have
that 
$\hat M (y,0,k)
\begin{pmatrix}
   1& 0 \\ -G(k)\eul^{2iky} & 1
\end{pmatrix}$ is not singular at $\bar \mu_j$.
Indeed,
\[
\hat M (y,0,k)
\begin{pmatrix}
   1& 0 \\ -G(k)\eul^{2iky} & 1
\end{pmatrix} = 
(\hat M^{(1)} (y,0,k)-\hat M^{(2)}(y,0,k)G(k)\eul^{2iky} \ \ \ 
\hat M^{(2)} (y,0,k))
\]
Taking into account that  $\Res_{\bar\mu_j}G = e^{2\nu_j\theta}\Res_{\bar\mu_j} {\mathfrak R}^* $ (see \eqref{G-G1}) and the residue condition \eqref{res_hatM-b},
the singularity of the first column cancels.

On the other hand, at $k=\ii\nu$ the first column is singular due to the singularity
of $G$, which, in this case, is $\Res_{\ii \nu_j}G = -\frac{b^*(\ii\nu_j)}{\dot a(\ii\nu_j)}
= -c_j$, and thus the corresponding residue condition take the form in \eqref{M-y-res}.

Finally, consider the case with $\ii\nu_j=\bar\mu_j$ for some $j$. Then 
\[
\Res_{\ii \nu_j}G = e^{2\nu_j\theta}\Res_{\ii \nu_j}{\mathfrak R}^*-\frac{b^*(\ii\nu_j)}{\dot a(\ii\nu_j)} = e^{2\nu_j\theta}\Res_{\ii \nu_j}{\mathfrak R}^* - c_j
\]
and thus, again taking into account the residue condition \eqref{res_hatM-b},
the residue of the first column becomes $c_j\ee^{-2\nu_j y}\hat M^{(2)}(y,0,\ii\nu_j)$,
which is consistent with the residue conditions \eqref{M-y-res}.

The singularities in $D_3=\bar D_2$ and $D_4=\bar D_1$ can be treated in a similar way.
\end{proof}

\begin{remark}\label{rem:const_t_0}
    Notice that \eqref{y-ini} allows us to explicitly construct the solution of RH problem \eqref{hatM-jump}-\eqref{sigma_j_hor} $\hat M(y,t,k)$ at $t=0$ from the solution  of RH problem \eqref{M-y-jump}-\eqref{y-struc} $M^{(y)}(y,k)$, whose existence is guaranteed by construction. 
\end{remark}

\subsection{Verification of periodicity}
The idea of verification of periodicity (Step 3) is as follows: transform 
two RH problems, parametrized by $t$, which are 
the master RH problem for respectively $y=0$ and for $y=\theta$,
to two RH problems having the same data.

For this purpose we introduce 
\begin{equation*}
    P(t,k)=\begin{cases}
\begin{pmatrix}
    a & 0 \\ 0 & a^{-1}
\end{pmatrix}, & 
k\in\{ \Im k > 0, |k|>\frac{1}{2}, |k-\frac{\ii}{2}| >\epsilon \},\\
\begin{pmatrix}
    \tilde a & 0 \\ 0 & {\tilde a}^{-1}
\end{pmatrix}, & 
k\in\{ \Im k > 0, |k|>\frac{1}{2}, |k-\frac{\ii}{2}| <\epsilon \},\\
\begin{pmatrix}
    {a^*}^{-1} & 0 \\ 0 & a^*
\end{pmatrix}, & 
k\in\{ \Im k < 0, |k|>\frac{1}{2}, |k+\frac{\ii}{2}| >\epsilon \},\\
\begin{pmatrix}
    {\tilde a^*}^{-1} & 0 \\ 0 & \tilde a^*
\end{pmatrix}, & 
k\in\{ \Im k < 0, |k|>\frac{1}{2}, |k+\frac{\ii}{2}| <\epsilon \},\\
\begin{pmatrix}
   a-b\tilde{\mathfrak R}^*&-b\eul^{-\frac{\ii k t}{\lambda}}\eul^{2\ii k(L-\theta)}\\
 0&a^*-b^* \tilde{\mathfrak R}  
\end{pmatrix} \ee^{\ii k(L-\theta)\sigma_3}, &
 k\in\{ \Im k > 0, |k|<\frac{1}{2},|k-\frac{\ii}{2}| <\epsilon\},\\

\begin{pmatrix}
   a-b\mathfrak R^*&-b\eul^{-\frac{\ii k t}{\lambda}}\\
 0&a^*-b^*\mathfrak R  
\end{pmatrix}, & 
 k\in\{ \Im k > 0, |k|<\frac{1}{2},|k-\frac{\ii}{2}| >\epsilon\},\\

\begin{pmatrix}
   a-b\mathfrak R^*&0\\
 -b^*\eul^{\frac{\ii k t}{\lambda}}&a^*-b^*\mathfrak R   
\end{pmatrix}, & 
 k\in\{ \Im k < 0, |k|<\frac{1}{2},|k+\frac{\ii}{2}| >\epsilon\},\\

\begin{pmatrix}
   a-b\tilde{\mathfrak R}^*&0\\
 -b^*\eul^{\frac{\ii k t}{\lambda}}\eul^{-2\ii k(L-\theta)}&a^*-b^*\tilde{\mathfrak R}   
\end{pmatrix} \ee^{\ii k(L-\theta)\sigma_3}, &
k\in\{ \Im k < 0, |k|<\frac{1}{2},|k+\frac{\ii}{2}| <\epsilon\}

\end{cases}
\end{equation*}
and
\begin{equation*}
    \hat P(t,k)=\begin{cases}
    \begin{pmatrix}
    a & 0 \\ 0 & a^{-1}
\end{pmatrix}
\begin{pmatrix}
   a^*-b^*\mathfrak R&0\\
 b^*\eul^{2\ii k\theta}\eul^{\frac{\ii k t}{\lambda}}&a-b\mathfrak R^*  
\end{pmatrix}, & 
 k\in\{ \Im k > 0, |k|>\frac{1}{2}, |k-\frac{\ii}{2}| >\epsilon\},\\
\begin{pmatrix}
    \tilde a & 0 \\ 0 & {\tilde a}^{-1}
\end{pmatrix}
\begin{pmatrix}
   a^*-b^*\tilde {\mathfrak R}&0\\
 b^*\eul^{2\ii k\theta}\eul^{\frac{\ii k t}{\lambda}}&a-b\tilde {\mathfrak R}^*  
\end{pmatrix}, & 
 k\in\{ \Im k > 0, |k|>\frac{1}{2}, |k-\frac{\ii}{2}| <\epsilon\},\\
\begin{pmatrix}
    {\tilde a^*}^{-1} & 0 \\ 0 & \tilde a^*
\end{pmatrix}
\begin{pmatrix}
   a^*-b^*\tilde {\mathfrak R}& b\eul^{-2\ii k\theta}\eul^{-\frac{\ii k t}{\lambda}}\\
0&a-b\tilde {\mathfrak R}^*  
\end{pmatrix}, & 
k\in\{ \Im k < 0, |k|>\frac{1}{2},|k+\frac{\ii}{2}| <\epsilon\},\\
\begin{pmatrix}
    { a^*}^{-1} & 0 \\ 0 &  a^*
\end{pmatrix}
\begin{pmatrix}
   a^*-b^*\mathfrak R& b\eul^{-2\ii k\theta}\eul^{-\frac{\ii k t}{\lambda}}\\
0&a-b\mathfrak R^*  
\end{pmatrix}, & 
k\in\{ \Im k < 0, |k|>\frac{1}{2},|k+\frac{\ii}{2}| >\epsilon\},\\
 \ee^{\ii k(L-\theta)\sigma_3}, & 
k\in\{ |k|<\frac{1}{2}\}\cap\left(\{|k-\frac{\ii}{2}|<\epsilon\}\cup\{|k+\frac{\ii}{2}|<\epsilon\}\right)
\\
I, & 
 k\in\{ |k|<\frac{1}{2}\}\setminus \left(\{|k-\frac{\ii}{2}|<\epsilon\}\cup\{|k+\frac{\ii}{2}|<\epsilon\}\right)
\end{cases}
\end{equation*}
and define $M^{(t)}(t,k):=\hat{M}(0,t,k)P(t,k)$ and 
$\hat M^{(t)}(t,k):=\hat{M}(\theta,t,k)\hat P(t,k)$. 
Our goal is to show that $M^{(t)}$ and $\hat M^{(t)}$ satisfy the same
conditions characterizing the respective RH problem.

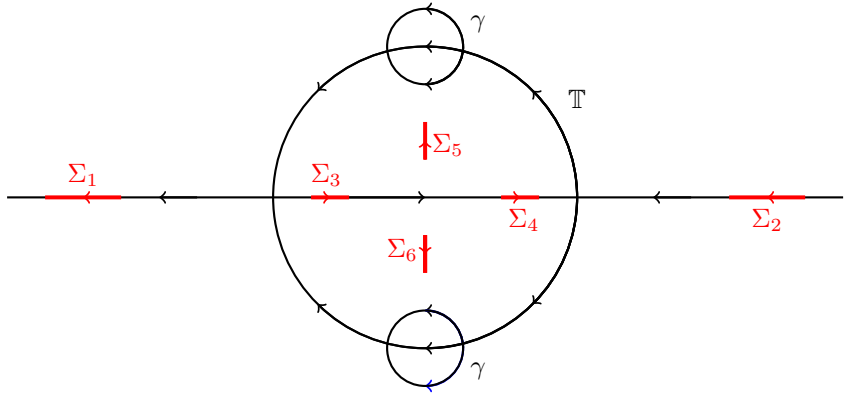
\begin{figure}
    \centering
    \begin{tikzpicture}
\draw[thick] (0,0) circle (2);

\draw[thick] (-5.5,0) -- (5.5,0);
 
\draw[->, thick] (2,0) arc[start angle=0, end angle=90, radius=2];
 
\draw[->, thick] (2,0) arc[start angle=360, end angle=270, radius=2];
    
\draw[<-, thick] (3,0) -- (3.5,0);
 
\draw[<-, thick] (-3.5,0) -- (-3,0);
  
\draw[->, thick] (-1,0) -- (0,0);
    
\draw[line width=1.5pt, red] (0, -0.5) -- (0, -1);

\draw[thick, red, ->] (0, -0.5) -- (0, -0.75);
             
\draw[line width=1.5pt, red] (0, 0.5) -- (0, 1);
        
\draw[thick, red, ->] (0, 0.5) -- (0, 0.75);

\draw[line width=1.5pt, red] (0, -0.5) -- (0, -1);

\draw[line width=1.5pt, red] (-5, 0) -- (-4, 0);
        
\draw[thick, red, <-] (-4.5, 0) -- (-4, 0);

\draw[line width=1.5pt, red] (4, 0) -- (5, 0);
        
\draw[thick, red, <-] (4.5, 0) -- (5, 0);

\draw[line width=1.5pt, red] (-1.5, 0) -- (-1, 0);
        
\draw[thick, red, ->] (-1.5, 0) -- (-1.25, 0);

\draw[line width=1.5pt, red] (1, 0) -- (1.5, 0);
\draw[thick, red, ->] (1, 0) -- (1.25, 0);

\draw[->, thick] (0.5,2) arc[start angle=0, end angle=-90, radius=0.5];

\draw[<-, thick] (0,2.5) arc[start angle=90, end angle=-90, radius=0.5];

\draw[->, thick] (0.5,-2) arc[start angle=0, end angle=90, radius=0.5];

\draw[->, thick, blue] (0,-1.5) arc[start angle=90, end angle=-90, radius=0.5];

\draw[thick] (0,2) circle (0.5);

\draw[thick] (0,-2) circle (0.5);

\draw[->, thick] (2,0) arc[start angle=0, end angle=45, radius=2];
    
\draw[->, thick] (2,0) arc[start angle=0, end angle=135, radius=2];
        
\draw[->, thick] (2,0) arc[start angle=0, end angle=-45, radius=2];

\draw[->, thick] (2,0) arc[start angle=0, end angle=-135, radius=2];

\node at (0.7, 2.3) {$\gamma$};

\node at (0.7, -2.3) {$\gamma$};

\node at (2, 1.3) {$\mathbb{T}$};

\node at (-4.5, 0.3) {$\red{\Sigma_1}$};

\node at (4.5, -0.3) {$\red{\Sigma_2}$};

\node at (-1.3, 0.3) {$\red{\Sigma_3}$};

\node at (1.3, -0.3) {$\red{\Sigma_4}$};

\node at (0.3, 0.7) {$\red{\Sigma_5}$};

\node at (-0.3, -0.7) {$\red{\Sigma_6}$};

\end{tikzpicture}
    \caption{Contour for $M^{(t)}$ and $\hat M^{(t)}$}
    \label{fig:hatSigma_t}
\end{figure}

The following proposition establishes that $M{(t)}(t,k)$ and $\hat M{(t)}(t,k)$ satisfy the same jump condition.

\begin{proposition}\label{prop:jump}
Let $\Sigma$ be oriented as in Figure \ref{fig:hatSigma_t}, and define the jump matrices $J^{(t)}(t,k)$ and $\hat J^{(t)}(t,k)$ by the relations
\[
M_+^{(t)}(t,k) = M_-^{(t)}(t,k)J^{(t)}(t,k), \quad 
\hat M_+^{(t)}(t,k) =\hat M_-^{(t)}(t,k)\hat J^{(t)}(t,k),
\]
for $k \in \Sigma$. Then the two jump matrices coincide. Moreover, \[J^{(t)}(t,k) = \hat J^{(t)}(t,k)=\ee^{-ik \frac{t}{2\lambda}\sigma_3 } J^{(t)}_0(k) 
\ee^{ik \frac{t}{2\lambda}\sigma_3}\] with
\[
J^{(t)}_0(k)=\begin{cases}
    \begin{pmatrix}
    1-{\mathfrak R}^* {\mathfrak R} & -{\mathfrak R} \\
    {\mathfrak R}^* & 1
\end{pmatrix}, \quad k\in \{\Im k =0, k\notin\Sigma_j \}\cup \{\Im k >0, |k|=\frac{\ii}{2}, |k-\frac{\ii}{2}|<\epsilon\},\\
    \begin{pmatrix}
    1-{\mathfrak R_-}^* {\mathfrak R_+} & -{\mathfrak R_+} \\
    {\mathfrak R_-}^* & 1
\end{pmatrix}, \quad k\in \{\Im k >0,\Re k=0, k\in\Sigma_j \},\\
\begin{pmatrix}
		1  & 0 \\
		{{\mathfrak R}^*}_+(k)-{{\mathfrak R}^*}_-(k)  & 1
	\end{pmatrix}, \quad k\in \{\Re k=0\}\cup\Sigma_j,\\
\begin{pmatrix}
    1-{\mathfrak {\tilde R}}^* {\mathfrak {\tilde R}} & -{\mathfrak {\tilde R}} \\
    {\mathfrak {\tilde R}}^* & 1
\end{pmatrix}, \quad k\in \{\Im k >0, |k|=\frac{\ii}{2}, |k-\frac{\ii}{2}|<\epsilon\},\\
\begin{pmatrix}
    1& 0\\
   \mathfrak R^*(k)-\tilde{\mathfrak R}^*(k) & 1
\end{pmatrix},\quad k\in \{|k-\frac{\ii}{2}|=\epsilon,|k|<\frac{1}{2}\},\\
\begin{pmatrix}
    1 & \tilde{\mathfrak R}(k)-\mathfrak R(k)\\
    0& 1
\end{pmatrix},\quad k\in \{|k-\frac{\ii}{2}|=\epsilon,|k|>\frac{1}{2}\}.
\end{cases}
\]

\end{proposition}

\begin{proof}
    Introducing the orientation of $\Sigma$ as in Figure \ref{fig:hatSigma_t}
 and the jump matrices 
 $J^{(t)}(t,k)$ and $\hat J^{(t)}(t,k)$ by 
\[
M_+^{(t)}(t,k) = M_-^{(t)}(t,k)J^{(t)}(t,k), \quad 
\hat M_+^{(t)}(t,k) =\hat M_-^{(t)}(t,k)\hat J^{(t)}(t,k),
\]
let's check that  $J^{(t)}(t,k) = \hat J^{(t)}(t,k)$. By the definitions of $M^{(t)}$ and $\hat M^{(t)}$, we have
\begin{align}
    M^{(t)}_+&=\hat M_+(0)P_+=\hat M_-(0)\hat J(0)P_+= M^{(t)}_-P_-^{-1}\hat J(0)P_+,\\
    \hat M^{(t)}_+&=\hat M_+(\theta)\hat P_+=\hat M_-(\theta)\hat J(\theta)\hat P_+= \hat M^{(t)}_-\hat P_-^{-1}\hat J(\theta)\hat P_+,
\end{align}
and, thus, we need to check that 
\begin{equation}\label{per_jump}
    P_-^{-1}(t,k)\hat J(0,t,k)P_+(t,k)=\hat P_-^{-1}(t,k)\hat J(\theta,t,k)\hat P_+(t,k),
\end{equation}
where $\hat J(0,t,k)=\ee^{-ik \frac{t}{2\lambda}\sigma_3 }\hat J_0(k) 
\ee^{ik \frac{t}{2\lambda}\sigma_3 }$ and 
$\hat J(\theta,t,k) = \ee^{-ik (\theta +\frac{t}{2\lambda})\sigma_3 }\hat J_0(k) 
\ee^{ik (\theta +\frac{t}{2\lambda})\sigma_3 }$ with $\hat J_0(k)$
given by \eqref{j0-yt}, \eqref{sigma_j_ver}, and \eqref{sigma_j_hor}.

\noindent (1) For $k\in  \{ \Im k > 0, |k|=\frac{1}{2}, |k-\frac{\ii}{2}| >\epsilon\}$, we have
\begin{align}
   P_-^{-1}&=\begin{pmatrix}
    a^{-1}&0\\
    0&a
\end{pmatrix},\quad P_+=\begin{pmatrix}
   a-b\mathfrak R^*&-b\eul^{-\frac{\ii k t}{\lambda}}\\
 0&a^*-b^*\mathfrak R  
\end{pmatrix}, \\
\hat P_-^{-1}&=        
\begin{pmatrix}
 a-b\mathfrak R^*    &0\\
 -b^*\eul^{2\ii k\theta}\eul^{\frac{\ii k t}{\lambda}}&a^*-b^*\mathfrak R
\end{pmatrix}\begin{pmatrix}
    a^{-1} & 0 \\ 0 & a
\end{pmatrix},\quad \hat P_+= I.
\end{align}
Thus, to establish \eqref{per_jump} in this case, it suffices to verify that

\begin{equation}\label{check1}
    \begin{pmatrix}
    a^{-1}&0\\0&a
\end{pmatrix}\hat{J_0}
\begin{pmatrix}
    a-b\mathfrak R^*&-b\\
    0& a^*-b^*\mathfrak R
\end{pmatrix}=\eul^{-\ii k\theta\sigma_3}\begin{pmatrix}
    a-b\mathfrak R^*&0\\
    -b^*& a^*-b^*\mathfrak R
\end{pmatrix}\begin{pmatrix}
    a^{-1}&0\\0&a
\end{pmatrix} \hat J_0 e^{\ii k\theta\sigma_3},
\end{equation}
where $\hat J_0(k)$ is given by 
\[
\hat J_0(k) = \begin{pmatrix}
          1& -{G}_1(k)\\
         0&1   
\end{pmatrix}
\begin{pmatrix}
          1& 0\\
         {G}(k)&1   
\end{pmatrix}.
\]
Then, using \eqref{det_rel}, \eqref{calR}, and  \eqref{propR5}, 
straightforward calculations show that the both sides of \eqref{check1} equals
$\begin{pmatrix}
    1-{\mathfrak R}^* {\mathfrak R} & -{\mathfrak R} \\
    {\mathfrak R}^* & 1
\end{pmatrix}
=\begin{pmatrix}
    1 & - {\mathfrak R} \\  0  & 1
\end{pmatrix}
\begin{pmatrix}
    1 & 0 \\  {\mathfrak R}^*   & 1
\end{pmatrix}
$.

None that
    the fact that at this part of the contour we arrived at 
    \begin{equation}\label{jump-t}
        J^{(t)}(t,k) = \hat J^{(t)}(t,k) = 
    \ee^{-ik \frac{t}{2\lambda}\sigma_3 }
    \begin{pmatrix}
    1-{\mathfrak R}^*(k) {\mathfrak R}(k) & -{\mathfrak R}(k) \\
    {\mathfrak R}^*(k) & 1
    \end{pmatrix}
    \ee^{ik \frac{t}{2\lambda}\sigma_3 }
    \end{equation}
  is not surprising: checking periodicity can be viewed as a particular 
  case of checking general boundary conditions, by
  transforming the master RH problem taken at the boundary values of $x$
  to those associated with the inverse problem for the $t$ equation 
  of the Lax pair, where the jump matrix should resemble that in the whole 
  line case (cf. \eqref{j0}), with the scattering coefficient $r$ replaced by the 
  scattering coefficients for the $t$ equation, i.e., $\frac{B}{A}$
  or $\frac{B_1}{A_1}$, c.f. \cite{BS04,BS08_2,BFS06}.

Particularly, one can expect the same form of \eqref{jump-t} for 
some other parts of $\Sigma$. 

Indeed, consider the jumps on $k \in \{ \Im k = 0,\, |k| > \tfrac{1}{2} \}$, where we have to distinguish two cases depending on whether $k \in \Sigma_j$ or $k \notin \Sigma_j$.  
  
\noindent (2) If $k\in  \{ \Im k = 0, |k|>\frac{1}{2}\}\cap \Sigma_j$, we have
\begin{align}\label{per_ver_j_2_1}
  P_-^{-1} & =  \begin{pmatrix}
   a^{-1}&0\\
 0&a  
\end{pmatrix},
 \quad P_+=\begin{pmatrix}
    a^{*-1}&0\\
    0&a^{*}
\end{pmatrix},\\\label{per_ver_j_2_2}
\hat P_-^{-1}&=   
\begin{pmatrix}
  a-b\mathfrak {R^*}_+  & 0\\
 -b^*\eul^{2\ii k\theta}\eul^{\frac{\ii k t}{\lambda}}&
 a^*-b^*\mathfrak R_+ 
 \end{pmatrix}\begin{pmatrix}
   a^{-1}&0\\
 0&a  
\end{pmatrix},\quad 
 \hat P_+=
 \begin{pmatrix}
    { a^*}^{-1} & 0 \\ 0 &  {a^*}
\end{pmatrix}
\begin{pmatrix}
 a^*-b^*\mathfrak R_-  & b\eul^{-2\ii k\theta}\eul^{-\frac{\ii k t}{\lambda}}\\
0& a-b\mathfrak {R^*}_-  
\end{pmatrix},
\end{align}
and
$\hat J_0(k)$ is given by 
\begin{equation}
\label{per_ver_j_2_3}    
\hat J_0(k) = \begin{pmatrix}
          1& 0 \\ -{G}^*_{1-}(k) & 1
\end{pmatrix}
\begin{pmatrix}
          1 - |r(k)|^2& r^*(k)\\
         -r(k)&1  
\end{pmatrix}
\begin{pmatrix}
          1& G_{1+}(k)\\
         0&1   
\end{pmatrix}.
\end{equation}
Then, using \eqref{propR8}, \eqref{propR5} and \eqref{propR9}, straightforward calculations show that the jumps 
\(J^{(t)}(t,k)\) and \(\hat J^{(t)}(t,k)\) are equal to
\[J^{(t)}(t,k)=\hat J^{(t)}(t,k)=\ee^{-ik \frac{t}{2\lambda}\sigma_3 }\begin{pmatrix}
    1-{{\mathfrak R}^*}_- {\mathfrak R}_+ & -{\mathfrak R}_+ \\
    {{\mathfrak R}^*}_- & 1
\end{pmatrix}\ee^{ik \frac{t}{2\lambda}\sigma_3 }
.\]

\noindent (3) If $k\in  \{ \Im k = 0, |k|>\frac{1}{2}\}\setminus \Sigma_j$, the equalities \eqref{per_ver_j_2_1} -- \eqref{per_ver_j_2_3} remain valid with $\mathfrak R_+$ and $\mathfrak R_-$ replaced by $\mathfrak R$ (and, correspondingly, $G_+$ and $G_-$ replaced by $G$, and $G_{1+}$ and $G_{1-}$  replaced by $G_1$) since $\mathfrak R$ has no jump across this part of the contour. In this case, using \eqref{calR}, we conclude that the jumps 
\(J^{(t)}(t,k)\) and \(\hat J^{(t)}(t,k)\) are equal to
\[J^{(t)}(t,k)=\hat J^{(t)}(t,k)=\ee^{-ik \frac{t}{2\lambda}\sigma_3 }\begin{pmatrix}
    1-{\mathfrak R}^* {\mathfrak R} & -{\mathfrak R} \\
    {\mathfrak R}^* & 1
\end{pmatrix}\ee^{ik \frac{t}{2\lambda}\sigma_3 }
.\]

\noindent (4) For \(k \in \{ \Im k = 0,\, |k| < \tfrac{1}{2} \}\), an analogous argument shows that the jumps 
\(J^{(t)}(t,k)\) and \(\hat J^{(t)}(t,k)\) coincide.  
Moreover, for \(k \in \{ \Im k = 0,\, |k| > \tfrac{1}{2} \}\), they are given by  
\[
J^{(t)}(t,k) = \hat J^{(t)}(t,k) =
\begin{cases}
\ee^{-ik \frac{t}{2\lambda}\sigma_3 }\begin{pmatrix}
    1 - \mathfrak{R}^* \mathfrak{R} & -\mathfrak{R} \\
    \mathfrak{R}^* & 1
\end{pmatrix}\ee^{ik \frac{t}{2\lambda}\sigma_3 },
& k \notin \Sigma_j, \\[2ex]
\ee^{-ik \frac{t}{2\lambda}\sigma_3 }\begin{pmatrix}
    1 - {\mathfrak{R}^*}_+ \mathfrak{R}_- & -{\mathfrak{R}}_- \\
    {\mathfrak{R}^*}_+ & 1
\end{pmatrix}\ee^{ik \frac{t}{2\lambda}\sigma_3 },
& k \in \Sigma_j .
\end{cases}
\] 

\noindent (5) For $k\in  \{ \Re k = 0\}\cap \Sigma_j$, we have 
\begin{align}
   P_-^{-1}&=\begin{pmatrix}
  a^*-b^*\mathfrak R_- &b\eul^{-\frac{\ii k t}{\lambda}}\\
 0&a-b\mathfrak {R^*}_-  
\end{pmatrix},\quad P_+=\begin{pmatrix}
   a-b\mathfrak {R^*}_+&-b\eul^{-\frac{\ii k t}{\lambda}}\\
 0&a^*-b^*\mathfrak R_+  
\end{pmatrix}, \\
\hat P_-^{-1}&=        
I,\quad \hat P_+= I.
\end{align}
Hence, in order to establish \eqref{per_jump}, it suffices to verify that
\begin{equation}\label{per_ver_j_3_1}
\begin{pmatrix}
  a^*-b^*\mathfrak R_- &b\\
 0&a-b\mathfrak {R^*}_-  
\end{pmatrix}\hat J_0 \begin{pmatrix}
   a-b\mathfrak {R^*}_+&-b\\
 0&a^*-b^*\mathfrak R_+  
\end{pmatrix}=\eul^{-\ii k \theta \sigma_3}\hat J_0 \eul^{\ii k \theta \sigma_3}
\end{equation}
where $\hat J_0(k)$ is given by
\[
\hat J_0(k)=\begin{pmatrix}
		1  & 0 \\
		\ee^{-2\ii k \theta}({{\mathfrak R}^*}_+(k)-{{\mathfrak R}^*}_-(k))  & 1
	\end{pmatrix}
\]
Then, using \eqref{propR5}, \eqref{propR6} and \eqref{propR9},
straightforward computation show that the both sides of \eqref{per_ver_j_3_1} equal $\begin{pmatrix}
		1  & 0 \\
		{{\mathfrak R}^*}_+(k)-{{\mathfrak R}^*}_-(k)  & 1
	\end{pmatrix}$. 

\noindent (6) For $k\in  \{  |k|>\frac{1}{2}, |k-\frac{\ii}{2}| =\epsilon\}$, we have
\begin{align}
   P_-^{-1}&=\begin{pmatrix}
    a^{-1} & 0 \\ 0 & a
\end{pmatrix},\quad P_+=\begin{pmatrix}
     {\tilde a}& 0 \\ 0 & {\tilde a}^{-1}
\end{pmatrix}, \\
\hat P_-^{-1}&= 
\begin{pmatrix}
 a-b\mathfrak R^* &0\\
- b^*\eul^{2\ii k\theta}\eul^{\frac{\ii k t}{\lambda}}& a^*-b^*\mathfrak R  
\end{pmatrix}
\begin{pmatrix}
    a^{-1} & 0 \\ 0 & a
\end{pmatrix},\quad \hat P_+= 
\begin{pmatrix}
    {\tilde a} & 0 \\ 0 & \tilde a^{-1}
\end{pmatrix}\begin{pmatrix}
 a^*-b^*\tilde {\mathfrak R}  &0\\b^*\eul^{2\ii k\theta}\eul^{\frac{\ii k t}{\lambda}}&a-b\tilde {\mathfrak R}^*  
\end{pmatrix}.
\end{align}
Thus, to establish \eqref{per_jump} in this case, it suffices to verify that
\begin{equation}\label{per_ver_j_3}
\begin{aligned}
& \begin{pmatrix}
    a^{-1} & 0 \\ 0 & a
\end{pmatrix} \hat {J_0 }^{-1} \begin{pmatrix}
    {\tilde a} & 0 \\ 0 & \tilde a^{-1}
\end{pmatrix} = \\
& \eul^{-\ii k \theta\sigma_3}  
\begin{pmatrix}
  a-b\mathfrak R^*&0\\
 -b^*&  a^*-b^*\mathfrak R 
\end{pmatrix} \begin{pmatrix}
    a^{-1} & 0 \\ 0 & a
\end{pmatrix} \hat {J_0 }^{-1} \begin{pmatrix}
    {\tilde a} & 0 \\ 0 & \tilde a^{-1}
\end{pmatrix} \begin{pmatrix}
a^*-b^*\tilde {\mathfrak R}   &0\\ b^*&a-b\tilde {\mathfrak R}^*   
\end{pmatrix} \eul^{\ii k \theta\sigma_3},
\end{aligned}
\end{equation}
where
\begin{equation}
  \hat J_0(k)=  
\begin{pmatrix}
    \ee^{\ii k(L-\theta)}& G_2(k)\\
   0 &\ee^{-\ii k(L-\theta)}
\end{pmatrix}
\end{equation}
with $G_2$ defined in \eqref{G2}. In the case with $\mathfrak R = \tilde{\mathfrak R}$, both sides of \eqref{per_ver_j_3} are equal to $I$. In the case with $\mathfrak R \neq \tilde{\mathfrak R}$, using \eqref{propR6} and \eqref{calR}, straightforward calculations show that the both sides of \eqref{per_ver_j_3} are equal to $\begin{pmatrix}
    1 & \tilde{\mathfrak R}(k)-\mathfrak R(k)\\
    0& 1
\end{pmatrix}$.

\noindent (7) For $k\in  \{  |k|<\frac{1}{2}, |k-\frac{\ii}{2}| =\epsilon\}$, we have
\begin{align}
   P_-^{-1}&=\ee^{-\ii k(L-\theta)\sigma_3}\begin{pmatrix}
 a^*-b^* \tilde{\mathfrak R} &b\eul^{-\frac{\ii k t}{\lambda}}\eul^{2\ii k(L-\theta)}\\
 0& a-b\tilde{\mathfrak R}^*  
\end{pmatrix},\quad P_+=\begin{pmatrix}
   a-b\mathfrak R^*&-b\eul^{-\frac{\ii k t}{\lambda}}\\
 0&a^*-b^*\mathfrak R  
\end{pmatrix}, \\
\hat P_-^{-1}&= 
\ee^{-\ii k(L-\theta)\sigma_3},\quad \hat P_+= 
I.
\end{align}
Thus, to establish \eqref{per_jump} in this case, it suffices to verify that
\begin{equation}\label{per_ver_j_4}
   \ee^{-\ii k(L-\theta)\sigma_3}\begin{pmatrix}
 a^*-b^* \tilde{\mathfrak R} &b\eul^{2\ii k(L-\theta)}\\
 0& a-b\tilde{\mathfrak R}^*  
\end{pmatrix}\hat J_0 \begin{pmatrix}
   a-b\mathfrak R^*&-b\eul^{-\frac{\ii k t}{\lambda}}\\
 0&a^*-b^*\mathfrak R  
\end{pmatrix} =\ee^{-\ii k(L-\theta)\sigma_3}\eul^{-\ii k\theta\sigma_3}\hat {J_0 }\eul^{\ii k\theta\sigma_3},
\end{equation}
where
\begin{equation}
  \hat J_0(k)=  
\begin{pmatrix}
 \ee^{\ii k(L-\theta)}&0\\
  {G}_3(k)&\ee^{-\ii k(L-\theta)}
\end{pmatrix}
\end{equation}
with $G_3$ defined in \eqref{G3}.  In the case with $\mathfrak R = \tilde{\mathfrak R}$, both sides of \eqref{per_ver_j_4} are equal to $\eul^{\ii k (L-\theta)\sigma_3}$. In the case with $\mathfrak R \neq \tilde{\mathfrak R}$, using \eqref{propR6} and \eqref{calR}, straightforward calculations show that the both sides of \eqref{per_ver_j_4} are equal to 
$
\begin{pmatrix}
    1& 0\\
   \mathfrak R^*(k)-\tilde{\mathfrak R}^*(k) & 1
\end{pmatrix}$.

\noindent (8) For $k\in  \{  |k|=\frac{1}{2}, |k-\frac{\ii}{2}| <\epsilon\}$, we have
\begin{align}
   P_-^{-1}&=\begin{pmatrix}
    \tilde a ^{-1}& 0 \\ 0 & {\tilde a}
\end{pmatrix},\quad P_+=\begin{pmatrix}
   a-b\tilde{\mathfrak R}^*&-b\eul^{-\frac{\ii k t}{\lambda}}\eul^{2\ii k(L-\theta)}\\
 0&a^*-b^* \tilde{\mathfrak R}  
\end{pmatrix}\ee^{\ii k(L-\theta)\sigma_3}, \\
\hat P_-^{-1}&= 
\begin{pmatrix}
 a-b\tilde {\mathfrak R}^*   &0\\
 -b^*\eul^{2\ii k\theta}\eul^{\frac{\ii k t}{\lambda}}&a^*-b^*\tilde {\mathfrak R} 
\end{pmatrix}\begin{pmatrix}
    \tilde a^{-1} & 0 \\ 0 & {\tilde a}
\end{pmatrix},\quad \hat P_+= 
\ee^{\ii k(L-\theta)\sigma_3}.
\end{align}
Thus, to establish \eqref{per_jump} in this case, it suffices to verify that
\begin{equation}\label{per_ver_j_5}
\begin{aligned}
    \begin{pmatrix}
    \tilde a ^{-1}& 0 \\ 0 & {\tilde a}
\end{pmatrix}\hat J_0 & \begin{pmatrix}
   a-b\tilde{\mathfrak R}^*&-b\eul^{2\ii k(L-\theta)}\\
 0&a^*-b^* \tilde{\mathfrak R}  
\end{pmatrix}\ee^{\ii k(L-\theta)\sigma_3}= \\
& \ee^{-\ii k \theta\sigma_3} \begin{pmatrix}
 a-b\tilde {\mathfrak R}^*   &0\\
 -b^*&a^*-b^*\tilde {\mathfrak R} 
\end{pmatrix}\begin{pmatrix}
    \tilde a^{-1} & 0 \\ 0 & {\tilde a}
\end{pmatrix} \hat J_0 \ee^{\ii k(L-\theta)\sigma_3} \ee^{\ii k \theta\sigma_3},
\end{aligned}
\end{equation}
where
\[
\hat J_0=\begin{pmatrix}
          1& -\tilde{G}_1(k)\\
         0&1   
\end{pmatrix}
\begin{pmatrix}
          1& 0\\
         \tilde{G}(k)&1   
\end{pmatrix}.
\]
Then, using \eqref{propR5}, \eqref{calR} and \eqref{s-tilde-s},
straightforward computation show that the both sides of \eqref{per_ver_j_5} equal 
\[
\begin{pmatrix}
		(1-{{\tilde{\mathfrak R}}^*}(k){{\tilde{\mathfrak R}}}(k))  & -{{\tilde{\mathfrak R}}}(k) \\
		 {{\tilde{\mathfrak R}}^*}(k) & 1
	\end{pmatrix}
    = 
    \begin{pmatrix}
		1 & -{\tilde{\mathfrak R}}(k) \\
        0 & 1 
        \end{pmatrix}
         \begin{pmatrix}
        1 & 0 \\
        {\tilde{\mathfrak R}}^*(k) &  1 
       	\end{pmatrix}.
\]

\end{proof}

Next, we investigate the behavior of $M{(t)}(t,k)$ and $\hat M{(t)}(t,k)$ at the special points.

\begin{proposition}\label{prop:beh_0}
 As $k\to 0$, $k\in\mathbb{C}_+$:
\begin{align}\label{beh_M_t_0_+}
 M^{(t)}(t,k)&=\frac{\ii}{k}\left(\hat c(0,t)\kappa-\rho \hat m^0_1(0,t)\right)\begin{pmatrix}
     0&1\\
     0&-1 \end{pmatrix}+\begin{pmatrix}
         \kappa \hat m_1^0(0,t)&*\\
         -\kappa \hat m_1^0(0,t)&*
     \end{pmatrix}+O(k)
\\\label{beh_hatM_t_0_+}
 \hat M^{(t)}(t,k)&=\frac{\ii }{k}\hat c(\theta,t)\begin{pmatrix}
     0&1\\
     0&-1 \end{pmatrix}+\begin{pmatrix}
         \hat m_1^0(\theta,t)& \hat m_2^0(\theta,t)\\
         - \hat m_1^0(\theta,t)& \hat m_4^0(\theta,t)
     \end{pmatrix}+O(k),
\end{align}
where 
\begin{equation}
    \label{kappa}
    \kappa:=a_0+b_0+\ii\rho \mathfrak R_1=\pm 1.
\end{equation}

 As $k\to 0$, $k\in\mathbb{C}_-$:
\begin{align}\label{beh_M_t_0_-}
 M^{(t)}(t,k)&=-\frac{\ii}{k}\left(\hat c(0,t)\kappa-\rho \hat m^0_1(0,t)\right)\begin{pmatrix}
     -1&0\\
     1&0 \end{pmatrix}+O(1)
\\\label{beh_hatM_t_0_-}
 \hat M^{(t)}(t,k)&=-\frac{\ii }{k}\hat c(\theta,t)\begin{pmatrix}
     -1&0\\
     1&0 \end{pmatrix}+O(1).
\end{align}

In particular, $M^{(t)}(t,k)$ and $\hat M^{(t)}(t,k)$ satisfy the same structural condition at $k=0$.

\end{proposition}
 
\begin{proof} First, notice that \eqref{hatM_at_0} implies \eqref{beh_hatM_t_0_+}, and \eqref{hatM_at_0_-} implies \eqref{beh_hatM_t_0_-}.

In order to obtain \eqref{beh_M_t_0_+} and \eqref{beh_M_t_0_-}, we begin by analyzing 
\[
\begin{pmatrix}
       a-b\mathfrak R^*&-b\eul^{-\frac{\ii k t}{\lambda}}\\
 0&a^*-b^*\mathfrak R  
\end{pmatrix}\ \ \text{as}\ k\to 0, \ 
\Im k > 0
\]
and 
\[
\begin{pmatrix}
   a-b\mathfrak R^*&0\\
 -b^*\eul^{\frac{\ii k t}{\lambda}}&a^*-b^*\mathfrak R   
\end{pmatrix}\ \ \text{as}\ k\to 0, \ 
\Im k < 0.
\]
In the generic case, where $\rho\neq 0$ in \eqref{ab-0}, \eqref{calR} and \eqref{propR10} imply that $\mathfrak R(k)=-1+k \mathfrak R_1+O(k^2)$ and $\mathfrak R^*(k)=-1-k \mathfrak R_1+O(k^2)$, and, thus,
\[ (a-b\mathfrak R^*)(k)=a_0+b_0+\ii\rho \mathfrak R_1+O(k),\quad (a^*-b^*\mathfrak R)(k)=a_0+b_0+\ii\rho \mathfrak R_1+O(k).\]
Moreover, \eqref{propR6} implies that 
$
  a_0+b_0+\ii\rho \mathfrak R_1=\pm 1  
$, and \eqref{kappa} follows.

Now, pushing the expansions \eqref{hatM_at_0} and \eqref{hatM_at_0_-} of $\hat M(k)$ a step further, and proceeding as in Section~\ref{sec:CH-y}  analogously to \eqref{hatm_i} we get $m_3^0=-m_1^0$, and straightforward computations yield \eqref{beh_M_t_0_+} and \eqref{beh_M_t_0_-}.

Finally, notice that we do not specify coefficients in the structural condition at $k=0$. 
\end{proof}


\begin{proposition}\label{prop:beh_infty}
    As $k\to \infty$, $M^{(t)}(t,k)=I+O(\frac{1}{k})$ and $\hat M^{(t)}(t,k)=I+O(\frac{1}{k})$. 
\end{proposition}
\begin{proof}
     Notice that \eqref{s_inf} together with \eqref{spec_R} implies that
$(a-b\mathfrak R^*)(k)\to 1$ as $k\to\infty$ in $\mathbb{C}_+$, and \ref{propR6} implies that $(a^*-b^*\mathfrak R)(\infty)\to 1$ as $k\to\infty$ in $\mathbb{C}_+$. Furthermore, \eqref{b_inf} implies that $b^*\eul^{2\ii k\theta}\to 0$ as $k\to \infty$. Thus, the statement follows. 
\end{proof}


\begin{proposition}\label{prop:beh_i_2}
    As $k\to \frac{\ii}{2}$:
    \begin{align}
        \label{beh_M_t_i_2}
             M^{(t)}(t,k)&=\frac{1}{2}\begin{pmatrix}
    ( \hat q(0,t)+\hat q^{-1}(0,t))\hat f(0,t)&(\hat q(0,t)-\hat q^{-1}(0,t))\hat f^{-1}(0,t)\\
   (\hat q(0,t)-\hat q^{-1}(0,t))\hat f(0,t)&(\hat q(0,t)+\hat q^{-1}(0,t))\hat f^{-1}(0,t)
\end{pmatrix}+
O\left(k-\tfrac{\ii}{2}\right),
        \\\label{beh_hatM_t_i_2}
  \hat M^{(t)}(t,k)&=\frac{1}{2}\begin{pmatrix}
    ( \hat q(\theta,t)+\hat q^{-1}(\theta,t))\hat f(\theta,t)\ee^{-\frac{L-\theta}{2}}&(\hat q(\theta,t)-\hat q^{-1}(\theta,t))\hat f^{-1}(\theta,t)\ee^{\frac{L-\theta}{2}}\\
   (\hat q(\theta,t)-\hat q^{-1}(\theta,t))\hat f(\theta,t)\ee^{-\frac{L-\theta}{2}}&(\hat q(\theta,t)+\hat q^{-1}(\theta,t))\hat f^{-1}(\theta,t)\ee^{\frac{L-\theta}{2}}
\end{pmatrix}+
O\left(k-\tfrac{\ii}{2}\right)        
    \end{align}
In particular, 
both $ M^{(t)}(t,k)$ and $ \hat M^{(t)}(t,k)$
satisfy the structural condition \eqref{hatM_at_i2_}.

\end{proposition}
 
\begin{proof}
Notice that 
$
 \tilde{\mathfrak R}\left(\frac{i}{2}\right) =
\tilde{\mathfrak R}^*\left(\frac{i}{2}\right) =0$, and $ a\left(\frac{i}{2}\right)=\ee^{\frac{L-\theta}{2}}
$. Furthermore, the off-diagonal entries of $P(t,k)$ and $\hat P(t,k)$ decay to $0$ exponentially fast as $k\to\frac{\ii}{2}$. Therefore, \eqref{beh_M_t_i_2} and \eqref{beh_hatM_t_i_2} follow (recall that 
the particular values of the coefficients in the 
structural condition at $k=\frac{\ii}{2}$ needn't to be equal).

\end{proof}

\medskip

Finally, let us check that the residue conditions match.

\begin{proposition}\label{prop:res}
(i) $M^{(t)(2)}(t,k)$ and $\hat M^{(t)(2)}(t,k)$ have simple poles at poles $\mu_j$ of $\mathfrak{R}$ with $\mu_j\in D_1\cup D_3$. Moreover,
\begin{align}\label{res_M_t_1}
\Res_{\mu_j} M^{(t)(2)}(t,k)& = M^{(t)(1)}(t,\mu_j)\ee^{-\ii\mu_j\frac{t}{\lambda(\mu_j)}}\Res_{\mu_j}\mathfrak R, \quad \mu_j\in D_1\cup D_3,
\end{align}
and $\hat M^{(t)(2)}(t,k)$ satisfies the same conditions 
\eqref{res_M_t_1}.

\noindent (ii)
$M^{(t)(1)}(t,k)$ and $\hat M^{(t)(1)}(t,k)$ have simple poles at poles $\bar\mu_j$ of $\mathfrak{R}^*$ with $\bar\mu_j\in D_2\cup D_4$. Moreover,
\begin{align}\label{res_M_t_3}
\Res_{\bar\mu_j} M^{(t)(1)}(t,k) = M^{(t)(2)}(t,\bar\mu_j)\ee^{\ii\bar\mu_j\frac{t}{\lambda(\bar\mu_j)}}\Res_{\bar\mu_j}\mathfrak R^*, \quad \mu_j\in D_2\cup D_4,
\end{align}
and $\hat M^{(t)(2)}(t,k)$ satisfies the same conditions 
\eqref{res_M_t_3}.

\noindent (iii)
There are no  other pole singularities.    
\end{proposition}

\begin{proof}

If  $\mu_j$ is such that  $\mathfrak R(\mu_j)=\infty$, then $b^*(\mu_j)=0$. Dividing \eqref{calR} by $R$ and passing to the limit as $k\to \mu_j$, we obtain
\begin{equation}\label{Rb*}
  \lim_{k\to \mu_j} \mathfrak R(k)b^*(k)=a^*(\mu_j)-a(\mu_j)\eul^{-2\ii \mu_j \theta}.  
\end{equation}
Consequently, 
\begin{equation}\label{a*-b*R}
(a^*-b^*\mathfrak R)(\mu_j)=a(\mu_j)\eul^{-2\ii\mu_j\theta}.  
\end{equation}
Now, combining \eqref{propR6}, \eqref{a*-b*R} and \eqref{det_rel}, we conclude that 

\begin{equation}\label{a-bR*}
(a-b\mathfrak R^*)(\mu_j)=a^*(\mu_j)\eul^{2\ii\mu_j\theta}.  
\end{equation}

Similarly, if $\bar \mu_j$ is such that $\mathfrak R^*(\bar \mu_j)=\infty$, then $b(\bar \mu_j)=0$. Proceeding as above,
we have
\begin{equation}\label{Rb*_2}
  \lim_{k\to \bar\mu_j} \mathfrak R^*(k)b(k)=a(\bar\mu_j)-a^*(\bar\mu_j)\eul^{2\ii \bar\mu_j \theta}.  
\end{equation}
Consequently, 
\begin{equation}\label{a-bR*_2}
(a-b\mathfrak R^*)(\bar\mu_j)=a^*(\bar\mu_j)\eul^{2\ii\bar\mu_j\theta}  
\end{equation}
and
\begin{equation}\label{a*-b*R_2}
(a^*-b^*\mathfrak R)(\bar\mu_j)=a(\bar\mu_j)\eul^{-2\ii\bar\mu_j\theta}.  
\end{equation}

Recall that for $k\in D_1$, we have
\begin{align}\label{res_M_hat_M_1}
    & M^{(t)(1)}(t,k)=\hat M^{(1)}(0,t,k) a(k), \\\label{res_M_hat_M_2}
    & M^{(t)(2)}(t,k)=\hat M^{(2)}(0,t,k) a^{-1}(k),\\\label{res_M_hat_M_3}
    & \hat M^{(t)(1)}(t,k)= (a^*-b^*\mathfrak R)(k)a(k)\hat M^{(1)}(\theta,t,k) +b^*(k)\ee^{2\ii k\theta}\ee^{\frac{\ii k t}{\lambda}}a^{-1}(k)\hat M^{(2)}(\theta,t,k),\\ \label{res_M_hat_M_4}
    & \hat M^{(t)(2)}(t,k)=(a-b\mathfrak R^*) a^{-1}(k)\hat M^{(2)}(\theta,t,k).
\end{align}
First, let us consider the first columns. It is immediate that $M^{(t)(1)}(t,k)$ is regular at $\mu_j$. Moreover, combining \eqref{Rb*}, \eqref{res_hatM-a} and \eqref{res_M_hat_M_3}, we obtain 
\begin{equation}\label{hatM_lim}
    \lim_{k\to \mu_j} \hat{M}^{(t)(1)}(t,k)=\hat{M}^{(1)}(\theta,t,\mu_j) a^*(\mu_j)
\end{equation}
and thus $\hat{M}^{(t)(1)}(t,k)$ is not singular at 
  $\mu_j$.

Considering the second columns, \eqref{res_hatM-a} and \eqref{res_M_hat_M_2} yield \eqref{res_M_t_1} while \eqref{res_hatM-a}, \eqref{a-bR*}, \eqref{hatM_lim} and \eqref{res_M_hat_M_4} give \eqref{res_M_t_1}
for $\hat M^{(t)}(t,k)$.

Finally, recall that  $a(k)$ has no zeros in $D_1$ and thus $a^{-1}(k)$ does not introduce any additional singularities there.

Therefore, both $M^{(t)}(t,k)$ and $\hat M^{(t)}(t,k)$ satisfy the same residue conditions in $D_1$.

Similarly, for $k\in D_2$, we have

\begin{align}\label{res_M_hat_M_1_}
    & M^{(t)(1)}(t,k)=\hat M^{(1)}(0,t,k) (a-b\mathfrak R^*)(k), \\\label{res_M_hat_M_2_}
    & M^{(t)(2)}(t,k)=-b(k)\ee^{-\frac{\ii k t}{\lambda}}
    \hat M^{(1)}(0,t,k)+(a^*-b^*\mathfrak R)(k)\hat M^{(2)}(0,t,k),\\\label{res_M_hat_M_3_}
    & \hat M^{(t)(1)}(t,k)= \hat M^{(1)}(\theta,t,k) ,\\ \label{res_M_hat_M_4_}
    & \hat M^{(t)(2)}(t,k)=\hat M^{(2)}(\theta,t,k)
\end{align}
and 
\begin{equation}\label{hatM_lim_2}
    \lim_{k\to \bar\mu_j} {M}^{(t)(2)}(t,k)=\hat{ M}^{(2)}(0,t,\bar\mu_j) a^*(\bar\mu_j)
\end{equation}
and thus $\hat{M}^{(t)(2)}(t,k)$ is not singular at 
  $\bar\mu_j$.

Proceeding as above, 
\eqref{res_hatM-b}, \eqref{a-bR*_2}, \eqref{hatM_lim_2} and \eqref{res_M_hat_M_1_} give
\eqref{res_M_t_3} whereas
 \eqref{res_hatM-b}, \eqref{a-bR*}, \eqref{hatM_lim_2} and \eqref{res_M_hat_M_3_} yield
\eqref{res_M_t_3}
for $\hat M^{(t)}(t,k)$.

By symmetry, the  residue conditions match in $D_3$ and $D_4$.

\end{proof}

Combining Propositions \ref{prop:jump}--\ref{prop:res} with Proposition \ref{prop:uniqness}, we conclude that $M^{(t)}(t,k)=\hat M^{(t)}(t,k)$ and thus
\begin{equation}
    \label{per_RH}
    \hat M(0,t,k)P(t,k)=\hat M(\theta,t,k)\hat P(t,k).
\end{equation}

Recall that $\hat u$, $\hat m$, and  $\hat v$
can be defined in terms of $\hat q$ and $\hat \gamma$
which, in turn, are defined from
the expansions of $\hat M(y,t,k)$ as $k\to\frac{\ii}{2}$ and $k\to 0$, see \eqref{q-M} and \eqref{gamma}.
Then Proposition \ref{prop:beh_i_2} combined with \eqref{per_RH} implies that $\hat q(0,t)=\hat q(\theta,t)$. 

On the other hand,
Proposition \ref{prop:beh_0} combined  with \eqref{per_RH} yields
\begin{equation*}
\hat c(0,t)\kappa-\rho \hat m^0_1(0,t)=\hat c(\theta,t), \quad \kappa \hat m^0_1(0,t) = \hat m^0_1(\theta,t) 
\end{equation*}
and thus $\hat \gamma(0,t)=\hat \gamma(\theta,t)$.
Consequently,
\begin{equation}
    \hat u(0,t)=\hat u(\theta,t), \quad  \hat v(0,t)=\hat v(\theta,t), \quad  \hat m(0,t)=\hat m(\theta,t).
\end{equation}

\section{Concluding remark}

The global well-posedness results (see \cites{Con97, CE98_2,CE98}) ensure the existence of a unique solution $u(x,t)$ of \eqref{CH-m1}--\eqref{ic} for all $t\in\mathbb{R}$. 

On the other hand, combining the results of Section \ref{sec:CH-y} and Section \ref{sec:RHpbl}, we have
\begin{proposition}
    If the solution of the Riemann--Hilbert problem \eqref{hatM-jump}--\eqref{hatM_at_i2_} exists, then the function $u(x,t)$, expressed in the parametric form by \eqref{hmbbgg} together with \eqref{x}, is a solution of the periodic initial-boundary value problem for the Camassa--Holm equation \eqref{CH-m1}--\eqref{ic}.
\end{proposition}
    
Furthermore, for $t=0$, the solution of the Riemann--Hilbert problem \eqref{hatM-jump}--\eqref{hatM_at_i2_} exists by construction (see Remark \ref{rem:const_t_0}). Then, the analytic Fredholm alternative (see \cite{Z89}) implies that any loss of solvability of the RH problem can occur only at a discrete set of points. Consequently,  such a breakdown  contradicts the existence of the global solution of $u(x,t)$.

\section{Acknowledgment}

IK acknowledges the support from 
 the Austrian Science Fund (FWF), grant no.
10.55776/ESP691. 
DS acknowledges the support from the Virtual Ukraine
Institute for Advanced Study (VUIAS), Fellow 2025/2026.

\bibliographystyle{RS}
\bibliography{shepelsky_etal}

\end{document}